\documentclass[11pt]{article}
\usepackage{amsfonts}
\usepackage{CJK}
\usepackage{graphics}
\usepackage{indentfirst}
\usepackage{cite}
\usepackage{latexsym}
\usepackage{amsmath}
\usepackage{amsthm}
\usepackage{amssymb}
\usepackage[dvips]{epsfig}
\usepackage{amscd}
\usepackage[marginal]{footmisc}
\usepackage[bookmarks,colorlinks,citecolor=blue,linkcolor=red]{hyperref}
\clearpage
\def\on{\bar\rho}
\usepackage{mathtools}
\allowdisplaybreaks
\newtheorem{theorem}{Theorem}[section]
\newtheorem{remark}{Remark}[section]
\newtheorem{definition}{Definition}[section]
\newtheorem{lemma}[theorem]{Lemma}
\newtheorem{proposition}{Proposition}[section]

\newcommand{\ltwo}{_{L^2}^2}

\newcommand{\lm}{\lambda}
\renewcommand{\div}{ {\rm div }  }

\newcommand{\na}{\nabla }

\newcommand{\pa}{\partial}
\renewcommand{\r}{\mathbb{R}}

\newcommand{\dis}{\displaystyle}
\newcommand{\ia}{\int_0^T}

\newcommand{\bt}{{\hat\theta}}
\newcommand{\bl}{\begin{lemma}}
\newcommand{\el}{\end{lemma}}
\newcommand{\et}{\end{theorem}}
\newcommand{\ga}{\gamma}

\newcommand{\curl}{{\rm curl} }
\newcommand{\te}{\theta}

\newcommand{\al}{\alpha}
\newcommand{\de}{\delta}
\newcommand{\ve}{\varepsilon}
\newcommand{\la}{\label}

\newcommand{\p}{p(\rho)  }

\newcommand{\ka}{\kappa}

\newcommand{\ol}{\overline}

\newcommand{\bn}{\begin{eqnarray}}
\newcommand{\en}{\end{eqnarray}}
\newcommand{\bnn}{\begin{eqnarray*}}
\newcommand{\enn}{\end{eqnarray*}}

\newcommand{\bnnn}{\begin{eqnarray*}}
\newcommand{\ennn}{\end{eqnarray*}}

\newcommand{\ba}{\begin{aligned}}
\newcommand{\ea}{\end{aligned}}
\newcommand{\be}{\begin{equation}}
\newcommand{\ee}{\end{equation}}
\def\O{\Omega}
\def\p{\partial}
\def\norm[#1]#2{\|#2\|_{#1}}

\newcommand{\ep}{\varepsilon}
\newcommand{\n}{\rho}
\newcommand{\si}{\sigma}

\def\la{\label}

\def\na{\nabla}
\def\on{\hat\rho}

\def\xl{\left}
\def\xr{\right}
\def\bp{\overline{P}}

\renewcommand{\thefootnote}{}
\newcommand\blfootnote[1]{%
  \begingroup
  \renewcommand\thefootnote{}\footnote{#1}%
  \addtocounter{footnote}{-1}%
  \endgroup
}

\def\QEDopen{{\setlength{\fboxsep}{0pt}\setlength{\fboxrule}{0.2pt}\fbox{\rule[0pt]{0pt}{1.3ex}\rule[0pt]{1.3ex}{0pt}}}} 

\def\QED{\QEDopen} 

\def\endproof{\hspace*{\fill}~\QED\par\endtrivlist\unskip}

\makeatletter      
\@addtoreset{equation}{section}
\makeatother       

\title{Global Well-Posedness of     Full Compressible Magnetohydrodynamic System in 3D Bounded Domains with  Large Oscillations and Vacuum}

\author{Yazhou C{\small HEN}$^a$, Yunkun C{\small HEN}$^b$, Xue W{\small ANG}$^c$   \\[3mm]
{\normalsize  a. College of Mathematics and Physics, }\\ {\normalsize  Beijing University of Chemical Technology, Beijing 100029, P. R. China;} \\
{\normalsize  b. Anshun University, Anshun, 561000, P.R. China; } \\ {\normalsize c.  School of Mathematical Sciences,}\\
{\normalsize  University of Chinese Academy of Sciences, Beijing 100049, P. R. China}}

\date{ }

\begin{document}
\maketitle

\blfootnote{Email: chenyz@mail.buct.edu.cn (Y.Chen), cyk2013@gznu.edu.cn (Y. Chen), xuewa@amss.ac.cn (X.  Wang)}

\begin{abstract}
The three-dimensional (3D) full compressible magnetohydrodynamic system is studied in a general bounded domain with slip boundary condition for the velocity filed,  adiabatic condition for the temperature and perfect conduction for the magnetic field.  For the regular initial data with small energy but possibly large oscillations,  the global existence of classical and weak solution as well as the exponential decay rate to the initial-boundary-value problem of this system is obtained. In particular, the density and temperature of such a classical solution are both allowed to vanish initially. Moreover, it is also shown that for the classical solutions, the oscillation of the density will grow unboundedly with an exponential rate when the initial vacuum appears (even at a point). Some new observations and useful estimates are developed to overcome the difficulties caused by the slip boundary conditions.
\end{abstract}



\section{Introduction}
We consider the motion of a viscous, compressible, and heat conducting magnetohydrodynamic (MHD) flow in a domain $\Omega\subset\r^{3}$, which can be described by the full compressible MHD equations (see \cite{hw2008}):
\begin{equation}\label{CMHD-1}
\begin{cases}
\rho_t+ \mathop{\mathrm{div}}\nolimits(\rho u)=0,\\
(\rho u)_t+\mathop{\mathrm{div}}\nolimits(\rho u\otimes u)+\nabla P
=\mathop{\mathrm{div}}\mathbb{S}+(\nabla\times H)\times H,\\
(\mathcal{E})_t+\mathop{\mathrm{div}}(\rho Eu+Pu)=\mathop{\mathrm{div}}\big((u\!\times H)\!\times H+H\!\times(\nu\nabla\!\times H)+u\mathbb{S}+\kappa\nabla \theta\big),\\
H_t -\nabla \times (u \times H)=-\nu \nabla \times (\nabla \times H),
\\
\mathop{\mathrm{div}}\nolimits H=0,
\end{cases}
\end{equation}
where $(x,t)\in\Omega\times (0,T]$, $t\geq 0$ is time, and $x=(x_1,x_2,x_3)$ is the spatial coordinate. $\rho=\rho(x,t)$ denotes the density, $u=(u^1,u^2,u^3)^{tr}(x,t)$ the velocity, $\theta=\theta(x,t)$ the temperature, $H=(H^1,H^2,H^3)^{tr}(x,t)$ the magnetic field;  $\mathbb{S}$ is the viscous stress tensor given by
\begin{equation}
  \mathbb{S}=2\mu \mathbb{D}u+\lambda \mathop{\mathrm{div}}u \mathbb{I}_3,
\end{equation}
where $\mathbb{D}u=(\nabla u+(\nabla u)^{tr})/2$ is the deformation tensor, $\mu$ and $\lambda$ are the shear viscosity and bulk coefficients respectively satisfying the following physical restrictions $\mu>0$ and $2\mu +{3}\lambda\geq 0$, and $\mathbb{I}_3$ is the $3\times3$ identity matrix; $\mathcal{E}$ is the total energy given by
\begin{equation}\label{ee}
  \mathcal{E}=\rho E+\frac{|H|^2}{2}\quad  \text{and}\quad E=e+\frac{|u|^2}{2},
\end{equation}
with $e$ the internal energy; $P=P(\rho,\theta)$ is the pressure, the constant $\kappa>0$ is the heat conductivity and $\nu >0$ is the magnetic resistivity coefficient. In the current paper, we consider the ideal  polytropic fluids so that $P$ and $e$ are given by the state equations:
\begin{equation}
  P=R\rho\theta, \quad e=\frac{R\theta}{\gamma-1},
\end{equation}
where $\gamma>1$  is the adiabatic constant  and $R$ is a positive constant.

Let $\Omega \subset \mathbb{R}^3 $ be a simply connected bounded domain.
Note that for the classical solutions, the system \eqref{CMHD-1} can be rewritten as
\begin{equation}\label{CMHD}
\begin{cases}
\rho_t+ \mathop{\mathrm{div}}\nolimits(\rho u)=0,\\
\rho (u_t+ u \cdot \nabla u)+ \nabla P=\mu\Delta u+(\mu+\lambda)\nabla\mathop{\mathrm{div}}u+H \!\cdot\! \nabla H-\nabla \frac{|H|^2}{2},\\
\frac{R}{\gamma-1}\rho ( \theta_t+u\cdot\nabla \theta)+P\mathop{\mathrm{div}} u=\kappa\Delta\theta+\lambda (\mathop{\mathrm{div}} u)^2+2\mu |\mathbb{D}(u)|^2+\nu|\mathop{\rm curl} H|^2, \\
H_t+u \cdot \nabla H-H \cdot \nabla u+ H \mathop{\mathrm{div}} u= -\nu \nabla \times \mathop{\rm curl} H,
\\
\mathop{\mathrm{div}}\nolimits H=0,
\end{cases}
\end{equation}
In addition, the system is solved subject to the given initial data
\begin{equation}\label{initial}
\displaystyle  \big(\rho,\rho u,\rho \theta,H\big)(x,0)=\big(\rho_0,\rho_0 u_0,\rho_0 \theta_0,H_0\big)(x),\quad x\in \Omega,
\end{equation}
and slip boundary conditions
\begin{align}
& u\cdot n=0,\,\,\,\curl u\times n=0, &\text{on} \,\,\,\partial\Omega, \label{navier-b}\\
& \nabla \theta \cdot n=0, &\text{on} \,\,\,\partial\Omega, \label{theta-b}\\
& H \cdot n=0,\,\,\,\curl H\times n=0,  &\text{on} \,\,\,\partial\Omega,\label{boundary}
\end{align}
where $n=(n^1,n^2,n^3)^{tr}$ is the unit outward normal vector to $\partial \Omega$.
It should be mentioned that the boundary condition \eqref{navier-b} for the velocity  can be regarded as a special case of the following general Navier-type slip boundary condition (see e.g., \cite{cl2019})
 \begin{align}\label{Navi2} \displaystyle  u \cdot n = 0, \,\,(2\mathbb{D}u\  n)_{\tau}+ \vartheta u_{\tau}=0 \,\,\,\text{on}\,\,\, \partial\Omega, \end{align}
which is proposed by Navier in \cite{Nclm1} and indicates that there is a stagnant layer of fluid close to the wall allowing a fluid to slip and the slip velocity is proportional to the shear stress. Here $\vartheta$ is a scalar friction function, subscript $\tau$ denotes the tangential component on $\partial\Omega$.
Additional, as shown in \cite{cl2019}, the Navier-slip condition \eqref{Navi2} is in fact a particular case of the following slip boundary condition:
 \begin{align}\label{navi1} u\cdot n=0,\,\,\,\curl u\times n=-(Au)_{\tau} \,\,\,&\text{on} \,\,\,\partial\Omega, \end{align}
where $A$ is a smooth symmetric matrix defined on $\partial \Omega$.
For the temperature field, we adopt the adiabatic condition \eqref{theta-b}.
For the magnetic field, the boundary condition \eqref{boundary} describes that the boundary $\partial \Omega$ is a perfect conductor (see e.g. \cite{djj2013}), that means the magnetic field is confined inside and separated from the exterior.
Therefore, it is appropriate to consider the compressible MHD equations with the boundary conditions \eqref{navier-b}-\eqref{boundary}.

Significant progress has been made in the analysis of the well-posedness and dynamic behavior to the solutions of the full compressible MHD system \eqref{CMHD} due to its physical importance and mathematical challenges, including the strong coupling and interplay interaction between fluid motion and magnetic field.
For one-dimensional case, We refer to \cite{cw2002,ko1982,Wang2003,hss2021} and their references.
For the multi-dimensional case, the local existence of strong solutions with initial vacuum was proved by Fan and Yu \cite{fy2009} in three-dimensional space and by Li and Huang \cite{lh2016} in two-dimensional space.
The local existence of the classical solutions to the 3D initial-boundary value problem with slip boundary condition was investigated by Xi and Hao \cite{xh2017}.
We refer to \cite{tg2016,xh2017,lh2015} for local strong solutions to the isentropic compressible MHD system.
Concerning with global weak solutions, the global existence of renormalized weak solutions in a 3D bounded domain was proved by Hu and Wang \cite{hw2008,hw2010} with large data and also investigated in \cite{sh2012,lyz2013,df2006} for isentropic case or non-isentropic case.
The global existence of smooth solutions to the Cauchy problem of compressible MHD system was first obtained by Kawashima \cite{k1984} when the initial data are close to a non-vacuum equilibrium in $H^3$-norm. The result was extended to full compressible MHD system in 3D exterior domains by Liu et al. \cite{llz2021} with  Navier-slip for the velocity filed and perfect conduction for the magnetic field.
For Cauchy problem of isentropic compressible MHD system with vacuum, the global existence and uniqueness of classical solutions with small initial energy but large oscillations in three-dimensional space was proved by Li et al. \cite{lxz2013} and the result was generalized by Hong et al. \cite{hhpz2017} for large initial data with $\gamma-1$ and $\nu^{-1}$ are suitably small. Later on, global existence of unqiue classical solutions in two-dimensional space and some better a priori decay with rates are obtained by Lv et al. \cite{lsx2016}. For the initial-boundary-value problem, the global classical solutions to isentropic compressible MHD system with slip boundary conditions in 3D bounded domains was studied by Chen et al. \cite{chs2020-mhd} for the regular initial data with small energy but possibly large oscillations and vacuum. More resently, the global well-posedness of strong and weak solutions of the isentropic compressible MHD system in 2D bounded domains with large initial data and vacuum was investigated by Chen et al. \cite{chs2021-mhd}.
As for the stong/classical solutions to the full compressible MHD system,
the global strong solutions to the 3D Cauchy problem was investigated by Liu and Zhong \cite{lz2020-mhd,lz2021-fmhd,lz2022-fmhd} and Hou et al. \cite{hjp2022} under some certain small conditions.
We refer the readers to \cite{fl2020,ls2019,tw2018,ww2017,zhu2015,ls2021} for more results on global solutions to multi-dimensional compressible non-resistive or inviscid MHD equations.


The main purpose of this paper is to establish the global existence and large time behavior of the classical solutions for full compressible MHD system \eqref{CMHD} in general bounded domains with the boundary condition \eqref{navier-b}-\eqref{boundary} in $\r^3$, which can be regarded as a continuation of our work in \cite{chs2020-mhd} for isentropic case.
It should be pointed out that it seems much more difficult and complicated to
study the global well-posedness of solutions to full compressible MHD system with vacuum 
since the structure of the constitutive quation has changed and some additional difficulties arise, such as the degeneracy of both momentum and energy equations, the strong coupling between the velocity and temperature, and more strong nonlinear terms need to be controlled.
Very resently,  Li et al. \cite{llw2022} established the global well-posedness of classical solutions to full compressible Navier-Stokes system in general 3D bounded domains with small energy but possibly large oscillations and  both the density and temperature are allowed to vanish initially.
Motivated by \cite{cl2019,llw2022}, we would like to obtain the time-independent upper bound of the density and the time-dependent higher-norm estimates of $(\rho,u,\theta,H)$ and extend the classical solution globally in time via some new estimates on velocity, temperature and magnetic field, which are important in controlling the strong coupling effects and boundary terms.

Before formulating our main result, we first explain the notation and conventions used throughout the paper.
For integer $k\geq 1$ and $1\leq q<+\infty$, We denote the standard Sobolev space by $W^{k,q}(\Omega)$ and $H^k(\Omega)\triangleq W^{k,2}(\Omega)$.
For simplicity, we denote $L^q(\Omega)$, $W^{k,q}(\Omega)$ and $H^k(\Omega)$  by $L^q$, $W^{k,q}$ and ${H^k}$ respectively.
For two $3\times 3$  matrices $A=\{a_{ij}\},\,\,B=\{b_{ij}\}$, the symbol $A\colon  B$ represents the trace of $AB^*$, where $B^*$ is the transpose of $B$, that is,
$$ A\colon  B\triangleq \text{tr} (AB^*)=\sum\limits_{i,j=1}^{3}a_{ij}b_{ij}.$$
Finally, for $v=(v^1,v^2,v^3)$, we denote $\nabla_iv\triangleq(\partial_iv^1,\partial_iv^2,\partial_iv^3)$ for $i=1,2,3,$ and the
material derivative of $v$ by  $\dot v\triangleq v_t+u\cdot\nabla v$.
We denote
$$\int fdx \triangleq \int_\Omega fdx,\quad \int_0^T\int fdxdt\triangleq\int_0^T\int_\Omega fdxdt,$$
and
$$\bar{f} \triangleq \frac{1}{|\Omega|}\int_\Omega fdx,$$
which is the average of a function $f$ over $\Omega$.
Without loss of generality, we assume that
\begin{equation}\label{rho0-b}
\overline{\rho_0}=\frac{1}{|\O|}\int \rho_0 dx=1.
\end{equation}

The initial total energy of \eqref{CMHD} is defined as
\begin{align}\label{c0}
\displaystyle  C_0 =\int_{\Omega}\left(\frac{1}{2}\rho_0|u_0|^2 + R G(\rho_0)+\frac{R}{\gamma-1}\rho_0\Phi(\theta_0)+\frac{1}{2}|H_0|^2 \right)dx.
\end{align}
where
\begin{align}\label{def-g}
\displaystyle  G(\rho)\triangleq 1+\rho\ln\rho-\rho,\quad \Phi(\theta)\triangleq\theta-\ln\theta-1.
\end{align}

Now we can state our first main result, Theorem \ref{th1}, concerning existence of global classical solutions to the problem  \eqref{CMHD}-\eqref{boundary}.
\begin{theorem}\label{th1}
Let $\Omega$ be a simply connected bounded domain in $\r^3$ and its smooth boundary $\partial\Omega$ has a finite number of 2-dimensional connected components. For $q\in (3,6)$ and some given constants $M>0$, $\hat{\rho}>2, \hat{\theta}>1$, and the initial data $(\rho_0,u_0,\theta_0,H_0)$ satisfies the boundary conditions \eqref{navier-b}-\eqref{boundary} and
\begin{align}
\displaystyle & 0\le\inf\rho_0\le\sup\rho_0< \hat{\rho},\quad 0  \le\inf\theta_0\le\sup\theta_0\le \hat{\theta},\label{dt1}\\
&\rho_0\in W^{2,q},\quad u_0\in  H^2,\quad \theta_0 \in  H^1,\quad H_0 \in  H^2,\quad \div H_0=0,\label{dt2}\\
& \|\nabla u_0\|_{L^2}\leq M,\quad \|\nabla H_0\|_{L^2}\leq M,\label{dt-s}
\end{align}
and the compatibility condition
\begin{align}\label{dt3}
\displaystyle  -\mu\triangle u_0-(\mu+\lambda)\nabla \mathop{\mathrm{div}}\nolimits u_0 + R\nabla (\rho_0\theta_0)- (\nabla \times H_0) \times H_0 = \rho_0^{\frac12}g,
\end{align}
for some  $ g\in L^2.$
Then there exists a positive constant $\ve$ depending only on  $\mu$, $\lambda$, $\nu$, $\ga$, $R$, $\kappa$, $\hat{\rho}$, $\hat{\theta}$, $\Omega$, and $M$  such that if
\be  \la{co14} C_0\le\ve, \ee
the system \eqref{CMHD}-\eqref{boundary} has a unique global classical solution $(\rho,u,\theta,H)$ in $\Omega\times(0,\infty)$ satisfying for any $0<\tau<T<\infty$,
\begin{align}\label{esti-rho}
\displaystyle  0\le \rho(x,t)\le 2\hat{\rho},\quad \theta(x,t)\geq 0,\quad  (x,t)\in \Omega\times(0,T),
\end{align}
\begin{equation}\label{esti-uh}
\begin{cases}
\rho\in C([0,T);W^{2,q} ),\\
u\in C([0,T);W^{1,\widetilde{p}} )\cap  L^\infty(0,T;H^2)\cap  L^\infty(\tau,T;W^{3,q}),\\
u_t\in L^{2}(0,T; H^1)\cap L^{\infty}(\tau,T; H^2)\cap H^1(\tau,T; H^1),\\
\theta\in C([\tau,T);W^{3,\widetilde{p}} )\cap  L^\infty(\tau,T;H^4),\\
\theta_t\in L^{\infty}(\tau,T; H^2)\cap H^1(\tau,T; H^1),\\
H \in C([0,T);H^2)\cap  L^\infty(\tau,T; H^4),\\
H_t\in L^{2}(0,T; H^1)\cap L^\infty(\tau,T; H^2)\cap H^1(\tau,T; H^1),	
\end{cases}
\end{equation}
for $\widetilde{p}\in [1,6)$.
Moreover,  for any $p\in [1,\infty)$ and $r\in [1,6],$ there exist positive constants $C$ and $\eta_0$ depending only  on $\mu,$  $\lambda,$ $\nu,$  $\gamma,$ $R$, $\kappa$, $\hat{\rho}$, $\hat{\theta}$, $\Omega$,   $M,$  $r$ and $p$  such that for $t>0,$
\begin{align}\label{esti-t}
\displaystyle  \|\rho(\cdot,t)-1\|_{L^p}+\|u(\cdot,t)\|_{W^{1,r}} +\|(\theta-\theta_\infty)(\cdot,t)\|^2_{H^2}+\|H(\cdot,t)\|_{H^2}\leq Ce^{-\eta_0 t}.
\end{align}
\end{theorem}

Next, thanks to the exponential decay rate \eqref{esti-t}, by the similar procedure as that in \cite{cl2019,lx2006}, we obtain the following large-time behavior of the gradient of the density when the initial density contains vacuum state.
\begin{theorem}\label{th2}
Under the conditions of Theorem \ref{th1}, assume further  that there exists some point $x_0\in \Omega$ such that $\rho_0(x_0)=0.$  Then the unique global classical solution $(\rho,u,\theta,H)$ to the problem \eqref{CMHD}-\eqref{boundary} obtained in
Theorem \ref{th1}  satisfies that for any $\tilde{r}>3,$   there exist positive constants $\tilde{C}_1$ and $\tilde{C}_2$ depending only  on $\mu$,  $\lambda$, $\nu$, $\gamma$, $R$, $\kappa$, $\hat{\rho}$, $\hat{\theta}$,  $\Omega$, $M,$ and  $\tilde{r}$   such that for any $t>0$,
\begin{align}\label{esti-2}
\displaystyle \|\nabla\rho (\cdot,t)\|_{L^{\tilde{r}}}\geq \tilde{C}_1 e^{\tilde{C}_2 t} .
\end{align}
\end{theorem}

Finally, we state our result concerning the global existence of weak solution to \eqref{CMHD-1}, whose definition is as follows.
\begin{definition}\la{def} We say that $(\n,u,\mathcal{E},H)$ is a weak solution to Cauchy problem (\ref{CMHD-1}), (\ref{navier-b})-(\ref{boundary}) provided that
$$\n\in L^\infty_{\rm loc}([0,\infty);L^\infty(\O)),\quad (u ,\te, H)\in L^2_{\rm loc} ([0,\infty); H^1(\O)),$$
and that for all test functions $\psi\in\mathcal{D}(\O\times(-\infty,\infty)),$
\be \la{def1}
\int_{\O}\n_0\psi(\cdot,0)dx+\int_0^\infty\int_{\O}\left(\n\psi_t+\n u\cdot\na\psi\right) dxdt=0,\ee
\be\la{def2}\ba&
\int_{\O}\n_0u^j_0\psi(\cdot,0)dx+\int_0^\infty\int_{\O}\left(\n u^j\psi_t
   +\n u^ju\cdot\na\psi+P(\n,\te)\psi_{x_j}\right)dxdt\\
&-\!\!\int_0^\infty\!\!\!\int_{\O}\!\!\left(  (\mu\na u^j\!\!+\!H^jH)\cdot\!\na\!\psi\!\!+\!((\mu\!\!+\!\!\lambda)\div u\!\!-\!\!\frac{|H|^2}{2})\psi_{x_j}\right) dxdt\!\!=\!0,\quad
j\!\!=\!1,2,3, \ea\ee
\be \la{def3}\ba
&\int_{\O}\left(\frac{1}{2}\n_0|u_0|^2 +\frac{R}{\ga-1}\n_0\te_0+\frac{1}{2}|H_0|^2 \right)\psi(\cdot,0)dx\\
&+\int_0^\infty\int_{\O}\left(\mathcal{E}\psi_t+ (\n E  +P)u\cdot\na\psi\right) dxdt\\
&-\int_0^\infty\int_{\O}\left(\ka\na\te+\frac{1}{2}\mu \na(|u|^2)
 +\mu  u\cdot \na u+\lambda\div u u\right)\cdot\na \psi dxdt\\
 &-\int_0^\infty\int_{\O}\left(|H|^2u-(H\cdot u)H+\nu H\times(\nabla\times H))\right)\cdot\na \psi dxdt=0.
\ea\ee
\be\la{def4}\ba&
\int_{\O}\!\!H^j_0\psi(\cdot,0)dx\!+\!\!\!\int_0^\infty\!\!\!\int_{\O}\!\!\left(H^j\psi_t
   \!+\!(H^ju\!-\!u^jH\!-\!\nu\nabla H^j)\cdot\na\psi\right)dxdt\!\!=\!0,\
j\!\!=\!1,2,3, \ea\ee
where $\mathcal{E}$ and $E$ are defined in \eqref{ee}.
\end{definition}
Then denoting the effective viscous flux $F$ and the vorticity $\omega$ as follows
\begin{align}\label{flux}
F\triangleq(\lambda+2\mu)\text{div}u-(P-\bar{P})-\frac{1}{2}(|H|^2-\overline{|H|^2}),\quad \omega\triangleq\nabla\times u,
\end{align}
we state our third main result as follows:
\begin{theorem}\la{th3}  Under the conditions of Theorem \ref{th1} except (\ref{dt3}), where the condition (\ref{dt2}) is replaced by
\be\la{dt7}(u_0, H_0)\in H^1,\ee
assume  further that $C_0$  as in (\ref{c0}) satisfies (\ref{co14}) with $\varepsilon$ as in Theorem \ref{th1}. Then there exists a global weak solution $(\n,u,\mathcal{E},H)$ to the problem  (\ref{CMHD-1}), (\ref{navier-b})-(\ref{boundary}) satisfying
\be\la{hq1}
\n\in C([0,\infty);L^p), \quad (\n u,\,\n |u|^2,\,\n\te)\in C([0,\infty);H^{-1}) ,
\ee
\be\la{hq2}u\in L^\infty(0,\infty;H^1)\cap C((0,\infty);L^2 )  ,\quad (\te,H)\in C((0,\infty); W^{1,\tilde p}),\ee
\be \la{hq3}
F(\cdot,t),  \,\,\omega(\cdot,t),\,\,u(\cdot,t),\,\,\na\te(\cdot,t),\,\,\na H(\cdot,t)\in H^1,\quad t>0,\ee
\be\la{hq4}\rho\in [0,2\hat{\rho}] \quad \mbox{ \rm a.e.}, \quad \te\ge 0 \quad\mbox{ \rm a.e.},\ee
 and the exponential decay property (\ref{esti-t})  with $p\in  [1 ,\infty)$, $\tilde p\in [1,6)$, and $r\in[1,6]$. In addition, there exists some  positive constant $C$ depending  only on $ \mu,$ $\lambda,$ $ \ka,$ $ R,$ $ \ga,$  $\on$, $\bt,$ $\O$, and $M $ such that, for  $\si(t)\triangleq\min\{1,t\},$  the following estimates hold
  \be\la{hq5}\ba  \sup_{t\in (0,\infty)} \| u \|_{H^1}^2  +\int_0^\infty\int\left| (\n u)_t+{\rm div}(\n u\otimes u)\right|^2dxdt\le C,\ea\ee
  \be\la{hq7}\ba &\sup_{t\in (0,\infty)} \int\left((\n-1)^2+\n |u|^2+\n(R\te-\overline P)^2+|H|^2\right)dx  \\&\quad   +\int_0^\infty\left( \|\na u\|^2_{L^2}+ \|\na \te\|_{L^2}^2+ \|\na H\|_{L^2}^2 \right)dt\le CC_0^{1/4}, \ea\ee
\be\la{hq8}\ba
&\sup_{t\in (0,\infty)}\left(\si\|\na u \|^2_{L^6}+\si\|H\|^2_{H^2}+\si^2\|\te\|^2_{H^2}\right) \\
&+\int_0^\infty\left( \si\|u_t\|^2_{L^2}+\si\|H_t\|^2_{H^1}+\si^2\|\na \dot u\|_{L^2}^2+\si^2\|\te_t\|^2_{H^1}\right)dt\le C. \ea\ee
Moreover, $(\n,u,\te,H)$ satisfies (\ref{CMHD-1})$_3$ in the weak form, that is, for any test function  $\psi\in
\mathcal{D}(\O\times(-\infty,\infty)),$
\be\la{vu019}\ba &\frac{R}{\ga-1} \int\n_0 \te_0\psi(\cdot,0) dx+\frac{R}{\ga-1}\int_0^\infty\int\n \te \left(\psi_t+u \cdot\na\psi
\right)dxdt\\ & =  \ka \int_0^\infty\int \na\te \cdot\na\psi dxdt+R\int_0^\infty\int  \n\te  \div u \psi dxdt\\&\quad -\int_0^\infty\int \left(\lambda(\div u)^2+2\mu   |\mathbb{D}u|^2 +\nu   |\curl H|^2 \right) \psi dxdt.\ea\ee
\end{theorem}

A few remarks are in order:
\begin{remark}\label{rem:1} From Sobolev's inequality and \eqref{esti-uh}$_1$ with $q>3$, it follows that
\begin{equation}\label{rem-1}
\displaystyle \rho, \nabla \rho \in C(\bar\Omega\times [0,T]).
\end{equation}
Moreover, it also follows from \eqref{esti-uh}$_{2-7}$  that
\begin{equation}\label{rem-2}
\displaystyle  u, \theta, H, \nabla u, \nabla\theta, \nabla H, \nabla^2 u, \nabla\theta, \nabla^2 H, u_t, \theta_t, H_t \in C(\bar\Omega\times [\tau,T]),
\end{equation}
due to the following simple fact that $$L^2(\tau,T;H^1)\cap H^1(\tau,T;H^{-1})\hookrightarrow C([\tau,T];L^2).$$
Finally, by \eqref{rem-2} \eqref{CMHD}$_1,$ we have
\begin{equation}\notag
\displaystyle \rho_t=-u\cdot \nabla \rho-\rho\div u\in C(\bar\Omega\times [\tau,T]).
\end{equation}
which together with \eqref{rem-1}-\eqref{rem-2} implies that the solution obtained in Theorem \ref{th1} is a classical one away from the initial time.
\end{remark}

\begin{remark}\label{rem:2}
When we consider the general slip boundary \eqref{navi1} for the velocity field, and assume that the $3\times 3$ symmetric matrix $A$ is smooth and positive semi-definite, and even if the restriction on $A$ is relaxed to $A\in W^{2,6}$ and the negative eigenvalues of $A$ (if exist) are small enough, in particular, set $A = B-2D(n)$, where $B\in W^{2,6}$ is a positive semi-definite $3\times 3$ symmetric matrix, Theorem \ref{th1} and \ref{th2} will still hold provided that $2\mu+3\lambda>0$. This can be achieved by a similar way as in \cite{cl2019}.
\end{remark}



\begin{remark}\label{rem:4}
In Theorem \ref{th1}, we get the global existence and uniqueness of classical solutions with vacuum with the compatibility condition on the velocity (\ref{dt3}) as in \cite{chs2020-mhd}, which is much weaker than those in \cite{lz2021-fmhd,lz2022-fmhd} where not only (\ref{dt3}) but also the following compatibility condition on the temperature
  \be\la{co1} \ka\Delta \te_0+2\mu|\mathbb{D}u_0|^2+\lambda (\div u_0)^2+\nu|\curl H_0|^2=\n_0^{\frac12}\tilde{g}, \quad \tilde{g} \in L^2\ee
  is needed. This reveals that the compatibility condition on the temperature (\ref{co1}) is not necessary for establishing the classical solutions with vacuum to the full MHD system, which is similar as the barotropic Navier-Stokes equation \cite{HLX2012, cl2019}.
\end{remark}

We now sketch the main idea used in the proof of Theorem \ref{th1}. Similar to the argument in \cite{llw2022,chs2020-mhd}, the key issue in our proof is to derive the time-independent upper bound of the density in Proposition \ref{pr1}, which can make the standard local classical solutions (see Lemma \ref{lem-local}) extend globally in time. It is worth pointing out that the effective viscous flux $F$ and the vorticity $\omega$ (see \eqref{flux} for the definition) play an important role in the proof.
However, the basic energy estimate can not be derived directly from the system \eqref{CMHD}-\eqref{boundary} as that in \cite{chs2020-mhd} due to the slip boundary condition and non-conservation of temperature equation. More precisely, the basic energy equality is expressed as
\begin{equation}\label{11a}
\begin{aligned}
&\displaystyle  E'(t) + \int\left(\frac{\lambda(\div u)^{2}+2\mu|\mathbb{D}u|^2+\nu|\curl H|^{2}}{\theta}+\frac{\kappa|\nabla\theta|^2}{\theta^2}\right) dx \\
\displaystyle  =&-\mu\int\left(2(\div u)^{2}+|\omega|^2-2|\mathbb{D}u|^2\right)dx,
\end{aligned}
\end{equation}
where the basic energy $E(t)$ is defined by
\begin{equation*}\label{b-energy}
  E(t)=\int \Big(\frac{1}{2}\rho |u|^{2}+ R G(\rho)+\frac{R}{\gamma-1}\rho\Phi(\theta)+\frac{1}{2}|H|^{2}\Big)dx,
\end{equation*}
with $G(\rho),\Phi(\theta)$ is defined in \eqref{def-g}.
Note that the right-hand term in \eqref{11a} is sign-undetermined, thus it seems difficult to obtain directly the usual standard energy estimate $E(t)\le CC_0$ in \cite{chs2020-mhd}.
To deal with this difficulty, we first assume that  $A_2(T)$ (see \eqref{As2}) a priori satisfies $A_2(T)\le 2C_0^{1/4}$ and  obtain the ``weaker" basic energy estimate (see also \eqref{basic1}):
\be\la{weaken} E(t) \le CC_0^{1/4},\ee
which will bring us some essential difficulties to obtain all the a priori estimates (see Proposition \ref{pr1}). Fortunately, we get the uniformly positive lower and upper bounds on the average of the pressure $\bp$ (see \eqref{p-b}) by both ``weaker" basic energy estimate \eqref{weaken} and Jensen's inequality (see \eqref{jen}).
Hence we can replace $\bar \te$ by $\bar P $ to overcome the difficulties caused by ``weaker" basic energy estimate. For the magnetic field, with the help of the magnetic diffusivity structure, we obtain the estimates of $\curl H$ and $\curl^2 H$ to control $\nabla H$ and $\nabla^2 H$, which can be used to deal with the strong coupling and interplay interaction between the fluid motion and the magnetic field, such as the magnetic force $(\nabla \times H)\times H$ and the convection term $\nabla \times (u\times H)$.
Next, in order to estimate $A_2(T)$, we adopt the ideas due to \cite{hl2018,cl2019} to  estimate the material derivatives $\dot u$ and $\dot\te$ (see Lemma \ref{lem-a3a4}). In addition, the slip boundary also makes the time-independent estimates of $A_3(T)$ and $A_4(T)$ more difficult. As in \cite{cl2019}, the observation $u\cdot\nabla u\cdot n=-u\cdot\nabla n\cdot u$ plays an important role to estimate the integrals on the boundary $\partial\Omega$ and combining this with the Poincar\'e-type inequality \eqref{kk} yields the estimate of $\dot{u}$ and $\dot{\te}$. Moreover, the spatial $L^2$-norm of $R\te-\overline P$ can be bounded precisely by the combination of the initial energy and the spatial $L^2$-norm of $\na \te$ (see \eqref{lem-s2}),  which lies in the central position in the process of estimating $A_2(T)$. Finally, with the aid of the uniform bound of $\overline{P}$, we rewrite the continuity equation into \eqref{rho1} and apply the Gr\"{o}nwall-type inequality (see Lemma \ref{lem-z}) and use the estimates on $R\te-\bp$, $F$ and $H$ to get the upper bound of the density.

The rest of the paper is organized as follows.
In Section \ref{se2}, we list some elementary inequalities and derive the elementary energy estimates that we use intensively in the paper.
Section \ref{se3} and Section \ref{se4} are devoted to deriving the necessary time-independent lower-order estimates and time-dependent higher-order estimates, which can guarantee the local classical solution to be a global classical one.
In Section \ref{se5}, the proof of Theorem \ref{th1}-\ref{th3} will be completed.

\section{Preliminaries}\label{se2}
In this section, we list some known facts and elementary inequalities that are used extensively in this paper. We also derive the elementary energy estimates for the system \eqref{CMHD}-\eqref{boundary} and some key a priori estimates.

We begin with the following well-known Gagliardo-Nirenberg-Sobolev-type inequality (see \cite{Nir1959}).
\begin{lemma}\label{lem-gn}
Assume that $\Omega$ is a bounded Lipschitz domain in $\r^3$. For  $p\in [2,6],\,q\in(1,\infty), $ and
$ r\in  (3,\infty),$ there exist positive constants
$C,\,\,C_1,\,\,C_2>0$ which may depend  on $p$, $q$, $r$, and $\Omega$ such that for any  $f\in H^1({\O }) $, $g\in  L^q(\O )\cap W^{1,r}(\O), $ and $\varphi,\psi\in H^2$,
\be\label{g1}\|f\|_{L^p(\O)}\le C \|f\|_{L^2}^{\frac{6-p}{2p}}\|\na
f\|_{L^2}^{\frac{3p-6}{2p}}+C_1\|f\|_{L^2} ,\ee
\be\label{g2}\|g\|_{C\left(\ol{\O }\right)} \le C
\|g\|_{L^q}^{{\frac{q(r-3)}{3r+q(r-3}})}\|\na g\|_{L^r}^{\frac{3r}{3r+q(r-3}} + C_2\|g\|_{L^2}.
\ee
\be\la{hs} \|\varphi\psi\|_{H^2}\le C \|\varphi\|_{H^2}\|\psi\|_{H^2}.\ee
Moreover, if $f\cdot n|_{\p \O}=0$ or $\overline{f}=0$, one has $C_1=0$.
Similarly, if $g\cdot n|_{\p \O}=0$ or $\overline{g}=0$, it holds  $C_2=0$.
\end{lemma}

Next, the following Gr\"{o}nwall-type inequality will be used to get the uniform (in time) upper bound of the density $\n$, whose proof can be found in \cite[Lemma 2.5]{hl2018}.
\begin{lemma}\label{lem-z}
  Let the function $y\in W^{1,1}(0,T)$ satisfy
  \be \notag
  y'(t)+\al(t) y(t)\le  g(t)\mbox{  on  } [0,T] ,\quad y(0)=y_0,
  \ee
  where $0<\al_0\le \al(t)$ for any $t\in[0,T]$ and  $ g \in L^p(0,T_1)\cap L^q(T_1,T)$  for some $p,\,q\ge 1, $  $T_1\in [0,T].$ Then it has
  \be \notag
  \sup_{0\le t\le T} y(t) \le |y_0| + (1+\al_0^{-1}) \left(\|g\|_{L^p(0,T_1)} + \|g\|_{L^q(T_1,T)}\right).
  \ee
\end{lemma}


Next, we introduce the following conclusion (see \cite[Theorem III.3.1]{Galdi1994}) to derive the exponential decay property of the solutions.
\begin{lemma} \label{lem-divf}
For the problem
\begin{equation}\label{divf}
\begin{cases}
{\rm div}v=f \,\,\,\,  in \,\,\Omega, \\
v=0\,\,\,\text{on}\,\,\,{\partial\Omega}.
\end{cases}
\end{equation}
There exists a linear operator operator $\mathcal{B} = [\mathcal{B}_1 , \mathcal{B}_2 , \mathcal{B}_3 ]$ enjoying
the properties:

1) $$\mathcal{B}:\{f\in L^p(\O)|\int_\O fdx=0\}\mapsto (W^{1,p}_0(\O))^3$$ is a bounded linear operator, that is,
\be \notag\|\mathcal{B}[f]\|_{W^{1,p}_0(\O)}\le C(p)\|f\|_{L^p(\O)}, \mbox{ for any }p\in (1,\infty),\ee

2) The function $v = \mathcal{B}[f]$ solve the problem \eqref{divf}.

3) if $f$ can be written in the form $f = \div  g$ for a certain $g\in L^r(\O), g\cdot n|_{\pa\O}=0,$  then
\be \notag\|\mathcal{B}[f]\|_{L^{r}(\O)}\le C(r)\|g\|_{L^r(\O)}, \mbox{ for any }r  \in (1,\infty).\ee
\end{lemma}

Consider the Lam\'{e}'s system
\be\label{lame1}\begin{cases}
-\mu\Delta u-(\lambda+\mu)\nabla\div u=f \,\, &in~ \Omega, \\
u\cdot n=0\,\,\text{and}\,\,\curl u\times n=0\,\,&on\,\,\partial\Omega,
\end{cases} \ee
Then, the following estimate is standard (see \cite{adn1964}).
\begin{lemma}  \label{lem-lame}
For the Lam\'{e}'s equation \eqref{lame1}, one has

(1) If $f\in W^{k,q}$ for some $q\in(1,\infty),\,\, k\geq0,$ then there exists a unique solution $u\in W^{k+2,q},$ such that
$$\|u\|_{W^{k+2,q}}\leq C(\|f\|_{W^{k,q}}+\|u\|_{L^q});$$

(2) If $f=\nabla g$ and $g\in W^{k,q}$ for some $q\geq1,\,\,k\geq0,$ then there exists a unique weak solution $u\in W^{k+1,q},$ such that
$$\|u\|_{W^{k+1,q}}\leq C(\|g\|_{W^{k,q}}+\|u\|_{L^q}).$$
\end{lemma}

The following div-curl inequalities (see \cite[heorem 3.2]{Von1992} and \cite[Propositions 2.6-2.9]{Aramaki2014}) are used to get the estimates on the spatial derivatives of the vorticity and magnetic field.
\begin{lemma}   \label{lem-vn}
Let $k\geq0$ be a integer, $1<q<+\infty$, and assume that $\Omega$ is a simply connected bounded domain in $\r^3$ with $C^{k+1,1}$ boundary $\partial\Omega$. Then for $v\in W^{k+1,q}$ with $v\cdot n=0$ on $\partial\Omega$, it holds that
\begin{align}\label{div-curl1}
\|v\|_{W^{k+1,q}}\leq C(\|\div v\|_{W^{k,q}}+\|\curl v\|_{W^{k,q}}).
\end{align}
\end{lemma}
\begin{lemma}   \label{lem-curl}
Let $k\geq0$ be a integer, $1<q<+\infty$. Suppose that $\Omega$ is a bounded domain in $\r^3$ and its $C^{k+1,1}$ boundary $\partial\Omega$ only has a finite number of 2-dimensional connected components. Then for $v\in W^{k+1,q}$ with $v\times n=0$ on $\partial\Omega$, we have
\begin{align}\label{div-curl2}
\|v\|_{W^{k+1,q}}\leq C(\|\div v\|_{W^{k,q}}+\|\curl v\|_{W^{k,q}}+\|v\|_{L^q}).
\end{align}
\end{lemma}
Next, similarly to the compressible Navier-Stokes equations, the effective viscous flux $F$ and the vorticity $\omega$ defined in \eqref{flux} plays an important role in our following analysis. Note that $\overline{F}=0$, then the standard $L^p$-estimate for the following elliptic equation:
\begin{equation*}
\begin{cases}
\Delta F=\div(\rho\dot{u}-H \cdot \nabla H)~~ &in\,\,\Omega,\\ \frac{\partial F}{\partial n}=(\rho\dot{u}-H \cdot \nabla H 
)\cdot n\,\, &on\,\, \partial\Omega.
\end{cases}
\end{equation*}
together with the div-curl type inequalities \eqref{div-curl1}-\eqref{div-curl2} yields the following essential estimates (see also \cite[Lemma 2.9]{cl2019}).
\begin{lemma}\label{lem-f-td}
Let $(\rho,u,\theta,H)$ be a smooth solution of \eqref{CMHD}-\eqref{boundary} on $\O \times (0,T]$. Then for any $p\in[2,6],\,\,1<q<+\infty,$ there exists a positive constant $C$ depending only on $p$, $q$, $\mu$, $\lambda$ and $\Omega$ such that
\begin{align}
\displaystyle & \|\nabla u\|_{L^q}\leq C(\|\div u\|_{L^q}+\|\omega\|_{L^q}),\label{tdu1}\\
& \|\nabla H\|_{L^q}\leq C\|\curl H\|_{L^q},\label{tdh1}\\
& \|\nabla F\|_{L^p}\leq C(\|\rho\dot{u}\|_{L^p}+\|H\! \cdot \nabla H\|_{L^p}),\label{tdf1}\\
& \|\nabla\omega\|_{L^p}\leq  C(\|\rho\dot{u}\|_{L^p}+\|H \cdot \nabla H\|_{L^p}+\|\nabla u\|_{L^2}),\label{tdxd-u1}\\
&\|F\|_{L^p}\leq  C(\|\rho\dot{u}\|_{L^2}\!\!+\!\|H \!\!\cdot \!\!\nabla H\|_{L^2})^{\frac{3p-6}{2p}}(\|\nabla u\|_{L^2}\!\!+\!\|P\!\!-\!\bar{P}\|_{L^2}\!\!+\!\||H|^2\!\!-\!\overline{|H|^2}\|_{L^2})^{\frac{6-p}{2p}},\label{f-lp}\\
&\|\omega\|_{L^p} \leq  C(\|\rho\dot{u}\|_{L^2}+\|H \cdot \nabla H\|_{L^2})^{\frac{3p-6}{2p}}\|\nabla u\|_{L^2}^{\frac{6-p}{2p}}+C\|\nabla u\|_{L^2},\label{xdu1}
\end{align}
Moreover,
 \begin{align}
&\|F\|_{L^p}+\|\omega\|_{L^p}\leq  C(\|\rho\dot{u}\|_{L^2}+\|H \cdot \nabla H\|_{L^2}+\|\nabla u\|_{L^2}),\label{f-curlu-lp}\\
&\begin{aligned}\label{tdu2}
&\|\nabla u\|_{L^p}\!\leq\!  C(\|\rho\dot{u}\|_{L^2}\!\!+\!\!\|H \!\!\cdot \!\!\nabla\! H\|_{L^2})^{\frac{3p-6}{2p}}(\|\nabla u\|_{L^2}\!\!+\!\|P\!\!-\!\bar{P}\|_{L^2}\!\!+\!\||H|^2\!\!-\!\!\overline{|H|^2}\|_{L^2})^{\frac{6-p}{2p}} \\
&\qquad\qquad+C(\|\nabla u\|_{L^2}+\|P\!-\!\bar{P}\|_{L^p}+\||H|^2-\overline{|H|^2}\|_{L^p}).
\end{aligned}
\end{align}
\end{lemma}

\begin{remark}
From Lemma \ref{lem-curl}, we can get the estimate of $\|\nabla^{2} H\|_{L^p}$ and $\|\nabla^{3} H\|_{L^p}$ for $p\in [2,6]$,
\begin{align}\label{2tdh}
\|\nabla^2 H\|_{L^p}\leq C \|\curl H\|_{W^{1,p}} \leq C(\|\curl^2  H\|_{L^p}+\|\curl H \|_{L^p}),
\end{align}
and
\begin{align}\label{3tdh}
\|\nabla^3 H\|_{L^p}\leq C \|\curl H\|_{W^{2,p}} \leq C(\|\curl^2  H\|_{W^{1,p}}+\|\curl H \|_{L^p}).
\end{align}
On the other hand, we can get the estimates of $\|\nabla^{2}u\|_{L^p}$ and $\|\nabla^{3}u\|_{L^p}$ for $p\in[2,6]$ by Lemma \ref{lem-vn}, which will be devoted to giving higher order estimates in Section \ref{se4}. In fact, by Lemma \ref{lem-vn} and, for $p\in[2,6]$,
\begin{align}\begin{aligned}\label{2tdu}
\displaystyle  &\quad\|\nabla^{2}u\|_{L^p}\leq C(\|\div u\|_{W^{1,p}}+\|\omega\|_{W^{1,p}})\\
&\leq C(\|\rho\dot{u}\|_{L^p}\!\!+\!\!\|H\!\! \cdot\!\! \nabla H\|_{L^p}\!\!+\!\!\|\nabla\! P\|_{L^p}\!\!+\!\!\|P\!\!-\!\!\bar{P}\|_{L^p} 
\!\!+\!\!\||H|^2\!\!-\!\!\overline{|H|^2}\|_{L^p}\!\!+\!\!\|\nabla u\|_{L^2}),
\end{aligned}
\end{align}
and
\begin{align}\begin{aligned}\label{3tdu}
 &\quad\|\nabla^{3}u\|_{L^p}\leq C(\|\div u\|_{W^{2,p}}+\|\omega\|_{W^{2,p}}) \\
&\leq C(\|\nabla(\rho\dot{u})\|_{L^p}\!\!+\!\!\|\nabla(H\!\! \cdot\!\! \nabla H)\|_{L^p}\!\!+\!\!\|\nabla^2\! P\|_{L^p}\!\!+\!\!\|H\!\! \cdot\!\! \nabla\! H\|_{L^p} \\
&\quad +\!\|\rho\dot{u}\|_{L^p}\!\!+\!\!\|\nabla P\|_{L^p}\!\!+\!\!\|P\!\!-\bar{P}\|_{L^p} \!\!+\!\!\||H|^2\!\!-\!\!\overline{|H|^2}\|_{L^p}\!\!+\!\!\|\nabla u\|_{L^2}).
\end{aligned}
\end{align}
\end{remark}

Next, we review the following Poincare-type inequality of $\dot u$ with $u\cdot n|_{\partial\Omega}=0$ (see \cite[Lemma 2.10]{cl2019}).
\begin{lemma}\label{lem-ud}
Let $\Omega \subset \mathbb{R}^3$ with $C^{1,1}$ boundary. Assume that $u$ is smooth enough and $u \cdot n|_{\p \O}=0$,  then there exists a generic positive constant $C$ depending only on   $\Omega$ such that
\begin{align}
&\|\dot{u}\|_{L^6}\le C(\|\nabla\dot{u}\|_{L^2}+\|\nabla u\|_{L^2}^2),\label{udot}\\
&\|\nabla\dot{u}\|_{L^2}\le C(\|\div \dot{u}\|_{L^2}+\|\curl \dot{u}\|_{L^2}+\|\nabla u\|_{L^4}^2).\label{tdudot}
\end{align}
\end{lemma}


Finally, we recall the following local existence theorem of classical solution of \eqref{CMHD}-\eqref{boundary}, which can be proved in a similar manner as that in \cite{fy2009,xh2017}, base on the standard contraction mapping principle.
\begin{lemma}\la{lem-local} Let $\O$ be as in Theorem \ref{th1}. Assume  that
 $(\n_0,u_0,\te_0,H_0)$ satisfies the boundary conditions \eqref{navier-b}-\eqref{boundary} and
 \be \la{2.1}
(\n_0,u_0,\te_0,H_0)\in H^3, \quad \inf\limits_{x\in\Omega}\n_0(x) >0, \quad \inf\limits_{x\in\Omega}\te_0(x)> 0.\ee
Then there exist  a small time
$T_0>0$ and a unique classical solution $(\rho , u,\te, H)$ to the problem  (\ref{CMHD})--(\ref{boundary}) on $\Omega\times(0,T_0]$ satisfying
\be\la{mn6}
  \inf\limits_{(x,t)\in\Omega\times (0,T_0]}\n(x,t)\ge \frac{1}{2}
 \inf\limits_{x\in\Omega}\n_0(x), \ee and
 \be\la{mn5}
 \begin{cases}
 (\rho, u, \te, H) \in C([0,T_0];H^3),\quad
 \n_t\in C([0,T_0];H^2),\\   (u_t, \te_t, H_t)\in C([0,T_0];H^1),
 \quad (u, \te, H)\in L^2(0,T_0;H^4).\end{cases}\ee
 \end{lemma}

 \begin{remark} Applying the same arguments as in \cite[Lemma 2.1]{hl2018}, one can deduce that the classical solution $(\rho , u,\te, H)$  obtained in Lemma \ref{lem-local} satisfies
\be \la{mn1}\begin{cases} (t u_{t},  t \te_{t}, tH_t) \in L^2( 0,T_0;H^3)  ,\quad (t u_{tt},  t\te_{tt}, tH_{tt}) \in L^2(0,T_0;H^1), \\
 (t^2u_{tt},  t^2\te_{tt},  t^2H_{tt}) \in L^2( 0,T_0;H^2), \quad
 (t^2u_{ttt},  t^2\te_{ttt},  t^2H_{ttt}) \in L^2(0,T_0;L^2).
\end{cases}\ee
Moreover, for any    $(x,t)\in \Omega\times [0,T_0],$ the following estimate holds:
   \be\la{mn2}
\te(x,t)\ge
\inf\limits_{x\in\Omega}\te_0(x)\exp\left\{-(\ga-1)\int_0^{T_0}
 \|\div u\|_{L^\infty}dt\right\}.\ee
\end{remark}

\section{\label{se3} A priori estimates(I): lower order estimates}

In this section, we will establish the time-independent a priori bounds of the solutions of the problem \eqref{CMHD}-\eqref{boundary}. Let $T>0$ be a fixed time and $(\rho,u,\theta,H)$ be a smooth solution to \eqref{CMHD}-\eqref{boundary} on $\Omega \times (0,T]$  with smooth initial data $(\rho_0,u_0,\theta_0,H_0)$ satisfying \eqref{dt1}-\eqref{dt-s}.

Set $\si=\si(t)\triangleq\min\{1,t \},$ we define
\begin{align}
 A_1(T) &\triangleq  ~\sup_{\mathclap{0\le t\le T}}
\left( \|\nabla u\|_{L^2}^2 \!+\! \|\nabla H\|_{L^2}^2 \right) \!+\! \int_0^{T} \!\! (\|\sqrt{\rho}\dot{u}\|_{L^2}^2 \!+\! \|\curl^2 H\|_{L^2}^2\!+\! \|H_t\|_{L^2}^2)dt,\label{As1}\\
 A_2(T)  &\triangleq  \frac{1}{2(\gamma-1)}~\sup_{\mathclap{0\le t\le T}}\|\sqrt{\rho}(R\theta-\bar{P})\|_{L^2}^2\!+\!\!\int_0^{T}\!\!(\|\nabla u\|_{L^2}^2\!+\!\|\nabla H\|_{L^2}^2\!+\!\|\nabla \theta\|_{L^2}^2)dt,\label{As2}\\
 A_3(T) &\triangleq  ~\sup_{\mathclap{0\le t\le T}}~\sigma\left(\|\nabla u\|_{L^2}^2\!+\!\|\nabla H\|_{L^2}^2\right)\!+\!\!\int_0^{T}\!\!\sigma(\|\sqrt{\rho}\dot{u}\|_{L^2}^2 \!+\! \|\curl^2 H\|_{L^2}^2\!+\! \|H_t\|_{L^2}^2)dt,\label{As3} \\
 A_4(T) &\triangleq  ~\sup_{\mathclap{0\le t\le T}}\sigma^2(\|\sqrt{\rho}\dot{u}\|_{L^2}^2\!+\!\|\curl^2 H\|_{L^2}^2\!+\! \|H_t\|_{L^2}^2\!+\!\|\nabla\theta\|_{L^2}^2)\nonumber\\
&\qquad+\!\!\int_0^{T}\!\!\sigma^2(\|\nabla\dot{u}\|_{L^2}^2\!+\!\|\nabla H_t\|_{L^2}^2\!+\!\!\|\sqrt{\rho}\dot{\theta}\|_{L^2}^2)dt,\label{As5}
\end{align}
where $\dot{v}=v_t+u \cdot \nabla v$ is the material derivative.

Now we will give the following key a priori estimates in this section, which guarantees the existence of a global classical solution of \eqref{CMHD}--\eqref{boundary}.
\begin{proposition}\label{pr1}
For given constants $M>0$, $\hat{\rho}> 2,$  and $\hat{\theta}> 1,$ assume further that $(\rho_0,u_0,\te_0,H_0) $  satisfies
\be \la{3.1}
0\!<\!\inf \rho_0 \le\sup \rho_0 \!<\!\hat{\rho},~ 0\!<\!\inf \te_0 \le\sup \te_0 \!\le\! \hat{\theta},~  \|\nabla u_0\|_{L^2}\leq M,~ \|\nabla H_0\|_{L^2}\leq M.
\ee
Then there exist  positive constants $K$,  $C^\ast$, $\alpha$,  $\theta_\infty$, and $\ve_0$ depending on $\mu$, $\lambda$, $\nu$, $\kappa$, $\ga$, $R$, $\hat{\rho}$, $\hat{\theta}$, $\Omega$, and $M$ such that if $(\rho,u,\theta,H)$  is a smooth solution of \eqref{CMHD}-\eqref{boundary}  on $\Omega\times (0,T] $ satisfying
\begin{equation}\label{key1}
0<\rho\le 2\hat{\rho},\quad
A_1(\sigma(T))\le 3K,\quad 
A_2(T)\leq 2C_0^{\frac14},\quad A_3(T)+A_4(T)\leq 2C_0^{\frac16},
\end{equation}
then the following estimates hold
 \begin{equation}\label{key2}
0<\rho\le \frac32\hat{\rho},\quad
A_1(\sigma(T))\le 2K,\quad 
A_2(T)\leq C_0^{\frac14},\quad A_3(T)+A_4(T)\leq C_0^{\frac16},
\end{equation}
and for any $t\geq1$,
\begin{equation}\label{key3}
\|\rho-1\|_{L^2}+\|u\|_{W^{1,6}}^2+\|H\|_{H^2}^2+\|\theta-\theta_\infty\|_{H^2}^2\le C^\ast e^{-\al t},
\end{equation}
provided   \be\la{z01}C_0\le \ve_0.\ee
\end{proposition}
\begin{proof} Proposition \ref{pr1} is a consequence of the following Lemmas \ref{lem-a1}, \ref{lem-a3a4}, \ref{lem-a2}, \ref{lem-brho}, \ref{lem-lim} below.
\end{proof}

In the following, we will use the convention that $C$ denotes a generic positive constant depending on $\mu , \lambda , \nu,  \ga ,  \kappa, R, \hat{\rho},  \hat{\theta}, \Omega$,  and $M,$ and use $C(\alpha)$ to emphasize that $C$ depends on $\alpha$. We begin with the following standard energy estimate.
\begin{lemma}\label{lem-basic2}
 Let $(\rho,u,\theta,H)$ be a smooth solution of \eqref{CMHD}-\eqref{boundary} satisfying \eqref{key1}. Then there exist positive constants $C$ and $P_1,P_2$ depending only on $\mu, \gamma, R, \hat{\rho}$ and $\Omega$ such that
 \begin{align}
&\displaystyle \sup_{0\le t\le T}(\|\sqrt{\rho}u\|_{L^2}^2+\|\rho-1\|_{L^2}^2+\|P-\bar{P}\|_{L^2}^2)\leq CC_{0}^{\frac14},\label{basic0}\\
& 0<P_1\leq \bar{P}\leq P_2, \quad for ~any~t\in[0,T].\label{p-b}
\end{align}
\end{lemma}
\begin{proof}
First, we rewrite \eqref{CMHD}$_2$ in the following form:
\begin{equation}\label{CMHD-u}
\rho u_t+\rho u \cdot \nabla u - (\lambda\! +\! 2\mu)\nabla\div u+\mu\nabla\!\times\!\omega + \nabla P=H \!\cdot\! \nabla H-\frac12\nabla |H|^2,
\end{equation}
where we used the fact $-\Delta u=-\nabla\div u+\nabla\times\omega$ and $\omega \triangleq \nabla\times u$.
Multiplying $\eqref{CMHD}_1 $ by $RG'(\rho)$, $\eqref{CMHD-u}$ by $u$, $\eqref{CMHD}_3$ by $\Phi'(\theta)$ and $\eqref{CMHD}_4$ by $H$ respectively, integrating by parts over $\Omega$, summing them up, using the boundary conditions \eqref{navier-b}-\eqref{boundary}, we have
\begin{equation}\label{m2}
\begin{aligned}
&\displaystyle  E'(t) + \int\left(\frac{\lambda(\div u)^{2}+2\mu|\mathbb{D}u|^2+\nu|\curl H|^{2}}{\theta}+\frac{\kappa|\nabla\theta|^2}{\theta^2}\right) dx \\
\displaystyle  =&-\mu\int\left(2(\div u)^{2}+|\omega|^2-2|\mathbb{D}u|^2\right)dx\leq 2\mu\int|\nabla u|^2 dx,
\end{aligned}
\end{equation}
where $E(t)$ is the basic energy defined by
\begin{equation*}
  E(t)=\int \Big(\frac{1}{2}\rho |u|^{2}+ R G(\rho)+\frac{R}{\gamma-1}\rho\Phi(\theta)+\frac{1}{2}|H|^{2}\Big)dx,
\end{equation*}
with $G(\rho)\triangleq 1+\rho\ln\rho-\rho,\ \Phi(\theta)\triangleq\theta-\ln\theta-1.$
Then, integrating \eqref{m2} with respect to $t$ over $(0,T)$ and using \eqref{key1} leads to the following energy estimates
\begin{align}
&\sup_{0\le t\le T}E(t)\leq C_0+2\mu\int_0^T\int|\nabla u|^2 dxdt \le CC_0^{\frac14}.\label{basic1}
\end{align}
It is easy to check that
\begin{align}\label{grho}
\displaystyle  (\rho-1)^{2}\geq G(\rho)=(\n-1)^2\int_0^1\frac{1-\al}{\al (\n-1)+1}d\al  \ge \frac{(\n-1)^2}{ 2(2\hat{\rho}+1)},
\end{align}
which together with \eqref{basic1} gives
\begin{align}\label{basic2}
&\displaystyle \sup_{0\le t\le T}(\|\sqrt{\rho}u\|_{L^2}^2+\|\rho-1\|_{L^2}^2)\leq CC_{0}^{\frac14}.
\end{align}
Next,  it is easy to deduce from $\eqref{CMHD}_1$  and \eqref{rho0-b} that for any $t\in[0,T]$,
\be \notag
\overline\rho(t)=\overline{\n_0}=1. \ee
Observe that $d\mu\triangleq|\Omega|^{-1}\n dx$ is a positive measure satisfying $\mu(\Omega)=1$ and $\Phi(\theta)$ is a convex function in $(0,\infty)$, it thus follows directly from Jensen's inequality ( see \cite[Theorem 3.3]{Rudin1987} )
\be\la{jen}\Phi\left(\int_\Omega fd\mu\right)\le \int_\Omega (\Phi\circ f)d\mu.\ee
and \eqref{basic1} that for any $t\in [0,T]$,
\begin{equation*}
  \Phi(\overline{\rho\theta}(t))\leq \int\Phi(\theta)|\Omega|^{-1}\n dx\leq C,
\end{equation*}
which gives \eqref{p-b}.
Finally, straight calculations show that for any $p\in [2,6],$
\begin{equation}\label{p-lp1}
  \begin{aligned}
\|R\theta-\bar{P}\|_{L^p}&\le R\|\theta- \overline\theta \|_{L^p}+C|R\overline\theta-\overline P|\\
&\leq C\|\nabla \theta\|_{L^2}+\frac{R}{|\Omega|}\left|\int(1-\rho)(\theta-\overline\theta)dx\right|\\
&\leq C\|\nabla \theta\|_{L^2}+C\|\rho-1\|_{L^2}\|\theta-\overline\theta\|_{L^2}\\
&\leq C\|\nabla \theta\|_{L^2}
  \end{aligned}
\end{equation}
Thus, it follows from \eqref{key1}, \eqref{basic1}, \eqref{basic2}, and \eqref{p-lp1} that for any $p\in[2,6]$,
\begin{equation}\label{p-lp2}
  \begin{aligned}
\|P-\bar{P}\|_{L^p}&= \| \n(R\te-\overline P )+ (\n -1)\overline P\|_{L^p}\\
&\le \|\n(R\te-\overline P )\|_{L^2}^{(6-p)/(2p)}\|\n(R\te-\overline P )\|_{L^6}^{ 3(p-2)/(2p)}+ P_2 \|\n-1\|_{L^p} \\
&\leq CC_0^{(6-p)/(16p)} \|\na\theta \|_{L^2}^{ 3(p-2)/(2p)}+ CC_0^{1/(4p)}.
  \end{aligned}
\end{equation}
Note that it holds
\begin{equation}\label{pt0}
  \begin{aligned}
\|P-\overline P\|_{L^2}^2(0)&=R^2\int(\n_0\te_0-\overline{\n_0\te_0})^2dx\\
 &\le C\int \n_0(\te_0-1)^2dx+C|1-\overline{\n_0\te_0}|^2+C\int(\n_0-1)^2dx\\
 &\le C(\hat{\n},\hat{\theta})\int \n_0\Phi(\te_0) dx+C\int(\n_0-1)^2dx\\
&\leq C(\hat{\n},\hat{\theta})C_0,
\end{aligned}
\end{equation}
where we have used \eqref{key1}, \eqref{grho}, and
 \be\la{cz1}\Phi(\te) =(\te-1)^2\int_0^1\frac{\al}{\al (\te-1)+1}d\al\geq\frac{1}{2(\|\te(\cdot,t)\|_{L^{\infty}}+1)}(\te-1)^2.\ee
Then, taking $p=2$ in \eqref{p-lp2} together with \eqref{basic2} yields \eqref{basic0}. The proof of Lemma \ref{lem-basic2} is completed.
\end{proof}

Now, we give the estimate of $A_1(T)$.
\begin{lemma}\label{lem-a1}
 Let $(\rho,u,\theta,H)$ be a smooth solution of
 \eqref{CMHD}-\eqref{boundary} satisfying \eqref{key1}.
  Then there is a positive constant
  $\ve_1>0 $, depending only on $\mu$, $\lambda$, $\nu$, $\kappa$, $\ga$, $R$, $\hat{\rho}$, $\hat{\theta}$, $\Omega$, and $M$ such that
  \begin{align}\label{basic-a1}
  \displaystyle A_1(T)\leq 2K,
  \end{align}
provided $C_0\leq \ve_1$.
\end{lemma}
\begin{proof}
First, multiplying \eqref{CMHD}$_3$ by $H$ and integrating by parts over $\Omega$, by \eqref{boundary} and \eqref{tdh1}, we have
\begin{align*}
\displaystyle \left(\frac{1}{2}\|H\|_{L^2}^2\right)_t\!+\!\nu \|\curl H\|_{L^2}^2\leq
\|\nabla u\|_{L^2}\|H\|_{L^4}^2\leq \frac{\nu}{2}\|\curl H\|_{L^2}^2\!+\!C\|\nabla u\|^4_{L^2}\|H\|^2_{L^2},
\end{align*}
which together with \eqref{tdh1}, \eqref{key1} and Gronwall's inequality gives
\begin{align}\label{tdh-2}
\displaystyle \sup_{0\le t\le T}\|H\|_{L^2}^2+\int_0^T\|\nabla H\|_{L^2}^2dt \leq
C\|H_0\|^2_{L^2}.
\end{align}
By Lemma \ref{lem-f-td}, one easily deduces from \eqref{CMHD}$_3$ and \eqref{boundary} that
\begin{equation}\label{tdh-3}
\begin{aligned}
\displaystyle & \left(\nu\|\curl H\|_{L^2}^2\right)_t+\nu^2 \|\curl^2  H\|_{L^2}^2+\|H_t\|^2_{L^2} \\
\leq & \int |H \cdot \nabla u-u \cdot \nabla H-H \div u|^2 dx  \\
\leq & C \|\nabla u\|^2_{L^2}\|H\|^2_{L^\infty}+C\|u\|^2_{L^6}\|\nabla H\|^2_{L^3}\\
\leq & C \|\nabla u\|^2_{L^2}\|\nabla H\|_{L^2}\|\curl^2  H\|_{L^2}+C\|\nabla u\|^2_{L^2}\|\nabla H\|_{L^2}^2 \\
\leq & \frac{\nu^2}{2}\|\curl^2  H\|^2_{L^2}+C(\|\nabla u\|^2_{L^2}+\|\nabla u\|^4_{L^2})\|\nabla H\|_{L^2}^2,
\end{aligned}
\end{equation}

using \eqref{tdh1} and Gronwall's inequality, we get
\begin{align}\label{tdh-4}
  \displaystyle \sup_{0\le t\le T  }\|\nabla H\|_{L^2}^2+\int_0^{T}
 \left(\|\curl^2 H\|_{L^2}^2+\|H_t\|_{L^2}^2\right)dt \leq C\|\nabla H_0\|^2_{L^2}\leq CM_2.
  \end{align}
On the other hand, multiplying \eqref{tdh-3} by $\sigma$ and integrating it over $(0,T)$, by \eqref{key1} and \eqref{tdh-2}, we obtain
\begin{align*}
\displaystyle \sup_{0\le t\le T  }\left(\sigma\|\nabla H\|_{L^2}^2\right)+\int_0^{T} \sigma
 \left(\|\curl^2 H\|_{L^2}^2+\|H_t\|_{L^2}^2\right)dt \leq C\|H_0\|^2_{L^2} \leq CC_0.
\end{align*}
Next, integrating $\eqref{CMHD-u}$ multiplied by $u_t $ over $\Omega $ by parts gives
 \be\ba \la{tdu-1}
 &\frac{d}{dt}\int \left(\frac{2\mu+\lambda}{2}(\div u)^2+ \frac{\mu}{2}|\omega|^2\right)dx+ \int\rho |\dot u|^2dx \\
&\le \int P\div u_t dx+ \int \n|u\cdot \na u|^2dx+\int(H \cdot \nabla H-\frac{1}{2} \nabla|H|^2)\cdot u_tdx\\
 &=  \frac{d}{dt}\int (P-\overline P) \div u  dx-\int (P-\overline{P})_t \div u dx+\int \n|u\cdot \na u|^2dx\\
 &\quad+\int(H \cdot \nabla H-\frac{1}{2} \nabla|H|^2)\cdot u_tdx\\
 &= \frac{d}{dt}\int (P-\overline P) \div u  dx-\frac{d}{dt}\int\frac{(P-\overline P)^2}{2(2\mu+\lambda)}dx+\frac{d}{dt}\int(H \!\cdot\!\nabla H\!-\!\frac{1}{2} \nabla|H|^2)\cdot udx\\
 &\quad-\frac{1}{2\mu\!+\!\lambda}\int P_t (F\!+\!\frac{1}{2}(|H|^2\!-\!\overline{|H|^2}))dx\!+\!\frac{1}{2\mu\!+\!\lambda}\int \overline{P}_t (F\!+\!\frac{1}{2}(|H|^2\!-\!\overline{|H|^2}))dx\\
 &\quad-\int(H \cdot \nabla H- \frac{1}{2}\nabla|H|^2)_t\cdot udx+\int \n|u\cdot \na u|^2dx,
 \ea\ee
where in the last equality  we have used the fact
\be  \notag
{\rm div}u=\frac{1}{2\mu+\lambda}(F +P-\overline P+(|H|^2-\overline{|H|^2})/2).\ee
Denote that
\begin{equation}\label{b1}
\begin{aligned}
 \widetilde{B}_1(t)\triangleq&\frac{2\mu+\lambda}{2}\|\div u\|_{L^2}^2+ \frac{\mu}{2}\|\omega\|_{L^2}^2+\frac{1}{2(2\mu+\lambda)}\|P-\overline P\|_{L^2}^2\\&-\int (P-\overline P)\div u dx-\int(H \!\cdot\!\nabla H\!-\!\frac{1}{2} \nabla|H|^2)\cdot udx,
\end{aligned}
\end{equation}
then \eqref{tdu-1} can be rewritten as
\begin{equation}\label{a1-1}
\begin{aligned}
 \widetilde{B}_1'(t)+ \int\rho |\dot u|^2dx
\le &-\frac{1}{2\mu\!+\!\lambda}\int P_t (F\!+\!\frac{1}{2}(|H|^2\!-\!\overline{|H|^2}))dx\\
&+\frac{1}{2\mu\!+\!\lambda}\int \overline{P}_t (F\!+\!\frac{1}{2}(|H|^2\!-\!\overline{|H|^2}))dx\\
 &-\int(H \cdot \nabla H- \frac{1}{2}\nabla|H|^2)_t\cdot udx+\int \n|u\cdot \na u|^2dx\\
  \triangleq & ~I_1+I_2+I_3+I_4.
\end{aligned}
\end{equation}
We have to estimate $I_1$, $I_2$, $I_3$ and $I_4$ one by one.
Note that \eqref{CMHD}$_3$ implies
\be \la{op3} \ba
P_t=&-\div (Pu) -(\gamma-1) P\div u+(\ga-1)\ka \Delta\te\\&+(\ga-1)\left(\lambda (\div u)^2+2\mu |\mathbb{D}(u)|^2+\nu|\curl H|^2\right),
\ea\ee
which along with \eqref{navier-b} and \eqref{theta-b} gives
\be \ba \la{pt}
\overline{P}_t =&
-(\gamma-1) \overline{P\div u} +(\ga-1)\left(\lambda
\overline{(\div u)^2}+2\mu \overline{|\mathbb{D}(u)|^2}+\nu\overline{|\curl H|^2}\right).
\ea \ee
By Lemma \ref{lem-f-td}, \eqref{g1}, \eqref{p-b}, and \eqref{p-lp2}, a direct calculation gives
\be\la{a1-i1}\ba
I_1
&\le C\int P(|F||\na u|+ |u||\na F|+||H|^2\!-\!\overline{|H|^2}||\na u|+ |u||H\na H|)dx\\
&\quad
+ C\int\left(|\na\te||\na F|+|\na\te||H\na H|+|\na u|^2|F|+|\curl H|^2|F|\right)dx \\
&\quad
+ C\int\left(|\na u|^2||H|^2\!-\!\overline{|H|^2}|+|\curl H|^2||H|^2\!-\!\overline{|H|^2}|\right)dx \\
&\le C\|P\!\!-\!\!\bp\|_{L^3}\big((\|F\|_{L^6}\!\!+\!\!\||H|^2\!\!\!-\!\!\overline{|H|^2}\|_{L^6})\|\na\! u\|_{L^2}\!\!+\!\!\|u\|_{L^6}(\|\na\! F\|_{L^2}\!\!+\!\!\|H\!\na\! H\|_{L^2})\big)\\
&\quad+ C\bp\big((\|F\|_{L^6}\!+\!\||H|^2\!\!\!-\!\!\overline{|H|^2}\|_{L^6})\|\na\! u\|_{L^2}\!+\!\|u\|_{L^6}(\|\na\! F\|_{L^2}\!\!+\!\|H\!\na\! H\|_{L^2})\big)\\
&\quad + C \|\na F\|_{L^2} \|\na \te\|_{L^2}+ C \|H\na H\|_{L^2} \|\na \te\|_{L^2} + C \| F\|_{L^6} \|\na u\|_{L^2}^{3/2}  \|\na u\|_{L^6}^{\frac12} \\
&\quad + C \| F\|_{L^6} \|\na H\|_{L^2}^{3/2}  \|\na H\|_{L^6}^{\frac12}+C\|\na u\|_{L^2}^2\|H\|_{L^\infty}^2+C\|\na H\|_{L^2}^2\|H\|_{L^\infty}^2 \\
&\le \de \|\na F\|_{L^2}^2 +\de \|H\na H\|_{L^2}^2 +\de \|\sqrt{\rho}\dot u\|^2_{L^2}+\de \|\curl^2\! H\|^2_{L^2}\\
&\quad+C(\de) \left( \|\na u\|_{L^2}^2+\|\na H\|_{L^2}^4+\|\na \te\|_{L^2}^2+ \|\na u\|^6_{L^2}+ \|\na H\|^6_{L^2}\right) \\
&\le \de(\|\sqrt{\rho}\dot u\|^2_{L^2}\!\!+\!\!\|\curl^2\! H\|^2_{L^2})\!\!+\!C\left( \|\na\! u\|_{L^2}^2\!\!+\!\!\|\na\! H\|_{L^2}^4\!\!+\!\!\|\na\! \te\|_{L^2}^2\!\!+\!\! \|\na\! u\|^6_{L^2}\!\!+\! \!\|\na\! H\|^6_{L^2}\right),
\ea\ee
where one has used
\begin{align}
&\|H\|_{L^\infty} \leq C \|H\|_{L^6}^{\frac12}\|\nabla H\|_{L^6}^{\frac12}\leq C\|\nabla H\|_{L^2}+\|\nabla H\|_{L^2}^{\frac12}\|\curl^2 H\|_{L^2}^{\frac12},\label{h-inf}\\
&\|H\nabla H\|_{L^2} \leq C \|H\|_{L^\infty}\|\nabla H\|_{L^2}\leq C\|\nabla H\|_{L^2}^2+\|\nabla H\|_{L^2}^{\frac32}\|\curl^2 H\|_{L^2}^{\frac12}.\label{hh}
\end{align}
Next, it follows from \eqref{pt} and \eqref{basic0} that
\be\la{pt2} \ba
|\overline P_t|
&\le C\| P-\bp\|_{L^2} \| \na u\|_{L^2} + C \|\na u\|_{L^2}^2+ C \|\na H\|_{L^2}^2 \\
&\le C(C^{\frac18}_0\| \na u\|_{L^2}+\| \na u\|_{L^2}^2+ \|\na H\|_{L^2}^2).
\ea \ee
Thus, we have
\be\la{a1-i2}\ba
I_2\le& C(\| \na u\|_{L^2}+\| \na u\|_{L^2}^2+ \|\na H\|_{L^2}^2)\| F\|_{L^2}\\
\le& \de \|\na F\|_{L^2}^2  +C(\de,\on)(\|\na u\|_{L^2}^2+\|\na u\|_{L^2}^4+ \|\na H\|_{L^2}^4)\\
\le&  \de(\|\sqrt{\rho}\dot u\|^2_{L^2}\!+\! \|\curl^2 H\|_{L^2}^2)\!+\!C(\de) (\|\na u\|_{L^2}^2\!+\!\|\na u\|_{L^2}^4\!+\! \|\na H\|_{L^2}^4\!+ \!\|\na H\|_{L^2}^6).
\ea\ee
Next, by \eqref{navier-b} and \eqref{h-inf}, a direct calculation yields
\begin{equation}\label{a1-i3}
  \begin{aligned}
I_3= & \int \big((H \otimes H)_t:\nabla u- (|H|^2/2)_t\div u\big)dx\\
   \leq &C\|H_t\|_{L^2}\|H\|_{L^\infty}\|\nabla u\|_{L^2}\\
    \leq  &\de(\|H_t\|_{L^2}^{2}+\|\curl^2  H\|_{L^2}^2)+C(\de)(\|\nabla H\|_{L^2}^2\|\nabla u\|_{L^2}^4+\|\nabla H\|_{L^2}^4+\|\nabla u\|_{L^2}^4).
  \end{aligned}
\end{equation}
Finally, it follows from Lemma \ref{lem-f-td} and \eqref{key1} that
\be\la{a1-i4}\ba
I_4&\le C\|u\|_{L^6}^2 \|\na u\|_{L^2} \|\na u\|_{L^6}  \\
&\le \de(\|\sqrt{\rho} \dot u\|_{L^2}^2+\|\curl^2  H\|_{L^2}^2)\\&\quad+ C(\de)\left(\|\na u\|_{L^2}^2\!+\! \|\na H\|_{L^2}^2\!+ \!\|\na\te\|_{L^2}^2+\|\na
u\|_{L^2}^6\!+\! \|\na H\|_{L^2}^6\right).
\ea\ee
Then, substituting \eqref{a1-i1}, \eqref{a1-i2}, \eqref{a1-i3}, and \eqref{a1-i4} into \eqref{a1-1} yields
\begin{equation}\label{a1-2}
\begin{aligned}
\widetilde{B}_1'(t)+ \int\rho |\dot u|^2dx
&\le C\de(\|\sqrt{\rho} \dot u\|_{L^2}^2+\|H_t\|_{L^2}^{2}+\|\curl^2  H\|_{L^2}^2)\\
&+\! C(\de)\left(\|\na\! u\|_{L^2}^2 +  \|\na\! H\|_{L^2}^2\!+ \!\|\na\!\te\|_{L^2}^2+\|\na u\|_{L^2}^6\!+\! \|\na\! H\|_{L^2}^6\right).
\end{aligned}
\end{equation}
Combining \eqref{tdh-4} and \eqref{a1-2}, one obtains after choosing $\de$ suitable small that
\begin{equation}\label{a1-3}
\begin{aligned}
&B_1'(t)+ \frac12\|\sqrt{\rho} \dot u\|_{L^2}^2+\frac{\nu^2}{4}\|\curl^2  H\|_{L^2}^2+\frac12\|H_t\|^2_{L^2}\\
\le& C(\de)\left(\|\na\! u\|_{L^2}^2 +  \|\na\! H\|_{L^2}^2\!+ \!\|\na\te\|_{L^2}^2+\|\na u\|_{L^2}^6\!+\! \|\na\! H\|_{L^2}^6\right),
\end{aligned}
\end{equation}
with
\begin{equation}\label{b11}
\begin{aligned}
 B_1(t)=&\nu\|\curl H\|_{L^2}^2+\widetilde{B}_1(t),
\end{aligned}
\end{equation}
where $\widetilde{B}_1(t)$ is defined by \eqref{b1}.
Then, integrating \eqref{a1-3} over $(0,T)$, one deduces from Lemma \ref{lem-f-td}, \eqref{key1}, \eqref{basic0}, and \eqref{pt0} that
\begin{equation}\label{a1-4}
\begin{aligned}
A_1(T)&\le CM_1^2+CM_2^2+C(\on,\hat\te)C_0^{\frac14} \\
&+ C(\on ) C_0^{\frac14}\sup_{0\le t\le T}(\|\na u\|_{L^2}^4+\|\na H\|_{L^2}^4) + C(\on ) C_0^{\frac18}\sup_{0\le t\le T}\|\na u\|_{L^2}\\
&\le CM_1^2+CM_2^2+C(\on,\hat\te)C_0^\frac{1}{12}+ C(\on ) C_0^{\frac14}\sup_{0\le t\le T}(\|\na u\|_{L^2}^4+\|\na H\|_{L^2}^4) \\
&\le K+9C(\on)C_0^{\frac14}K^2 \\
&\le 2K,
\end{aligned}
\end{equation}
with $K\triangleq CM_1^2+CM_2^2+C(\on,\hat\te) +1$, provided $C_0\le \ep_1 \triangleq \min\left\{1,\xl(9C(\on)K\xr)^{-4}\right\}.$
The proof of Lemma \ref{lem-a1} is completed.
\end{proof}

Next, motivated by  Huang-Li \cite{hl2018},  we will   establish some elementary estimates on $\dot u$ and $\dot\te$ in the following Lemma  \ref{lem-dot}, which are crucial to  estimate $A_3(T)$. The boundary terms shall be solved by the ideas in \cite{cl2019}.

\begin{lemma}\label{lem-dot}
Let $(\rho,u,\theta,H)$ be a smooth solution of
 \eqref{CMHD}-\eqref{boundary} satisfying \eqref{key1}.
  Then there exist positive constants $C$, $ C_1$, and $C_2$ depending only on
  $\ve_1>0 $, depending only on $\mu$, $\lambda$, $\nu$, $\kappa$, $\ga$, $R$, $\hat{\rho}$, $\hat{\theta}$, $\Omega$, and $M$ such that, for any $\eta\in (0,1]$ and $m\geq0,$
the following estimates hold:
\begin{equation}\label{b1-d}
\begin{aligned}
&(\sigma B_1)'(t) + \sigma \left(\frac12\|\sqrt{\rho} \dot u\|_{L^2}^2+\frac{\nu^2}{4}\|\curl^2  H\|_{L^2}^2+\frac12\|H_t\|^2_{L^2}\right)\\
&\le   C C_0^{\frac14} \sigma' + C\left(\|\na u\|_{L^2}^2+\|\na H\|_{L^2}^2+\|\na\te\|_{L^2}^2\right),
 \end{aligned}
\end{equation}
\begin{equation}\label{b2-d}
\begin{aligned}
&(\sigma^{m} B_2)'(t) + C_1\sigma^{m} \left(\|\nabla \dot u\|_{L^2}^2+\|\nabla H_t\|^2_{L^2}\right)\\
&\le   C_2\sigma^{m}\|\sqrt{\rho} \dot \theta\|_{L^2}^2+ Cm\sigma^{m-1}\sigma' \|H_t\|\ltwo+C\sigma^{m}\|\nabla u\|_{L^4}^4+C\sigma^{m}\|\theta\nabla u\|_{L^2}^2\\
&\quad+C(\sigma^m\!+\!\sigma^{m-1}\sigma')(\|\nabla H\|_{L^2}^2\!+\!\|\nabla u\|_{L^2}^2)(\|\sqrt{\rho}\dot{u}\|_{L^2}^{2}\!+\!\|\curl^2  H\|\ltwo\!+\!\|H_t\|\ltwo)\\
&\quad+ C\left(\|\na u\|_{L^2}^2+\|\na H\|_{L^2}^2\right),
 \end{aligned}
\end{equation}
\begin{equation}\label{b3-d}
\begin{aligned}
&(\sigma^{m} B_3)'(t) + \sigma^{m} \|\sqrt{\rho} \dot \theta\|_{L^2}^2 \\
&\le   C \eta \sigma^{m}\left(\|\nabla \dot u\|_{L^2}^2+\|\nabla H_t\|^2_{L^2}\right)+ C\|\na\te\|_{L^2}^2+C\sigma^{m}(\|\nabla u\|_{L^4}^4+\|\nabla H\|_{L^4}^4)\\
&\quad+C\sigma^{m}(\|\theta\nabla u\|_{L^2}^2+\|\theta\curl H\|_{L^2}^2),
 \end{aligned}
\end{equation}
where
$B_1$ is defined by \eqref{b11} and
\begin{align}
&B_2(t)\triangleq \frac12\|\sqrt{\rho} \dot u\|_{L^2}^2+\!\frac12\|H_t\|^2_{L^2}\!+\int_{\partial\Omega}\sigma^{m} (u\cdot\nabla n\cdot u)Fds,\label{b22}\\
&B_3(t)\triangleq\frac{\ga-1}{R}\left(\ka \|\na
\te\|_{L^2}^2-2 \int (\lambda (\div u)^2+2\mu|\mathbb{D}u|^2+\nu|\curl H|^2)\te dx\right).\label{b33}
\end{align}
\end{lemma}
\begin{proof}
First, multiplying \eqref{a1-3} by $\sigma$ and integrating it over $(0,T)$, by \eqref{key1} and \eqref{basic0} yields \eqref{b1-d}.
Now we will claim \eqref{b2-d}. By the definition of $\dot{u}$ and $F$, we rewrite \eqref{CMHD}$_2$ as
\begin{equation}\label{CMHD-uu}
\rho \dot{u}=F-\mu\nabla\!\times\!\omega + \div (H\otimes H).
\end{equation}
Operating $ \sigma^{m}\dot{u}^{j}[\pa/\pa t+\div (u\cdot)] $ to $ \eqref{CMHD-uu}^j,$ summing with respect to $j$, and integrating over $\Omega,$ together with $ \eqref{CMHD}_1 $, we get
\begin{equation}\label{J0}
\begin{aligned}
\displaystyle  &\left(\frac{\sigma^{m}}{2}\int\rho|\dot{u}|^{2}dx\right)_t-\frac{m}{2}\sigma^{m-1}\sigma'\int\rho|\dot{u}|^{2}dx  \\
& = \int_{\partial\Omega}\sigma^{m}F_t\dot{u}\cdot nds-\int\sigma^{m}\big(F_t\,\div\dot{u}-u \cdot \nabla\dot{u}^j\partial_jF\big)dx\\
&\quad+\mu\int\sigma^{m}(-\dot{u}\cdot\nabla\times\omega_t-\dot{u}^{j}\div((\nabla\times\omega)^j\,u))dx  \\
&\quad+\int\sigma^{m}(\dot{u}\cdot(\div (H\otimes H))_t+\dot{u}^{j}\div((\div (H\otimes H^j)\,u))dx  \\
& \triangleq J_1+ J_2+ J_3+J_4.
\end{aligned}
\end{equation}
Let us estimate $J_1, J_2, J_3$ and $J_4$.
It is necessary to estimate the boundary terms $J_1$ and we will use \eqref{key1}, Lemma \ref{lem-f-td}, Lemma \ref{lem-ud}, Sobolev trace theorem and Young's, H\"{o}lder's, Poincar\'{c0}'s, Sobolev's inequalities. 
\begin{align}\label{J10}
J_1=&-\int_{\partial\Omega}\sigma^{m}F_t\,(u\cdot\nabla n\cdot u)ds \nonumber\\
= & -\left(\int_{\partial\Omega}\sigma^{m}(u\cdot\nabla n\cdot u)Fds\right)_t+m\sigma^{m-1}\sigma'\int_{\partial\Omega}(u\cdot\nabla n\cdot u)Fds \nonumber\\
&\quad +\int_{\partial\Omega}\sigma^{m}\big(F\dot{u}\cdot\nabla n \cdot u+Fu\cdot\nabla n \cdot\dot{u}\big)ds \nonumber \\
&\quad  -\int_{\partial\Omega}\sigma^{m}\big(F(u \cdot \nabla) u\cdot\nabla n \cdot u+Fu\cdot\nabla n \cdot(u \cdot \nabla) u\big)ds\nonumber\\
\leq& -\left(\int_{\partial\Omega}\sigma^{m}(u\cdot\nabla n\cdot u)Fds\right)_t+Cm\sigma^{m-1}\sigma'\|\nabla u\|_{L^2}^{2}\|\nabla F\|_{L^2} \nonumber\\
&\quad+\delta\sigma^{m}\|\dot{u}\|_{H^1}^2+C\sigma^{m}\|\nabla u\|_{L^2}^{4}+C\sigma^{m}\|\nabla u\|_{L^2}^{2}\|\nabla F\|_{L^2} \\
&\quad +C\|\na F\|_{L^6} \|\na u\|^3_{L^2}+ C  \|  F\|_{H^1}\|\na u\|_{L^2} \left(\|\na u\|^2_{L^4} +\|\na u\|^2_{L^2}\right)\nonumber\\
\leq& -\left(\int_{\partial\Omega}\sigma^{m}(u\cdot\nabla n\cdot u)Fds\right)_t+Cm\sigma^{m-1}\sigma'\|\nabla u\|_{L^2}^{2}(\|\sqrt{\rho}\dot{u}\|_{L^2}^2+\|\curl^2H\|_{L^2}^2) \nonumber\\&
\quad+\delta\sigma^{m}\|\nabla\dot{u}\|_{L^2}^2+C\sigma^{m}\|\nabla u\|_{L^2}^{2}(\|\sqrt{\rho}\dot{u}\|_{L^2}^2+\|\curl^2H\|_{L^2}^2)\nonumber\\
&\quad+C\sigma^{m}\|\nabla u\|_{L^2}^2(\|\nabla u\|_{L^2}^{4}+1)+C\sigma^{m}\|\nabla H\|_{L^2}^4(\|\nabla H\|_{L^2}^{2}+1), \nonumber
\end{align}
where we used the fact that
\begin{equation}\label{bz5}
\begin{aligned}
\|F\|_{H^1}+\|\omega\|_{H^1}
&\le C(\|\n \dot u\|_{L^2}+\|H\na  H\|_{L^2}+\|\na  u\|_{L^2})\\
&\le C(\|\n \dot u\|_{L^2}\!+\!\|\nabla H\|_{L^2}^2\!+\!\|\nabla H\|_{L^2}^{\frac32}\|\curl^2 H\|_{L^2}^{\frac12}\!+\!\|\na  u\|_{L^2}),
\displaystyle
\end{aligned}
\end{equation}
and
\be \la{bz6}\ba  &\|\na F\|_{L^6}+\| \dot u\|_{H^1}
 \le C(\|\n \dot u\|_{L^6}+\|H \na H\|_{L^6})+\| \dot u\|_{H^1}
\\&\le C(\|\na \dot u\|_{L^2}\!+\!\|\nabla H\|_{L^2}^2\!+\!\|\nabla H\|_{L^2}^{\frac12}\|\curl^2 H\|_{L^2}^{\frac32}\!+\!\|\na  u\|_{L^2}^2).\ea \ee
By \eqref{navier-b} and \eqref{CMHD}$_1$, a direct computation yields
\begin{align}\label{J20}
J_2 
& = -\int\sigma^{m}F_t\,\div\dot{u}dx- \int\sigma^{m}u \cdot \nabla\dot{u}^j\partial_jFdx \nonumber \\
& =  - (\lambda+2\mu)\int\sigma^{m}(\div\dot{u})^{2}dx + (\lambda+2\mu)\int\sigma^{m}\div\dot{u}\,\nabla u:\nabla udx \nonumber\\
& \quad -\int\sigma^{m} P\div\dot{u}\,\div udx+\int\sigma^{m}\div \dot{u}\, u \cdot \nabla F dx \nonumber \\
 &\quad +\int\sigma^{m}\div\dot{u}\,H \cdot H_tdx+\int\sigma^{m}\div\dot{u}\,u \cdot \nabla H \cdot H dx \nonumber\\
 &\quad +R\int\sigma^{m}\div\dot{u}\,\rho \dot{\theta}dx-R\overline{\rho \dot{\theta}}\int\sigma^{m}\div\dot{u}dx-\int\sigma^{m}u \cdot\nabla\dot{u}^j\partial_jF dx \nonumber\\
& \leq  - (\lambda\!+\!2\mu)\int\sigma^{m}(\div\dot{u})^{2}dx + C\sigma^{m}\|\nabla\dot{u}\|_{L^2}\|\nabla u\|_{L^4}^2+C\sigma^{m}\|\nabla\dot{u}\|_{L^2}\|\theta\nabla u\|_{L^2}\\
& \quad+ C\sigma^{m}\|\nabla\dot{u}\|_{L^2}\|\nabla F\|_{L^3}\|u\|_{L^6}+ C\sigma^{m}\|\nabla\dot{u}\|_{L^2}\|H_t\|_{L^2}^{\frac12}\|\nabla H_t\|_{L^2}^{\frac12}\|H\|_{L^6}\nonumber\\
& \quad+ C\sigma^{m}\|\nabla\dot{u}\|_{L^2}\|H\|_{L^6}\|\nabla H\|_{L^6}\|u\|_{L^6}+ C\sigma^{m}\|\nabla\dot{u}\|_{L^2}\|\sqrt{\rho}\dot{\theta}\|_{L^2}\nonumber\\
& \leq  - (\lambda\!+\!2\mu)\int\sigma^{m}(\div\dot{u})^{2}dx + \delta\sigma^{m}(\|\nabla\dot{u}\|_{L^2}^2\!+\!\|\nabla H_t\|_{L^2}^2)+\!C\sigma^{m}\|\nabla u\|_{L^4}^4\nonumber\\
& \quad+C\sigma^{m}\|\nabla u\|_{L^2}^{2}(\|\sqrt{\rho}\dot{u}\|_{L^2}^2+\|\curl^2H\|_{L^2}^2)+\!C\sigma^{m}(\|\sqrt{\rho}\dot{\theta}\|_{L^2}^2\!+\!\|\theta\nabla u\|_{L^2}^2)\nonumber\\
& \quad+C\sigma^{m}\|\nabla u\|_{L^2}^2(\|\nabla u\|_{L^2}^{4}+1)+C\sigma^{m}\|\nabla H\|_{L^2}^4(\|\nabla H\|_{L^2}^{2}+1),\nonumber
\end{align}
where in the second equality we have used
\begin{equation*}
\begin{aligned}
\displaystyle F_t
=&(2\mu+\lm)\div\dot  u-(2\mu+\lm)\na u:\na u  - u\cdot\na F+u \cdot \nabla H \cdot H\\&+ P\div u-H \cdot H_t-R\rho\dot{\theta}+R\overline{\rho\dot{\theta}}.
\end{aligned}
\end{equation*}
Next, by $ \omega_t=\curl \dot u-u\cdot \na \omega-\na u^i\times \pa_iu$ and a straightforward calculation leads to
\begin{align}\label{J30}
J_3 
&=
-\mu\int\sigma^{m}|\curl \dot{u}|^{2}dx+\mu\int\sigma^{m}\curl\dot{u}\cdot(\nabla u^i\times\nabla_i u)dx \nonumber\\
&\quad+\mu\int\sigma^{m}u \cdot \nabla\omega \cdot \curl\dot{u}dx+\mu\int\sigma^{m} u \cdot\nabla\dot{u} \cdot(\nabla\times \omega)  dx\nonumber\\
&\leq -\mu\int\sigma^{m}|\curl \dot{u}|^{2}dx+\delta\sigma^{m}\|\nabla\dot{u}\|_{L^2}^2+C\sigma^{m}\|\nabla u\|_{L^4}^4 \\
&\quad +C\sigma^{m}\|\nabla u\|_{L^2}^{2}(\|\sqrt{\rho}\dot{u}\|_{L^2}^2+\|\curl^2H\|_{L^2}^2)+C\sigma^{m}(\|\nabla u\|_{L^2}^4+\|\nabla H\|_{L^2}^6).\nonumber
\end{align}

Finally, a directly computation shows that
\begin{align}\label{J40}
\displaystyle J_4 
& =-\int\sigma^{m}\nabla\dot{u}:(H\otimes H)_t dx-\mu\int\sigma^{m}H \cdot \nabla H^j u \cdot \nabla\dot{u}^{j}dx\nonumber \\
& \leq C\sigma^{m}(\|\nabla\dot{u}\|_{L^2} \|H\|_{L^6} \|H_t\|_{L^3}+\|\nabla\dot{u}\|_{L^2} \|H\|_{L^6} \|\nabla H\|_{L^6}\|u\|_{L^6} ) \nonumber \\
& \leq \delta\sigma^{m}(\|\nabla\dot{u}\|_{L^2}^2+\|\nabla H_t\|\ltwo)+C\sigma^m(\|\nabla H\|_{L^2}^4+\|\nabla u\|_{L^2}^4)\|\curl^2  H\|\ltwo \\
& \quad +C\sigma^m \|\nabla H\|_{L^2}^4\|H_t\|_{L^2}^2+C\sigma^m (\|\nabla H\|_{L^2}^6+\|\nabla u\|_{L^2}^6).\nonumber
\end{align}
Hence, submitting \eqref{J10}, \eqref{J20}, \eqref{J30}, and \eqref{J40}, into \eqref{J0}, one obtains
\begin{align}\label{J01}
&\left(\frac{\sigma^{m}}{2}\|\rho^{\frac{1}{2}}\dot{u}\|_{L^2}^2+\int_{\partial\Omega}\sigma^{m}(u\cdot\nabla n\cdot u)Fds\right)_t+(\lambda+2\mu)\sigma^{m}\|\div\dot{u}\|_{L^2}^2+\mu\sigma^{m}\|\curl\dot{u}\|_{L^2}^2 \nonumber\\
&\leq 4\delta\sigma^{m}(\|\nabla\dot{u}\|\ltwo+\|\nabla H_t\|\ltwo)+\!C\sigma^{m}(\|\sqrt{\rho}\dot{\theta}\|_{L^2}^2\!+\!\|\theta\nabla u\|_{L^2}^2+\|\nabla u\|^4_{L^4})\\
& \quad +C\sigma^m(\|\nabla H\|_{L^2}^2+\|\nabla u\|_{L^2}^2)(\|\sqrt{\rho}\dot{u}\|_{L^2}^{2}+\|\curl^2  H\|\ltwo+\|H_t\|\ltwo) \nonumber\\
& \quad +Cm\sigma^{m-1}\sigma'\|\nabla u\|_{L^2}^{2}(\|\sqrt{\rho}\dot{u}\|_{L^2}^{2}+\|\curl^2  H\|\ltwo)+C\sigma^{m}(\|\nabla u\|_{L^2}^{2}+\|\nabla H\|\ltwo).\nonumber
\end{align}
Next, we need to estimate the term $\|\nabla H_t\|_{L^2}$. Noticing that
\begin{equation*}
\begin{cases}
 H_{tt}+\nu \nabla \times (\curl H_t)=(H \cdot \nabla u-u \cdot \nabla H-H \div u)_t,&\text{in}\quad \Omega,\\
 H_t \cdot n=0,\quad \curl H_t \times n=0,& \text{on}\quad \partial\Omega,\\
 \end{cases}
\end{equation*}
and after directly computations we obtain
\begin{align}\label{ht1}
&\quad \left(\frac{\sigma^{m}}{2}\|H_t\|_{L^2}^2\right)_t+\nu\sigma^{m}\|\curl H_t\|_{L^2}^2-\frac{m}{2}\sigma^{m-1}\sigma' \|H_t\|\ltwo\nonumber\\
&= \int \sigma^{m}(H_t \cdot \nabla u-u \cdot \nabla H_t-H_t \div u)\cdot H_t dx \nonumber \\
&\quad+ \int \sigma^{m}(H \cdot \nabla \dot{u}-\dot{u} \cdot \nabla H-H \div \dot{u})\cdot H_t dx  \\
&\quad - \int \sigma^{m}(H \cdot \nabla (u \cdot \nabla u)-(u \cdot \nabla u)\cdot \nabla H-H \div(u \cdot \nabla u) )\cdot H_t dx \nonumber \\
& \triangleq K_1+K_2+K_3.\nonumber
\end{align}
By Lemma \ref{lem-gn} and Lemma \ref{lem-f-td}, a direct calculation leads to
\begin{align}\label{htk1}
K_1 
& \leq C\sigma^{m}(\|H_t\|_{L^3}\|H_t\|_{L^6}\|\nabla u\|_{L^2}
+\|u\|_{L^6}\|H_t\|_{L^3}\|\nabla H_t\|_{L^2})\nonumber\\
&\leq \delta\sigma^{m}\|\nabla H_t\|_{L^2}^{2}+C\sigma^{m}\|\nabla u\|_{L^2}^4\|H_t\|_{L^2}^{2}.
\end{align}
Similarly,
\begin{align}\label{htk20}
K_2 
& \leq C\sigma^{m}\|H\|_{L^6}\|H_t\|_{L^3}\|\nabla \dot{u}\|_{L^2}+
C\sigma^{m}\|\nabla H\|_{L^2}\|H_t\|_{L^3}\| \dot{u}\|_{L^6} \nonumber\\
&\leq \delta\sigma^{m}(\|\nabla \dot{u}\|_{L^2}^{2}\!+\!\|\nabla H_t\|_{L^2}^{2})\!+\!C\sigma^{m}\|\nabla H\|_{L^2}^{4}\|H_t\|_{L^2}^2\!+\!C\sigma^{m}\|\nabla u\|_{L^2}^{4}.
\end{align}
By Lemma \ref{lem-f-td}, a direct computation yields
\begin{align}\label{htk3}
K_3 
& =\int \sigma^{m} H \cdot\nabla H_t \cdot( u \cdot \nabla u) dx 
+\int_{\partial \Omega}\sigma^{m} H\cdot H_t\,(u \cdot \nabla u \cdot n) ds
\nonumber \\
& \quad -\int \sigma^{m} u \cdot \nabla u \cdot \nabla H_t \cdot H dx \nonumber \\
&\leq \int_{\partial \Omega}\sigma^{m} H\cdot H_t\,(u \cdot \nabla u \cdot n) ds+C\sigma^{m}\|H\|_{L^6}\|\nabla H_t\|_{L^2}\|\nabla u\|_{L^6}\|u\|_{L^6}\\
&\leq \int_{\partial \Omega}\sigma^{m} H\cdot H_t\,(u \cdot \nabla u \cdot n) ds+\delta\sigma^{m}\|\nabla H_t\|_{L^2}^2+C\sigma^{m}(\|\nabla H\|_{L^2}^6+\|\nabla u\|_{L^2}^6)\nonumber\\
&\quad +(\|\nabla H\|_{L^2}^4+\|\nabla u\|_{L^2}^4)(\|\sqrt{\rho}\dot{u}\|_{L^2}^2+\|\curl^2H\|_{L^2}^2+1).\nonumber
\end{align}
By Sobolev trace theorem and Lemma \ref{lem-f-td}, it indicates that
\begin{align}\label{htk31}
& \quad\int_{\partial \Omega}\sigma^{m}(u \cdot \nabla n \cdot u)( H \cdot H_t) ds\nonumber\\
&\leq C\sigma^{m}(\|u\|_{L^6}\|\nabla u\|_{L^2}\|H\|_{L^6}\|H_t\|_{L^6}+\|u\|_{L^6}^2\|\nabla H\|_{L^2}\|H_t\|_{L^6}\nonumber\\
& \quad+\|u\|_{L^6}^2\|\nabla H_t\|_{L^2}\|H\|_{L^6}+\|u\|_{L^4}^{2}\|H\|_{L^3}\|H_t\|_{L^6})\\
&\leq \delta\sigma^{m}\|\nabla H_t\|_{L^2}^{2}+C\sigma^{m}\|\nabla u\|_{L^2}^4\|\nabla H\|_{L^2}^{2}.\nonumber
\end{align}
Combining \eqref{htk3} and \eqref{htk31}, we have
\begin{equation}\label{htk32}\begin{aligned}
 K_3 &\leq 2\delta\sigma^{m}\|\nabla H_t\|_{L^2}^{2}+C\sigma^{m}(\|\nabla u\|_{L^2}^4+\|\nabla H\|_{L^2}^{4})(\|\rho^{\frac{1}{2}}\dot{u}\|_{L^2}^{2}+\|\curl^2  H\|\ltwo)\\
 &\quad+C\sigma^{m}(\|\nabla u\|_{L^2}^4+\|\nabla H\|_{L^2}^{4}+\|\nabla u\|_{L^2}^6+\|\nabla H\|_{L^2}^{6}).
\end{aligned}\end{equation}
Putting \eqref{htk1}, \eqref{htk20} and \eqref{htk32} into \eqref{ht1}, we have
\begin{align}\label{ht3}
&\quad \left(\frac{\sigma^{m}}{2}\|H_t\|_{L^2}^2\right)_t+\nu\sigma^{m}\|\nabla H_t\|_{L^2}^2\nonumber\\
&\leq 4\delta\sigma^{m}(\|\nabla \dot{u}\|_{L^2}^{2}+\|\nabla H_t\|_{L^2}^{2})+Cm\sigma^{m-1}\sigma' \|H_t\|\ltwo\\
&\quad+C\sigma^{m}(\|\nabla u\|_{L^2}^2+\|\nabla H\|_{L^2}^{2})(\|\rho^{\frac{1}{2}}\dot{u}\|_{L^2}^{2}+\|\curl^2  H\|\ltwo+\|H_t\|_{L^2}^2) \nonumber\\
 &\quad+C\sigma^{m}(\|\nabla u\|_{L^2}^2+\|\nabla H\|_{L^2}^{2}).\nonumber
\end{align}
Combining \eqref{J01} with \eqref{ht3}, we obtains
\begin{align}\label{b22-1}
&B_2'(t)+(\lambda+2\mu)\sigma^{m}\|\div\dot{u}\|_{L^2}^2+\mu\sigma^{m}\|\curl\dot{u}\|_{L^2}^2+\nu\sigma^{m}\|\nabla H_t\|_{L^2}^2 \nonumber\\
&\leq 8\delta\sigma^{m}(\|\nabla\dot{u}\|\ltwo+\|\nabla H_t\|\ltwo)+\!C\sigma^{m}(\|\sqrt{\rho}\dot{\theta}\|_{L^2}^2\!+\!\|\theta\nabla u\|_{L^2}^2+\|\nabla u\|^4_{L^4})\\
& \quad +C(\sigma^m+\sigma^{m-1}\sigma')(\|\nabla H\|_{L^2}^2+\|\nabla u\|_{L^2}^2)(\|\sqrt{\rho}\dot{u}\|_{L^2}^{2}+\|\curl^2  H\|\ltwo+\|H_t\|\ltwo) \nonumber\\
& \quad +Cm\sigma^{m-1}\sigma'\|H_t\|\ltwo+\|\curl^2  H\|\ltwo)+C\sigma^{m}(\|\nabla u\|_{L^2}^{2}+\|\nabla H\|\ltwo).\nonumber
\end{align}
Applying \eqref{tdh1} and \eqref{tdudot} to \eqref{b22-1} and choosing $\delta$ small enough infer \eqref{b2-d} directly.

Now we will prove \eqref{b3-d}. First, we rewrite \eqref{CMHD}$_3$ as the following elliptic problem:
\begin{equation}\label{CMHD-th}
\displaystyle \kappa\Delta\theta=\frac{R}{\gamma-1}\rho \dot{\theta}+R\rho\theta\mathop{\mathrm{div}} u-\lambda (\mathop{\mathrm{div}} u)^2-2\mu |\mathbb{D}(u)|^2-\nu|\mathop{\rm curl} H|^2.
\end{equation}
with the boundary condition \eqref{theta-b}. Then, we have
\begin{equation}\label{td2-th}
\displaystyle \|\na^2\te\|_{L^2} \le C (\on)\left(\|\n^{\frac12}\dot \te\|_{L^2}+ \|\na u\|_{L^4}^2+ \|\na H\|_{L^4}^2+\|\te\na u\|_{L^2}\right).
\end{equation}
Multiplying \eqref{CMHD-th} by $\sigma^{m}\dot{\theta}$ and integrating the resulting equality over $\Omega$ yield that
\begin{equation}\label{b33-1}
\begin{aligned}
&\frac{\ka
{\sigma^m}}{2}\left( \int |\nabla\theta|^2dx\right)_t+\frac{R\sigma^m}{\ga-1} \int\rho|\dot{\te}|^2dx
\\&=-\ka\sigma^m\int\na\te\cdot\na(u\cdot\na\te)dx
+\lambda\sigma^m\int  (\div u)^2\dot\te dx\\&\quad
+2\mu\sigma^m\int |\mathbb{D}(u)|^2\dot\te dx+\nu\sigma^m\int |\curl H|^2\dot\te dx-R\si^m\int\n\te \div
u\dot\te dx \\
&\triangleq \sum_{i=1}^5L_i .
\end{aligned}
\end{equation}
First, by \eqref{key1} and \eqref{td2-th}, we have
\be\notag
\ba
L_1
&\le C\sigma^{m}   \|\na u\|_{L^2}\|\na\te\|^{\frac12}_{L^2}
\|\na^2\te\|^{3/2}_{L^2}  \\
&\le  \de\sigma^{m} \|\n^{\frac12}\dot\te\|^2_{L^2}+C\si^m \left(\|\na u\|_{L^4}^4+\|\na H\|_{L^4}^4+\|\te\na u\|_{L^2}^2\right) +C\sigma^{m}
\|\na\te\|^2_{L^2}.\ea\ee
Next, it holds  that  for any $\eta\in (0,1],$
\be\notag
\ba L_2 =&\lambda\si^m\int (\div u)^2 \te_t
dx+\lambda\si^m\int (\div u)^2u\cdot\na\te
dx\\
=&\lambda\si^m
\left(\int (\div u)^2 \te dx\right)_t-2\lambda\si^m\int \te \div u
\div \dot udx\\&+2\lambda\si^m\int \te \div u \pa_i u^j\pa_j  u^i dx
+ \lambda\si^m\int u \cdot\na\left(\te   (\div u)^2 \right)dx
 \\
\le &\lambda\left(\si^m\int (\div u)^2 \te dx\right)_t-\lambda
m\si^{m-1}\si'\int (\div u)^2 \te dx\\& +\eta\si^m\|\na \dot
u\|_{L^2}^2+C(\eta)\si^m\|\te\na u\|_{L^2}^2+C\si^m\|\na u\|_{L^4}^4,\ea\ee
and
 \be \notag
 \ba L_3&\le 2\mu\left(\si^m\int
|\mathbb{D}u|^2 \te dx\right)_t-2\mu m\si^{m-1}\si'\int
|\mathbb{D}u|^2 \te dx
 \\&\quad+ \eta\si^m\|\na \dot
u\|_{L^2}^2+C(\eta)\si^m\|\te\na u\|_{L^2}^2+C\si^m\|\na u\|_{L^4}^4 .    \ea\ee
Similarly, it holds that
\be\notag
\ba L_4 =&\nu\si^m\int |\curl H|^2 \te_t
dx+\nu\si^m\int |\curl H|^2u\cdot\na\te
dx\\
=&\nu
\left(\si^m\int |\curl H|^2 \te dx\right)_t-\nu\si^{m-1}\si'\int  |\curl H|^2\te dx\\
&-2\nu\si^m\int \curl H\cdot\curl H_t \te dx+\nu\si^m\int |\curl H|^2 u\cdot \nabla \te dx \\
\le &\nu
\left(\si^m\int |\curl H|^2 \te dx\right)_t-\nu\si^{m-1}\si'\int  |\curl H|^2\te dx\\
& +\eta\si^m\|\na H_t\|_{L^2}^2+C(\eta)\si^m\|\te\curl H\|_{L^2}^2+C\si^m\|\te\na u\|_{L^2}^2+C\si^m\|\na H\|_{L^4}^4.\ea\ee
Finally, Cauchy's inequality gives
 \be\notag
 \ba
 |L_5|    \le  \delta \si^m \|\sqrt{\n}\dot\te\|_{L^2}^2+C( \delta,\on)\si^m \|\te\na u\|_{L^2}^2.  \ea\ee
Substituting  these estimates of $L_i,(i=1,\cdots,5)$ into \eqref{b33-1} and choosing $\de$ suitably small yield \eqref{b3-d} and complete the proof of Lemma \ref{lem-dot}.
\end{proof}

Now we are in a position to prove the following estimate on $A_3(T)$ and $A_4(T)$.
\begin{lemma}\label{lem-a3a4} Let $(\rho,u,\theta,H)$ be a smooth solution of
 \eqref{CMHD}-\eqref{boundary} satisfying \eqref{key1}.
  Then there is a positive constant $\ve_2>0 $, depending only on $\mu$, $\lambda$, $\nu$, $\kappa$, $\ga$, $R$, $\hat{\rho}$, $\hat{\theta}$, $\Omega$, and $M$ such that
 \begin{align}\label{a3a4}
 \displaystyle  A_3(T)+A_4(T)\le C_0^{\frac{1}{6}},
 \end{align}
provided $C_0\leq\varepsilon_2$.
\end{lemma}

\begin{proof} First, it follows from \eqref{tdu1}, \eqref{tdh1}, \eqref{p-lp2}, and \eqref{b1-d} that
	\be\ba \label{a3-1} B_1(t)
	&\ge  C\|\na u\|_{L^2}^2\!+\!C_3\|\na H\|_{L^2}^2\!-\!\delta\|\na u\|_{L^2}^2\!-\!C\|P\!-\!\bp\|^2_{L^2}\!-\!C(\delta)C_0\|\na H\|_{L^2}^2 \\
	&\ge C  (\|\nabla u\|_{L^2}^2+\|\na H\|_{L^2}^2)-C C_0^{\frac14},\ea\ee
with $C_0\leq \frac{C_3}{2C(\delta)}$ and $\delta$ small enough.
Then \eqref{a3-1} together with \eqref{key1} implies that
	\begin{align}\label{a3-0}
		A_3(T) \le CC_0^{\frac14},
	\end{align}
provided $C_0\leq \ve_{21}\triangleq\min\{1,\frac{C_3}{2C(\delta)}\}.$

Next, adding \eqref{b3-d} multiplying by $C_2+1$ to \eqref{b2-d} and choosing $\eta$ suitably small gives
\begin{equation}\label{b44}
\begin{aligned}
&\left(\sigma^m B_4\right)'(t) + \frac{C_1}{2}\sigma^m\left(\|\nabla\dot{u}\|_{L^2}^2+\|\nabla H_t\|_{L^2}^2\right) +\sigma^m\|\sqrt{\rho}\dot{\theta}\|_{L^2}^2\\
	&\le C(\si^{m-1}\si'+\si^m )(\|\sqrt{\rho}\dot{u}\|_{L^2}^{2}\!+\!\|\curl^2  H\|\ltwo\!+\!\|H_t\|\ltwo)\\
	&\quad+ C(\|\na u\|_{L^2}^2 +\|\na H\|_{L^2}^2 + \|\na \te\|_{L^2}^2)+C\si^m (\|\na u\|^4_{L^4} + \|\na H\|_{L^4}^4)\\
	&\quad+C\sigma^{m}(\|\theta\nabla u\|_{L^2}^2+\|\theta\curl H\|_{L^2}^2),
\end{aligned}
\end{equation}
where $B_4(t):=(C_2+1)B_2(t)+B_3(t)$ and $B_2(t), B_3(t)$ are defined by \eqref{b22}-\eqref{b33}.
Next, we deal with the terms of right hand of \eqref{b44}.
By \eqref{key1}, \eqref{p-b}, \eqref{p-lp1}, and \eqref{p-lp2}, one gets
\be \la{b44-3}\ba
	\|\te\na u\|_{L^2}^2
	&\le C\|R\te-\bp \|_{L^6}^2 \|\na u\|_{L^2} \|\na u\|_{L^6} + C\bp^2 \|\na u\|_{L^2}^2 \\
	&\le C\|\na \theta\|_{L^2}^2\left( \|\n^{\frac12}\dot u\|^2_{L^2} +\|\curl^2H\|^2_{L^2}+\|\na\te\|^2_{L^2} +\|\nabla H\|^4_{L^2}\right)\\
	&\quad +C\|\na \theta\|_{L^2}^2\|\na u\|_{L^2}^2+C\|\na u\|_{L^2}^2,\ea\ee
where we used the fact that
\begin{equation}\label{u-l6}
\begin{aligned}
&\|\nabla u\|_{L^6}\leq C(\|\rho\dot{u}\|_{L^2}+\|H\nabla H\|_{L^2}+\|\nabla u\|_{L^2}+\|P\!-\!\overline{P}\|_{L^6}+\||H|^2\!-\!\overline{|H|^2}\|_{L^6}) \\
&\leq C(\|\rho\dot{u}\|_{L^2}\!+\!\|\nabla H\|_{L^2}^2\!+\!\|\nabla H\|_{L^2}^{3/2}\|\curl^2 H\|_{L^2}^{\frac12}\!+\!\|\nabla u\|_{L^2}\!+\!\|\nabla \theta\|_{L^2}\!+\!C_0^{1/{24}}).
\end{aligned}
\end{equation}
due to \eqref{tdu2} and \eqref{hh}.
Similarly,
\be \la{b44-4}\ba
	\|\te\curl H\|_{L^2}^2
	&\le C\|R\te-\bp \|_{L^6}^2 \|\na H\|_{L^2} \|\curl^2H\|_{L^6} + C\|\na H\|_{L^2}^2 \\
	&\le C\|\na \theta\|_{L^2}^2\left(\|\curl^2H\|^2_{L^2}+\|\nabla H\|^2_{L^2}\right)+C\|\na H\|_{L^2}^2.\ea\ee
On the other hand, by virtue of \eqref{tdu2}, \eqref{hh} and \eqref{key1}, one obtains
\be\la{b44-5}\ba
	\|\na u\|_{L^4}^4
	&\le C(\|\n\dot u\|_{L^2}+\|H\nabla H\|_{L^2})^3(\|\na u\|_{L^2}+\|P-\bp\|_{L^2}+\||H|^2\!-\!\overline{|H|^2}\|_{L^2})\\&\quad+C(\|\na u\|_{L^2}^4+\|P-\bp\|_{L^4}^4+\||H|^2\!-\!\overline{|H|^2}\|_{L^4}^4)\\
	&\le C\left( \|\n^{\frac12}\dot u\|_{L^2}^3+\|\curl^2H\|_{L^2}^3 + \|\na \te\|_{L^2}^3\right) \\&\quad +C(\on)\|\n-1\|_{L^2}^2+C(\|\na u\|_{L^2}^2+\|\na H\|_{L^2}^2),
	\ea\ee
and
\be\la{b44-6}\ba
	&\|\na H\|_{L^4}^4\le C\|\curl^2H\|_{L^2}^3 +C\|\na H\|_{L^2}^4.
	\ea\ee
Then it follows from \eqref{key1}, \eqref{b44-5} and \eqref{b44-6} that
\be \la{b44-7}\ba
	\si (\|\na u\|_{L^4}^4+\|\na H\|_{L^4}^4)&\le C\left(\|\n^{\frac12}\dot u\|_{L^2}^2  +\|\curl^2H\|_{L^2}^2+\|\na \te\|_{L^2}^2\right)\\&\quad+C\left(\|\na u\|_{L^2}^2+\|\na H\|_{L^2}^2\right)+C(\on)\si\|\n-1\|_{L^2}^2.
	\ea\ee
Thus, taking $m=2$ in \eqref{b44}, by \eqref{key1}, \eqref{b44-3}, \eqref{b44-4} and \eqref{b44-7}, we have
\begin{equation}\label{b44-8}
\begin{aligned}
&\left(\sigma^2 B_4\right)'(t) + \sigma^2\left(\frac{C_1}{2}(\|\nabla\dot{u}\|_{L^2}^2+\|\nabla H_t\|_{L^2}^2)+\|\sqrt{\rho}\dot{\theta}\|_{L^2}^2\right) \\
	&\le C\si(\|\sqrt{\rho}\dot{u}\|_{L^2}^{2}\!+\!\|\curl^2  H\|\ltwo\!+\!\|H_t\|\ltwo)\\
	&\quad+ C(\|\na u\|_{L^2}^2 +\|\na H\|_{L^2}^2 + \|\na \te\|_{L^2}^2)+C\si\|\rho-1\|^2_{L^2}.
\end{aligned}
\end{equation}
Next, we deduce from \eqref{tdu2}, \eqref{key1}, \eqref{p-b}, \eqref{p-lp1}, and \eqref{hh} that
\begin{equation}\label{b44-1}
\begin{aligned}
&\int (\lambda (\div u)^2+2\mu|\mathbb{D}u|^2+\nu|\curl H|^2)\te dx\\
&\leq C \int |R\theta-\overline{P}|(|\nabla u|^2+|\curl H|^2)dx+C\overline{P}(\|\nabla u\|_{L^2}^2+\|\nabla H\|_{L^2}^2)\\
&\leq C\|R\theta\!-\!\overline{P}\|_{L^6}(\|\nabla u\|_{L^2}^{\frac32}\|\nabla u\|_{L^6}^{\frac12}\!+\!\|\nabla H\|_{L^2}^{\frac32}\|\nabla H\|_{L^6}^{\frac12})\!+\!C(\|\nabla u\|_{L^2}^2\!+\!\|\nabla H\|_{L^2}^2)\\
&\leq \delta(\|\nabla \theta\|_{L^2}^2\!+\!\|\sqrt{\rho}\dot{u}\|_{L^2}^2\!+\!\|H_t\|_{L^2}^2)+C\|\nabla u\|_{L^2}^2\!+\!\|\nabla H\|_{L^2}^2),
\end{aligned}
\end{equation}
where in the last inequality we have used the fact
\begin{equation}\label{h2xd1}
\begin{aligned}
\|\curl^2  H\|_{L^2}
&\leq C(\|H_t\|+\|\curl^2  H\|_{L^2}^{\frac12}\|\nabla H\|_{L^2}^{\frac12}\|\nabla u\|_{L^2}+\|\nabla H\|_{L^2}\|\nabla u\|_{L^2})\\
 &\leq \frac{1}{2}\|\curl^2  H\|_{L^2}+C(\|H_t\|+\|\nabla H\|_{L^2}\|\nabla u\|_{L^2}^2+\|\nabla H\|_{L^2}\|\nabla u\|_{L^2}),
\end{aligned}
\end{equation}
due to \eqref{CMHD}$_4$ and \eqref{tdh1}.
Then it follows that for $\delta$ suitably small,
\begin{equation}\label{b44-22}
\begin{aligned}
B_4(t)\geq & \frac{\kappa(\gamma-1)}{2R}\|\nabla \theta\|_{L^2}^2\!+\frac{1}{4}\!\|\sqrt{\rho}\dot{u}\|_{L^2}^2\!+\frac{1}{4}\!\|H_t\|_{L^2}^2\\&
+\int_{\partial\Omega}(u\cdot\nabla n\cdot u)Fds-C(\|\nabla H\|_{L^2}^2+\|\nabla u\|_{L^2}^2).
\end{aligned}
\end{equation}
Thus, we have
\begin{equation}\label{b44-2}
\begin{aligned}
\sigma^2 B_4(t)\geq & C\sigma^2(\|\nabla \theta\|_{L^2}^2\!+\!\|\sqrt{\rho}\dot{u}\|_{L^2}^2\!+\!\|H_t\|_{L^2}^2)-CC_0^{\frac14},
\end{aligned}
\end{equation}
where we used \eqref{a3-0} and the fact that
\be \ba\label{b44-9}
	&\left|\int_{\p \O}  \sigma^2 u \cdot \na n \cdot u F dS\right|
	\le C \sup_{ 0\le t\le T} \si^2   \|u\|_{H^1}^2 \|F\|_{H^1} \\
	\le & C(\si \|\na u\|_{L^2}^2)(\si (\|\sqrt{\rho} \dot u\|_{L^2}\!+\!\|\curl^2H\|_{L^2}))\!+\!(\si \|\na u\|_{L^2}^2)(\si \|\na H\|_{L^2}^2\!+\!\si \|\na u\|_{L^2})\\
	\le & C(\on) C_0^{\frac14},
	\ea \ee
due to \eqref{bz5} and \eqref{key1}.
Thus, integrating \eqref{b44-8} over $(0,T)$, using \eqref{a3-0} and \eqref{b44-2} yields
\begin{equation}\label{a4-0}
\displaystyle A_4(T)\leq CC_0^{\frac14}+C\int_0^T \si\|\rho-1\|^2_{L^2}dt.
\end{equation}
Finally, note that \eqref{CMHD-uu} is equivalent to
\be
\label{CMHD-uuu}
\bp(\n-1)=-F+(2\mu+\lambda)\div u-\n(R\te-\bp),
\ee
which together with \eqref{f-curlu-lp}, \eqref{hh}, \eqref{p-lp1} and \eqref{key1} implies
\be\ba\la{a4-1}
&\int_0^T\si\|\n-1\|_{L^2}^2 dt\\
&\le C\int_0^T\si(\|F\|_{L^2}^2+\|\na u\|_{L^2}^2) dt+C(\on)\int_0^T\|R\te-\bp\|_{L^2}^2 dt\\
&\le C(\on)\int_0^T\left(\si\|\n^{\frac12}\dot u\|_{L^2}^2+\si\|\curl^2H\|_{L^2}^2+\|\na u\|_{L^2}^2+\|\na H\|_{L^2}^2+\|\na \te\|_{L^2}^2\right)dt\\
&\le CC_0^{\frac14}.
\ea\ee
Thus, combining \eqref{a3-0} with \eqref{a4-0} and \eqref{a4-1} yields
\begin{equation*}
\displaystyle A_3(T)+A_4(T)\leq C_4C_0^{\frac14}\leq C_0^{1/6},
\end{equation*}
provided $C_0\leq \ve_{2}\triangleq\min\left\{\ve_{21},C_4^{-12}\right\}.$
The proof of Lemma \ref{lem-a3a4} is completed.
\end{proof}

Note that the modified basic energy estimate in short time $[0, \si(T)]$ is necessary, so that the spatial $L^2$-norm of $R\te-\overline P$ can be bounded precisely by the combination of the initial energy and the spatial $L^2$-norm of $\na \te$,  which lies in the central position in the process of estimating $A_2(T)$. We thus give the following lemma.
\begin{lemma}\label{lem-s}
Let $(\rho,u,\theta,H)$ be a smooth solution of
 \eqref{CMHD}-\eqref{boundary} satisfying \eqref{key1}.
  Then there exist positive constants $C$ and $\varepsilon_3$ depending only on
  $\ve_1>0 $, depending only on $\mu$, $\lambda$, $\nu$, $\kappa$, $\ga$, $R$, $\hat{\rho}$, $\hat{\theta}$, $\Omega$, and $M$ such that, for any $\eta\in (0,1]$ and $m\geq0,$
the following estimates hold:
\begin{align}\label{lem-s1}
 &\sup_{0\le t\le \si(T)}\int\left( \n |u|^2+|H|^2+(\n-1)^2 + \n\Phi(\te) \right)dx\le C C_0,
 \end{align}
and
\begin{align}\label{lem-s2}
\displaystyle \|(R\te-\bp)(\cdot,t)\|_{L^2} \le C \left(C_0^{\frac12} +C_0^\frac13\|\na\te(\cdot,t)\|_{L^2}\right), \quad \text{for}~ t\in (0,\sigma(T)],
\end{align}
provided $C_0\leq\varepsilon_3$.
\end{lemma}
\begin{proof}
First, multiplying $\eqref{CMHD}_1 $ by $RG'(\rho)$, $\eqref{CMHD-u}$ by $u$, and $\eqref{CMHD}_4$ by $H$ respectively, integrating by parts over $\Omega$, summing them up, using \eqref{navier-b}, \eqref{boundary}, and \eqref{cz1}, we have
\be \la{ss-1} \ba
	&\frac{d}{dt}\int\left(\frac{1}{2}\n |u|^2\!+\!RG(\rho)\!+\!\frac{1}{2}|H|^2
	\right)dx\\
	&\qquad+ \int(\mu|\curl u|^2\!+\!(2\mu\!+\!\lambda)(\div u)^2\!+\!\nu|\curl H|^2)dx \\
	&=  R \int \rho (\te -1) \div u dx\\
	&\le \de \|\na u\|_{L^2}^2 + C(\de, \on) \int \n(\te-1)^2dx\\
&\le\de \|\na u\|_{L^2}^2 + C(\de, \on)(\|\te(\cdot,t)\|_{L^{\infty}} +1)  \int \rho \Phi(\theta)dx.
	\ea\ee
By virtue of Lemma \ref{lem-f-td} and for $\delta$ small enough in \eqref{ss-1}, one obtains
\be \la{ss-2} \ba
	&\frac{d}{dt}\int\left(\frac{1}{2}\n |u|^2\!+\!RG(\rho)\!+\!\frac{1}{2}|H|^2
	\right)dx+ C_5(\|\nabla u\|_{L^2}^2+\|\nabla H\|_{L^2}^2) \\
	&\le C(\de, \on)(\|\te(\cdot,t)\|_{L^{\infty}} +1)  \int \rho \Phi(\theta)dx.
	\ea\ee
Then, adding \eqref{ss-2} multiplied by $(2\mu+1)C_5^{-1}$ to \eqref{m2}, one has
\be \la{ss-3} \ba
	&\left((2\mu\!+\!1)C_5^{-1}\!+\!1\right)\frac{d}{dt}\int\left(\frac{1}{2}\n |u|^2\!+\!RG(\rho)\!+\!\frac{1}{2}|H|^2
	\right)dx\\&\quad+\!\frac{R}{\gamma\!-\!1}\frac{d}{dt}\int \n \Phi(\te) dx+ \|\nabla u\|_{L^2}^2+(2\mu+1)\|\nabla H\|_{L^2}^2 \\
	&\le C(\|\te(\cdot,t)\|_{L^{\infty}} +1)  \int \rho \Phi(\theta)dx\\
	&\leq C(\|R\te-\overline{P}\|_{L^{\infty}} +1)\int \rho \Phi(\theta)dx,	\ea\ee
where in the last term we used \eqref{p-b}.
Now we proceed to estimate the term $\|R\te-\overline{P}\|_{L^{\infty}}$.
Taking $m=1$ into \eqref{b44} and integrating the resulting inequality, one deduces from \eqref{key1}, \eqref{h2xd1}, \eqref{b44-3}, \eqref{b44-4}, \eqref{b44-5}, and \eqref{b44-6} that
\begin{equation*}
\begin{aligned}
&\sigma B_4 + \int_0^T\sigma\left(\frac{C_1}{2}(\|\nabla\dot{u}\|_{L^2}^2+\|\nabla H_t\|_{L^2}^2) +\|\sqrt{\rho}\dot{\theta}\|_{L^2}^2\right)dt\\
	&\le C\int_0^T(\|\sqrt{\rho}\dot{u}\|_{L^2}^{2}\!+\!\|\curl^2  H\|\ltwo\!+\!\|H_t\|\ltwo+\|\na H\|_{L^2}^2+\|\na \te\|_{L^2}^2)dt\\
	&\quad+ C\int_0^T\sigma\|\rho-1\|_{L^2}^2dt +C\int_0^T\|\na \te\|_{L^2}^2\sigma(\|\sqrt{\rho}\dot{u}\|_{L^2}^{2}\!+\!\|H_t\|\ltwo\!+\!\|\nabla \te\|\ltwo) dt,
\end{aligned}
\end{equation*}
which together with \eqref{b44-22}, \eqref{key1} and \eqref{a4-1} yields
\begin{equation*}
\begin{aligned}
&\sigma(\|\sqrt{\rho}\dot{u}\|_{L^2}^{2}\!+\!\|H_t\|\ltwo\!+\!\|\nabla \te\|\ltwo)\\&\quad + \int_0^T\sigma\left(\frac{C_1}{2}(\|\nabla\dot{u}\|_{L^2}^2+\|\nabla H_t\|_{L^2}^2) +\|\sqrt{\rho}\dot{\theta}\|_{L^2}^2\right)dt\\
	&\le C+C\int_0^T\|\na \te\|_{L^2}^2\sigma(\|\sqrt{\rho}\dot{u}\|_{L^2}^{2}\!+\!\|H_t\|\ltwo\!+\!\|\nabla \te\|\ltwo) dt,
\end{aligned}
\end{equation*}
then Gr\"onwall inequality together with \eqref{key1} yields
\begin{equation}\label{ss-5}
\begin{aligned}
&\sigma(\|\sqrt{\rho}\dot{u}\|_{L^2}^{2}\!+\!\|H_t\|\ltwo\!+\!\|\nabla \te\|\ltwo)
\!+ \!\!\int_0^T\!\!\!\sigma\left(\|\nabla\!\dot{u}\|_{L^2}^2\!\!+\!\!\|\nabla\! H_t\|_{L^2}^2\!\!+\!\!\|\sqrt{\rho}\dot{\theta}\|_{L^2}^2\right)dt\le C.
\end{aligned}
\end{equation}
Thus, \eqref{td2-th} together with \eqref{ss-5}, \eqref{b44-3}, \eqref{b44-4}, and \eqref{a4-1} get
	\be\ba  \la{ss-6}
	\int_0^{T }\si \|\na^2\te\|_{L^2}^2dt
	&\le   C\int_0^{T } \left(\si \|\n^{\frac12}\dot\te\|_{L^2}^2+
	\|\n^{\frac12}\dot u \|_{L^2}^2+
	\|\curl^2H\|_{L^2}^2  \right) dt\\
	&\quad+ C\int_0^{T } \left( \| \na u\|_{L^2}^2+\| \na H\|_{L^2}^2+\| \na\te\|_{L^2}^2+\si\|\n-1\|_{L^2}^2\right) dt \\
	&\le C(\hat{\n},M).
	\ea\ee
Next, by \eqref{g1} and \eqref{p-lp1}, we have
\begin{equation}
	    \label{ss-7}\ba
	    \|R\te-\bp\|_{L^\infty} \le& C \|R\te-\bp\|_{L^6}^{\frac12} \|\na\te\|_{L^6}^{\frac12} + \|R\te-\bp\|_{L^2}\\
	    \le& C(\hat{\rho}) \|\na \te\|_{L^2}^{\frac12} \|\na^2 \te\|_{L^2}^{\frac12} + C(\hat{\n}) \|\na \te\|_{L^2},
	    \ea
	\end{equation}
which together with  \ref{key1} and \eqref{ss-6}  gives that
	\be\notag
  \ba
	& \int_0^{\si(T)}\|R\te-\bp\|_{L^\infty}dt \\
	&\le C(\hat{\rho}) \int_0^{\si(T)}\|\na\te\|_{L^2}^{\frac12} \left(\si\|\na^2\te\|^2_{L^2}\right)^{\frac14}\si^{-1/4}dt + C(\hat{\rho}) \left(\int_0^{\si(T)} \|\na\te\|_{L^2}^2 dt\right)^{\frac12}\\
	&\le C(\hat{\rho}) \left(\int_0^{\si(T)} \|\na \te\|_{L^2}^2dt \int_0^{\si(T)}\si\|\na^2\te\|_{L^2}^2dt\right)^{\frac14} 
+ C(\on)C_0^{\frac18}\\
	&\le C(\on,M)C_0^{1/16}.
	\ea\ee
Combining this with \eqref{ss-3}, \eqref{grho}, and Gr\"onwall inequality  implies \eqref{lem-s1} directly.

Note that
\begin{equation}\label{ss2-1}
\displaystyle \| R\te-\bp \|_{L^2} \le R \|\te  -1\|_{L^2}+C\left|\int\n(1-\te)dx\right|\le C(\on) \|\te  -1\|_{L^2}.
\end{equation}
Direct calculations together with \eqref{cz1} lead to
  \be\notag\ba
\Phi(\te)
\ge \frac{1}{8} (\te-1)1_{(\te(\cdot,t)>2)
}+\frac{1}{12}(\te-1)^21_{(\te(\cdot,t)<3)},\ea\ee
with $(\te(\cdot,t)> 2)\triangleq \left.\left\{x\in \Omega\right|\te(x,t)> 2\right\}$
and  $(\te(\cdot,t)< 3)\triangleq \left.\left\{x\in \Omega\right|\te(x,t)<3\right\}.$
Combining this with \eqref{lem-s1} gives
\be \la{ss2-2}\ba
	\sup_{0\le t\le \si(T)}\int \left(\n(\te-1)1_{(\te(\cdot,t)>2)}+\n(\te-1)^21_{(\te(\cdot,t)<3)}\right)dx \le CC_0.
	\ea\ee
Next, it follows from \eqref{ss2-2}, \eqref{lem-s1}, and the Sobolev inequality that for $t \in (0,\sigma(T)]$,
	\begin{equation}\label{ss2-3}
	\begin{aligned}
	 &\|\te -1\|_{L^2(\te(\cdot,t)<3)}^2 \\
	 &\le\int  \n (\te-1)^2 1_{(\te(\cdot,t)<3)}dx  +  \left|\int  (\n-1) (\te -1)^2 dx \right|\\
	&\le C(\on,M)C_0  +  C\|\n-1\|_{L^2} \| \te-1 \|_{L^2}^{\frac12} \|\te -1\|_{L^6}^{3/2}\\
	&\le C(\on,M)\left(C_0+ C(\delta) C_0^{2/3} \|\na \te\|_{L^2}^2 + (\delta+C_0^{\frac12}) \| \te-1 \|_{L^2}^2 \right),
	\end{aligned}
	\end{equation}
	and
	\be\label{ss2-4}\ba
	 &\|\te-1\|_{L^2(\te(\cdot,t)> 2)}^2\\
	&\le \|\te-1\|^{4/5}_{L^1(\te(\cdot,t)> 2)} \| \te-1\|_{L^6}^{6/5}\\
	&\le C(\on,M) \left(C_0 +C_0^{\frac12} \|\te-1\|_{L^2}\right)^{4/5} (\| \te-1\|_{L^2}+ \|\na \te\|_{L^2})^{6/5}\\
	&\le  C(\on,M)\left(C_0+C(\delta)C_0^{2/3}\|\na \te\|_{L^2}^2 + (\delta+C_0^{2/5} ) \|\te -1\|_{L^2}^2\right),
	\ea\ee
Hence, adding \eqref{ss2-3} with \eqref{ss2-4} together and choosing $\delta$ small enough in the resulting inequality, one has  for any $t\in (0,\sigma(T)],$
	\be\ba\notag
	\|\te-1\|_{L^2}^2 \le  C_6\left(C_0+C_0^{2/3}\|\na \te\|_{L^2}^2 + C_0^{2/5}  \|\te -1\|_{L^2}^2\right),
	\ea\ee
which implies that
	\be\ba\la{ss2-5}
	\|\te-1\|_{L^2}^2 \le  C(\on,M)\left(C_0+C_0^{2/3}\|\na \te\|_{L^2}^2 \right),
	\ea\ee
provided $C_0\le \ve_{3} \triangleq\min\left\{1,(2C_6)^{-5/2}\right\}.$
Thus, \eqref{ss2-1} together with \eqref{ss2-5} yields \eqref{lem-s2} and the proof of Lemma \ref{lem-s} is completed.
\end{proof}

\begin{lemma}\label{lem-a2}
Let $(\rho,u,\theta,H)$ be a smooth solution of
 \eqref{CMHD}-\eqref{boundary} satisfying \eqref{key1}.
  Then there exists a positive constant  $\varepsilon_4$ depending only on
  $\ve_4>0 $, depending only on $\mu$, $\lambda$, $\nu$, $\kappa$, $\ga$, $R$, $\hat{\rho}$, $\hat{\theta}$, $\Omega$, and $M$ such that the following estimates hold:
\begin{align}\label{a2-0}
A_2(T)\leq C_0^{\frac14},
 \end{align}
provided $C_0\leq\varepsilon_4$.
\end{lemma}
\begin{proof}
First, multiplying $\eqref{CMHD}_1 $ by $\overline{P}G'(\rho)$, $\eqref{CMHD-u}$ by $u$, $\eqref{CMHD}_3$ by $\overline{P}^{-1}(R\te-\overline{P})$  and $\eqref{CMHD}_4$ by $H$ respectively, integrating by parts over $\Omega$, summing them up, using \eqref{navier-b}-\eqref{boundary}, we have
\be\la{a2-1} \ba
&  \frac{dW(t)}{dt}+D(t)\\
&= -  \frac{1}{\ga-1} {\bp}^{-1} \bp_t \int
\n {(R\te-\bp)}  dx - \frac{1}{2(\ga-1)} {\bp}^{-2} \bp_t \int \n {(R\te-\bp)^2} dx\\
& \quad+\bp_t \int  G(\rho) dx - \bp^{-1}\int  \n {(R\te-\bp)^2} \div u dx \\
&\quad+ \bp^{-1} \int  {(R\te-\bp)} (\lambda (\div u)^2+2\mu |\mathfrak{D}(u)|^2+\nu \|\curl H\|_{L^2}^2) dx\\&\triangleq \sum_{i=1}^{5} M_i.
\ea\ee
where
\begin{align}
& W(t)\triangleq \int\left( \frac{1}{2}\n |u|^2+\overline{P}G(\rho)+\frac{1}{2}|H|^2+\frac{\bp^{-1}}{2(\ga-1)} \n {(R\te-\bp)^2}\right)dx,\label{ww}\\
& D(t)\triangleq \mu \|\curl u\|_{L^2}^2 +(2\mu+\lambda)\|\div u\|_{L^2}^2+\nu \|\curl H\|_{L^2}^2 +  \ka R \bp^{-1}  \|\na\te\|_{L^2}^2.\label{dd}
\end{align}
The terms $M_i \,(i=1,\cdots,5)$ can be estimated as follows.
It follows from \eqref{key1}, \eqref{p-b}, \eqref{p-lp1}, and \eqref{pt2} that
\be \ba\notag
M_1+M_2 \le & C|\overline P_t|\left(\| \n (R\te-\bp)\|_{L^2}+\| \n^{\frac12} (R\te-\bp)\|_{L^2}^2\right)\\
\le &C(\on)\left(C_0^{\frac18} \|\na u \|_{L^2}+\| \na u\|_{L^2}^2+\| \na H\|_{L^2}^2 \right)\|\n^{\frac12}(R\te-\bp)\|_{L^2}\\
\le & C(\on) C_0^{\frac18} \left( \|\na u\|_{L^2}^2+\| \na \te \|_{L^2}^2+\| \na H\|_{L^2}^2\right).
\ea \ee
Similarly, by virtue of \eqref{grho}, \eqref{basic0}, and \eqref{p-b}, we have
\be \ba\notag
M_3 
\le & C(\on)\left(  C_0^{\frac18}\|\na u \|_{L^2}+\| \na u\|_{L^2}^2+\| \na H\|_{L^2}^2 \right)\| \n -1\|_{L^2}^2\\
\le & C(\on) C_0^{\frac14}  (\|\na u\|_{L^2}^2 +\| \na H\|_{L^2}^2+ \| \n -1\|_{L^2}^2).
\ea \ee
and
\be \ba\notag
M_4 \le & C   \|\n^{\frac12} {(R\te-\bp)}\|_{L^2}^{\frac12} \|\n^{\frac12} {(R\te-\bp)}\|_{L^6}^{3/2} \|\na u\|_{L^2} \\
\le & C(\on) C_0^{1/16} \|\na \te\|_{L^2}^{3/2} \|\na u\|_{L^2}\\
\le & C(\on,M) C_0^{1/16} (\|\na u\|_{L^2}^2+\|\na \te\|_{L^2}^{2}).
\ea \ee
Now, we will estimate the term $M_5$ for the short time $t\in [0, \si(T))$ and the large time $t\in [\si(T),T]$, respectively.

For $t\in[0, \si(T))$, it follows from \eqref{key1}, \eqref{p-lp2}, \eqref{p-b}, \eqref{u-l6}, and \eqref{lem-s2} that
\be \ba\notag
M_5 \le& C \int  |R\te-\bp| |\na u|^2dx\\
\le & C \| {R\te-\bp}\|_{L^2}^{\frac12} \|R\te-\bp\|_{L^6}^{\frac12}  \|\na u\|_{L^2} \|\na u\|_{L^6}\\
\le & C(\on)  \|{R\te-\bp}\|_{L^2}^{\frac12} \|\na \te\|_{L^2}^{\frac12}  \|\na u\|_{L^2}\\
 & \left(\|\n^{\frac12}\dot u\|_{L^2}+\|\curl^2H\|_{L^2}+\|\nabla H\|_{L^2}^2+\|\na u\|_{L^2}+\|\na\te\|_{L^2}+C_0^{1/24}\right)\\
\le & C (\on) \| {R\te-\bp}\|_{L^2}^{\frac12} \|\na \te\|_{L^2}^{\frac12}  \|\na u\|_{L^2}(\|\n^{\frac12}\dot u\|_{L^2}+\|\curl^2H\|_{L^2})\\
&+C(\on,M) C_0^{1/24} (\|\na \te\|_{L^2}^{2}+\|\na H\|_{L^2}^{2}+  \|\na u\|_{L^2}^2)\\
\le & C(\on,M) C_0^{7/24} (\|\rho^{\frac12}\dot u\|_{L^2}^2+\|\curl^2H\|_{L^2})\\
& + C(\on,M)C_0^{1/24} (\|\na u\|_{L^2}^2 +\|\na H\|_{L^2}^2 +  \|\na \te\|_{L^2}^2).
\ea \ee
For $t\in[\si(T), T]$, it holds that
\be \ba\notag
M_5 
\le &C\|R\te-\bp\|_{L^3}\|\na u\|_{L^2}\|\na u\|_{L^6}
\le  C(\on)C_0^{1/24}(\|\na u\|_{L^2}^2+\|\na \te\|_{L^2}^2),
\ea \ee
where one has used \eqref{p-lp1} and the following fact:
\be\notag\ba \sup_{0\le t\le T}\xl(\si\|\na
u\|_{L^6}\xr)
&\le C(\on)C_0^{1/24} \ea\ee
due to \eqref{key1} and \eqref{u-l6}.
Finally, substituting the estimates of $M_i (i=1,\cdots,5)\ $
 into \eqref{a2-1}, one obtains after using \eqref{key1}, \eqref{p-b}, and \eqref{pt0} that
\be \ba  \label{a2-2}
&  \sup_{ 0\le t\le T} W(t) + \int_{0}^{T} D(t)dt\\
&\le  C(\on,M)C_0^{1/24} \int_{ 0}^{T} \!\!\! ( \|\na u\|_{L^2}^2\!\!+\!\|\na H\|_{L^2}^2\!\!+\!\| \na \te \|_{L^2}^2 ) dt\!\!+\! C(\on) C_0^{\frac14}\int_{ 0}^{T}\!\!\! \| \n \!-\!1\|_{L^2}^2 dt
\\&\quad + C(\on,M)C_0^{7/24} \int_{0}^{\si(T)} (\|\rho^{\frac12}\dot{u}\|_{L^2}^2+\|\curl^2H\|_{L^2}^2) dt+CC_0\\
&\le  C(\on,\hat{\te},M)C_0^{7/24},
\ea \ee
where one has used
\be\la{a2-3}
\int_{ 0}^{T} \| \n -1\|_{L^2}^2 dt \le \sup_{0\le t\le\si(T)} \|\n -1\|_{L^2}^2 + \int_{\si(T)}^{T} \| \n -1\|_{L^2}^2 dt \le C(\on,M) C_0^{\frac14},
\ee
due to \eqref{a1-4} and \eqref{lem-s1}. Thus, one deduces from \eqref{a2-2}, \eqref{tdu1}, and \eqref{p-b} that
\be \notag
A_2(T)\le C_7C_0^{7/24}\le C_0^{1/4},  \ee 
provided $C_0\le \ve_4\triangleq\min
\left\{\ve_{3},  (C_7)^{-24}\right\}$. The proof
of Lemma \ref{lem-a2} is completed.
\end{proof}

We now proceed to proof the uniform (in time) upper bound for the
density.
\begin{lemma}\label{lem-brho}
Let $(\rho,u,\theta,H)$ be a smooth solution of
 \eqref{CMHD}-\eqref{boundary} satisfying \eqref{key1}.
  Then there exists a positive constant $\ve_5$ depending only on $\mu$, $\lambda$, $\nu$, $\kappa$, $\ga$, $R$, $\hat{\rho}$, $\hat{\theta}$, $\Omega$, and $M$ such that the following estimates hold:
\begin{align}\label{brho-0}
\sup_{0\le t\le T}\|\n(t)\|_{L^\infty}  \le
\frac{3\hat{\rho} }{2},
 \end{align}
provided $C_0\leq\varepsilon_5$.
\end{lemma}
\begin{proof}
First, the equation of  mass conservation $\eqref{CMHD}_1$ can be equivalently rewritten in the form
\begin{align*}
\displaystyle  (2\mu+\lambda)D_t \n&=-\overline{P}\rho(\rho-1)-\rho^2(R\theta-\overline{P})-\rho F-\frac{\rho}{2}(|H|^2-\overline{|H|^2})\\
&\leq -\overline{P}(\rho-1)+C(\hat{\rho})(\|R\theta-\overline{P}\|_{L^\infty}+\|F\|_{L^\infty}+\||H|^2-\overline{|H|^2}\|_{L^\infty}),
\end{align*}
which gives
\begin{align}\label{rho1}
\displaystyle  D_t (\n-1)+\alpha(\n-1)\leq g(t),
\end{align}
where
\begin{align*}
 &\displaystyle D_t\rho\triangleq\rho_t+u \cdot\nabla \rho ,\qquad
\alpha\triangleq\frac{\overline{P}}{2\mu+\lambda},\\
 &g(t)\triangleq C(\hat{\rho})(\|R\theta-\overline{P}\|_{L^\infty}+\|F\|_{L^\infty}+\||H|^2-\overline{|H|^2}\|_{L^\infty}).
 \end{align*}
Naturally, we shall prove our conclusion by Lemma \ref{lem-z}.
It follows from \eqref{tdf1}, \eqref{hh}, \eqref{h2xd1}, and \eqref{udot} that
\begin{equation}\label{f-inf}
\begin{aligned}
&\|F\|_{L^\infty}\leq C\|\nabla F\|_{L^2}^{\frac12}\|\nabla F\|_{L^6}^{\frac12}\\
&\leq C(\|\sqrt{\n}\dot{u}\|_{L^2}\!\!+\!\|H\nabla\! H\|_{L^2})^{\frac12}(\|\nabla\!\dot{u}\|_{L^2}\!\!+\!\|\nabla\! u\|_{L^2}^2\!\!+\!\|\nabla\! H\|_{L^2}^2\!\!+\!\|\nabla\! H\|_{L^2}^{\frac12}\| H_t \|_{L^2}^{3/2})^{\frac12}\\
&\leq C(\|\sqrt{\n}\dot{u}\|_{L^2}^{\frac12}\!+\!\| H_t \|_{L^2}^{\frac12})\|\nabla\!\dot{u}\|_{L^2}^{\frac12}\!+\!\|\nabla\! H\|_{L^2}\|\nabla\!\dot{u}\|_{L^2}^{\frac12}\!+\!C\|\nabla\! H\|_{L^2}\| H_t \|_{L^2}\\
&\quad +\!C(\!\|\nabla\! u\|_{L^2}\!\!+\!\|\nabla\! H\|_{L^2})(\|\sqrt{\n}\dot{u}\|_{L^2}^{\frac12}\!+\!\| H_t \|_{L^2}^{\frac12})\!+\!C\|\nabla\! H\|_{L^2}^{\frac14}\|\sqrt{\n}\dot{u}\|_{L^2}^{\frac12}\|H_t\|_{L^2}^{\frac34},
\end{aligned}
\end{equation}
and
\begin{equation}\label{h2-inf}
\displaystyle  \||H|^2\!-\!\overline{|H|^2}\|_{L^\infty}\leq C\|H \nabla\! H\|_{L^2}^{\frac12}\|H \nabla\! H\|_{L^6}^{\frac12}\leq C(\|\nabla\! H\|_{L^2}^2\!+\!\|\nabla\! H\|_{L^2}\|H_t\|_{L^2}).
\end{equation}
Then we deduce from \eqref{f-inf}, \eqref{h2-inf}, \eqref{ss-5}, and \eqref{key1} that
\be\notag\ba &\int_0^{\si(T)}(\|F\|_{L^\infty}+\||H|^2\!-\!\overline{|H|^2}\|_{L^\infty})dt\\
&\le C(\on)\int_0^{\si(T)}\!\!\left(\si(\|\n \dot u\|_{L^2}\!\!+\!\| H_t \|_{L^2})\right)^{\frac14} \left(\si(\|\n \dot u\|_{L^2}^2\!\!+\!\| H_t \|_{L^2}^2)\right)^{\frac18} \left(\si\|\na \dot u\|^2_{L^2}\right)^{\frac14}\si^{-\frac58}dt\\
& \quad+ C(\on) \int_0^{\si(T)}\left(\si \|\nabla H\|_{L^2}^2\right)^{\frac14}\left(\si\|\na \dot u\|^2_{L^2}\right)^{\frac14} \|\na H\|_{L^2}^{\frac12} \si^{-\frac12} dt\\
& \quad+ C(\on) \int_0^{\si(T)}\left(\si \|H_t\|_{L^2}^2\right)^{\frac12}\left(\si\|\nabla H\|^2_{L^2}\right)^{\frac14} \si^{-\frac34} dt\\
& \quad+ C(\on) \int_0^{\si(T)}\left(\si(\|\n \dot u\|_{L^2}\!\!+\!\| H_t \|_{L^2})\right)^{\frac12}(\!\|\nabla\! u\|_{L^2}\!\!+\!\|\nabla\! H\|_{L^2})\si^{-\frac12} dt\\
& \quad+ C(\on)\int_0^{\si(T)}\!\!\left(\si\|\n \dot u\|_{L^2}^2\right)^{\frac14} \left(\si\| H_t \|_{L^2}^2\right)^{\frac38} \left(\si\|\nabla H\|^2_{L^2}\right)^{\frac18}\si^{-\frac34}dt\\
&\le C(\on,M)C_0^{1/48}\left(\int_0^{\si(T)} \si\|\na \dot u\|^2_{L^2} dt\right)^{\frac14}\left(\int_0^{\si(T)} \si^{-\frac56}dt\right)^{\frac34} \\
&\quad + C(\on,M)C_0^{1/32}\left(\int_0^{\si(T)} \si\|\na \dot u\|^2_{L^2} dt\right)^{\frac14}\left(\int_0^{\si(T)} \si^{-\frac23}dt\right)^{\frac34} \\
&\quad+ C(\on,M) C_0^{1/24} \int_0^{\si(T)}
(\si^{-\frac34}+\si^{-\frac12}) dt+ C(\on,M) C_0^{1/48} \int_0^{\si(T)}
\si^{-\frac12} dt\\
&\le C(\on,M)C_0^{1/48},
\ea\ee
which together with \eqref{ss-5} yields
\begin{equation}\label{rho-1}
\displaystyle \int_0^{\sigma(T)} g(t)dt\leq C(\on,M)C_0^{1/48}.
\end{equation}
On the other hand, it follows from \eqref{ss-6}, \eqref{ss-7}, \eqref{f-inf}, \eqref{h2-inf}, and \eqref{key1} that
\begin{equation}\label{rho-2}
\begin{aligned}
 \int_{\si(T)}^T|g(t)|^2dt
\leq& \int_{\si(T)}^T\|R\te-\bp\|^2_{L^\infty}dt+\int_{\si(T)}^T(\|F\|^2_{L^\infty}+\||H|^2\!-\!\overline{|H|^2}\|^2_{L^\infty})dt\\
\le& C\left(\int_{\si(T)}^T\!\!\!\|\na\te\|^2_{L^2}dt\right)^{\frac12}\!\!\! \left(\int_{\si(T)}^T\!\! \|\na^2\te\|^2_{L^2}dt\right)^{\frac12}\!\! +\! C\int_{\si(T)}^T\!\!\|\na\te\|^2_{L^2}dt \\
 &+ C\int_{\si(T)}^T\!\! (\|\sqrt{\n}\dot{u}\|_{L^2}^2\!\!+\!\| H_t \|_{L^2}^2+\|\nabla\!\dot{u}\|_{L^2}^2\!\!+\!\|\nabla\! H\|_{L^2}^2\!\!+\!\|\nabla\!u\|_{L^2}^2)dt \\
\le& C(\on,M) C_0^{1/8}.
\end{aligned}
\end{equation}
Thus we deduce from \eqref{rho1}, \eqref{rho-1}, \eqref{rho-2}, and Lemma \ref{lem-z} that
 \bnn\ba\n
 & \le \on+1 +C\left(\|g\|_{L^1(0,\si(T))}+\|g\|_{L^2(\si(T),T)}\right) \le \on+1 +C_8C_0^{1/48},
 \ea\enn
which gives \eqref{brho-0}
 provided $C_0\le \ve_5\triangleq\min\left\{1,\left(\frac{\hat \n-2 }{2C_8}\right)^{48}\right\}$ and  complete the proof of Lemma \ref{lem-brho}.
\end{proof}

\begin{lemma}\label{lem-th}
Let $(\rho,u,\theta,H)$ be a smooth solution of
 \eqref{CMHD}-\eqref{boundary} satisfying \eqref{key1}.
  Then there exists a positive constant  $C$ depending only on $\mu$, $\lambda$, $\nu$, $\kappa$, $\ga$, $R$, $\hat{\rho}$, $\hat{\theta}$, $\Omega$, and $M$ such that the following estimates hold:
\be \la{ae3.7}\sup_{0< t\le T}\si^2 \|\sqrt{\n}\dot\te\|_{L^2}^2dx + \int_0^T\si^2 \|\na\dot\te\|_{L^2}^2dt\le C.\ee
Moreover, it holds that
\be\la{vu15}\ba
&\sup_{0< t\le T}\left(  \si\|\na u \|^2_{L^6}+\si\|H\|^2_{H^2}+\si^2\|\te\|^2_{H^2}\right)\\
&+\!\int_0^T(\si \|\na\! u \|_{L^4}^4\!\!+\!\si\|H_t\|_{H^1}^2\!\!+\!\si\|\na\!\te \|_{H^1}^2\!\!+\!\si\|u_t\|_{L^2}^2\!\!+\!\si^2\|\te_t\|^2_{H^1}\!\!+\!\|\n\!\! -\!1\|_{L^2}^2)dt\le C.
\ea\ee
\end{lemma}
\begin{proof}
First, applying the operator $\pa_t+\div(u\cdot) $ to \eqref{CMHD}$_3 $ and using  \eqref{CMHD}$_1$, one   gets
\be\la{3.96}\ba
&\frac{R}{\ga-1} \n \left(\pa_t\dot \te+u\cdot\na\dot \te\right)\\
&=\ka \Delta  \te_t +\ka \div (\Delta \te u)+\left( \lambda (\div u)^2+2\mu |\mathfrak{D}(u)|^2+\nu|\curl H|^2\right)\div u \\
&\quad -R\n \dot\te \div u-R\n \te\div \dot u+R\n \te  \pa_ku^l\pa_lu^k +2\lambda \left( \div\dot u-\pa_ku^l\pa_lu^k\right)\div u\\
&\quad + \mu (\pa_iu^j+\pa_ju^i)\left( \pa_i\dot u^j+\pa_j\dot u^i-\pa_iu^k\pa_ku^j-\pa_ju^k\pa_ku^i\right)\\
&\quad +2\nu\curl H\cdot(\curl H_t+u\cdot\nabla\curl H).
\ea\ee
Multiplying (\ref{3.96}) by $\dot \te$ and integrating the resulting equality over $\O$, it holds that
\be\la{3.99}\ba
& \frac{R}{2(\ga-1)}\left(\int \n |\dot\te|^2dx\right)_t + \ka   \|\na\dot\te\|_{L^2}^2 \\
&\le  C  \int|\na \dot \te|\left( |\na^2\te||u|+ |\na \te| |\na u|\right)dx+C\int  \n|R\te-\bp| |\na\dot u| |\dot \te|dx\\
&\quad +C(\on)  \int|\na u|^2|\dot\te|\left(|\na u|+|R\te-\bp| \right)dx+C   \int |\na\dot u|\n|\dot \te| dx \\
&\quad +C   \int\left( |\na u|^2|\dot \te|+\n  |\dot
\te|^2|\na u|+|\na u| |\na\dot u| |\dot \te|\right)dx \\
&\quad +C   \int\left( |\na H|^2||\na u|\dot \te|+|\na H||\na H_t||\dot \te|+|u| |\na H||\na^2H| |\dot \te|\right)dx \\
&\le C\|\na u\|^{1/2}_{L^2}\|\na u\|^{1/2}_{L^6}\|\na^2\te\|_{L^2}\|\na \dot \te\|_{L^2}+C(\on)\|\na\te\|_{L^2} \|\na\dot u\|_{L^2} \|\dot\te\|_{L^6}\\
&\quad+C(\on)  \|\na u\|_{L^2}\|\na u\|_{L^6}\left(\|\na u\|_{L^6}+\|\na \te\|_{L^2}\right)
\|\dot\te\|_{L^6} +C  \|\na\dot u\|_{L^2} \|\n\dot\te\|_{L^2} \\
&\quad+C\|\na u\|^{1/2}_{L^6}\|\na u\|^{1/2}_{L^2} \|\dot\te\|_{L^6}\left(\|\na u\|_{L^2}
+\|\n\dot\te\|_{L^2}+\|\na\dot u\|_{L^2}\right)  \\
&\quad+C\left(\|\na\! u\|_{L^6}\|\na\! H\|_{L^2} \|\nabla\! H\|_{L^6}\!\!
+\!\|u\|_{L^6}\|\na\! H\|_{L^2}\|\curl^2\! H\|_{L^2}\right)\|\dot\te\|_{L^6}\\
&\quad+C\|\na\! H\|_{L^2}\|\na\! H_t\|_{L^2}\|\dot\te\|_{L^6}\\
&\le\frac{\ka}{2}\|\na\dot\te\|_{L^2}^2\!\!+\!C(\|\na u\|_{L^2}^2+\|\na H\|_{L^2}^2)\left(\|\na u\|_{L^6}^4\!\!+\!\|\curl^2H\|_{L^2}^4\!\!+\!\|\na\te \|_{L^2}^4\right)\\
&\quad+C\left(1\!+\!\|\na u\|_{L^6}\!+\!\|\na\te\|_{L^2}^2\right) \left(\|\na^2\te\|_{L^2}^2\!+\!\|\na\dot u\|_{L^2}^2\!+\!\|\rho^{1/2} \dot \te\|_{L^2}^2\!+\!\|\curl^2H\|_{L^2}^2\right)\\
&\quad +\!C\|\na\! u\|_{L^6}\|\na\! u\|_{L^2}^2\!\!+\!C\|\na\! u\|_{L^6}^2\|\na\! H\|_{L^2}^2\!\!+\!C(\|\na\! H\|_{L^2}^2\!\!+\!\|\curl^2 H\|_{L^2}^2)\|\nabla\! H_t\|_{L^2}^2,
\ea\ee
where we have used \eqref{g1}, \eqref{g2}, \eqref{key1}, \eqref{p-lp1}, and the following Poincar\'e-type inequality (\cite[Lemma 3.2]{Feireisl2004} ):
\be \la{kk}
\|f\|_{L^p}\le C(\on)(\|\n^{1/2}f\|_{L^2}+\|\na f \|_{L^2}),~~~p\in[2,6],
\ee
for any $f\in\{h\in H^1 \left|\n^{1/2}h\in L^2\}\right.$.
Multiplying (\ref{3.99}) by $\si^2$ and integrating the resulting inequality over $(0,T),$
we obtain after integrating by parts that
\begin{align*}
& \sup_{0\le t\le T}\si^2\int \n|\dot\te|^2dx + \int_0^T\si^2 \|\na\dot\te\|_{L^2}^2dt  \\
&\le C(\on) \sup_{0\le t\le T} \left(\si^2(\|\na u\|_{L^6}^4\!\!+\!\|\curl^2H\|_{L^2}^4+\|\na\te\|_{L^2}^4)\right)\int_0^T(\|\na u\|_{L^2}^2+\|\na H\|_{L^2}^2)dt \\
&\quad + C(\on,M) \sup_{0\le t\le T} \left(\si\left(1+\|\na u\|_{L^6}+\|\na\te\|_{L^2}^2\right)\right)\\
&\quad \quad \quad \quad \quad \quad \quad \cdot\int_0^T\si\left(\|\na^2\te\|_{L^2}^2+\|\na\dot u\|_{L^2}^2+\|\rho^{1/2} \dot \te\|_{L^2}^2+\|\curl^2H\|_{L^2}^2\right)dt\\
&\quad +C(\on,M) \sup_{0\le t\le T} \left(\si\|\na u\|_{L^6} \right) \int_0^T\|\na u\|_{L^2}^2dt\\
&\quad +C(\on,M) \sup_{0\le t\le T} \left(\si\|\na u\|_{L^6}^2 \right) \int_0^T\|\na H\|_{L^2}^2dt\\
&\quad +C(\on,M) \sup_{0\le t\le T} \left(\si(\|\na\! H\|_{L^2}^2\!\!+\!\|\curl^2 H\|_{L^2}^2)\right) \int_0^T\sigma\|\nabla\! H_t\|_{L^2}^2dt\\
&\le  C(\on,M),
\end{align*}
where we have used   (\ref{key1}), (\ref{ss-5}),  (\ref{ss-6}), and the following fact:
\be\ba\la{ong}\sup_{0\le t\le T}(\si\|\na u\|_{L^6}^2+\si\|H\|_{H^2}^2)\le C(\on,M)\ea\ee
due to \eqref{u-l6}, \eqref{h2xd1}, \eqref{ss-5}, and \eqref{key1}.
Next, it follows from (\ref{key1} ),  \eqref{ss-5} , (\ref{td2-th} ), (\ref{b44-3}), (\ref{b44-7}), (\ref{ae3.7}), (\ref{a2-3}), \eqref{ss2-5}, and  (\ref{ss-6}) that
\be \la{vu02}\ba
\sup_{0\le t\le T}\left(\si^2\|\te\|^2_{H^2}\right)\!\!+\!\!\int_0^T \!\!\!\left(\si(\|\na u  \|_{L^4}^4\!\!+\!\|H_t\|_{H^1}^2\!\!+\!\|\na\te \|_{H^1}^2)\!\!+\!\|\n\! -\!1\|_{L^2}^2\right)dt
 \le C(\on,M), \ea\ee
which along with (\ref{key1} ), \eqref{ss-5} , (\ref{ss-6}), \eqref{kk}, \eqref{ong}, and \eqref{ae3.7} gives
 \be\la{vu12}\ba    \int_0^T  \si \|u _t\|_{L^2}^2dt
 &\le C\int_0^T  \si(\| \dot u \|_{L^2}^2+\|u\cdot\na  u \|_{L^2}^2)dt\\
 &\le C(\hat\n)\int_0^T  \si(\|\n^{1/2} \dot u \|_{L^2}^2\!\!+\!\|\na\!\dot u\|_{L^2}^2+\|u \|_{L^\infty}^2\|\na  u \|_{L^2}^2)dt\\
 &\le C(\hat \n,M) ,\ea\ee
\be\la{vu11}\ba
\int_0^T  \si^2 \|  \te _t\|_{L^2}^2dt
&\le C\int_0^T  \si^2(\| \dot \te \|_{L^2}^2+\|u \cdot\na  \te \|_{L^2}^2)dt\\
&\le C(\hat\n)\int_0^T  \si^2(\|\n^{1/2} \dot \te\|_{L^2}^2+\|\na\dot\te\|_{L^2}^2+\|u\|_{L^6}^2\|\na\te\|_{L^3}^2)dt\\
&\le C(\hat \n,M) ,\ea\ee
 and
\be\la{vu01}\ba
\int_0^T  \si^2 \|  \na\te _t\|_{L^2}^2dt
&\le C\int_0^T  \si^2\|\na \dot \te \|_{L^2}^2  dt+ C\int_0^T  \si^2\|\na(u \cdot\na  \te )\|_{L^2}^2dt\\
&\le C(\on,M) +C\int_0^T\si^2\left(\|\na u \|_{L^3}^2+\|u \|_{L^\infty}^2\right)\|\na^2 \te \|_{L^2}^2dt  \\ &\le C(\on,M).\ea\ee
Hence, (\ref{vu15}) is derived from (\ref{ong})--\eqref{vu01} immediately.
The proof of Lemma \ref{lem-th} is finished.
\end{proof}

Finally, we end this section by establishing the exponential decay-in-time for the classical solutions.
\begin{lemma}\label{lem-lim}
Let $(\rho,u,\theta,H)$ be a smooth solution of
 \eqref{CMHD}-\eqref{boundary} satisfying \eqref{key1}.
  Then there exist positive constants $C^\ast$, $\alpha$,  $\theta_\infty$, and $\ve_0$ depending only on $\mu$, $\lambda$, $\nu$, $\kappa$, $\ga$, $R$, $\hat{\rho}$, $\hat{\theta}$, $\Omega$, and $M$ such that for any $t\geq 1$,
\begin{align}\label{lim-0}
\|\rho-1\|_{L^2}+\|u\|_{W^{1,6}}^2+\|H\|_{H^2}^2+\|\theta-\theta_\infty\|_{H^2}^2\le C^\ast e^{-\al t},
\end{align}
provided $C_0\leq\varepsilon_0$.
\end{lemma}
\begin{proof}
First, it follows from \eqref{a2-1}, \eqref{key1}, \eqref{p-lp1}, \eqref{p-b}, \eqref{pt2}, \eqref{grho}, \eqref{basic0}, and \eqref{u-l6} that for any $t\geq 1,$
\begin{equation}\label{lim-1}
\begin{aligned}
 W'(t)+ D(t)
&\le C|\bp_t|(\|R\te-\bp\|_{L^2}+\|R\te-\bp\|_{L^2}^2)+|\bp_t| \|\n-1\|_{L^2}^2 \\
&\quad+C\|R\te-\bp\|_{L^\infty}(\|R\te-\bp\|_{L^2}^2+\|\na u\|_{L^2}^2+\|\na H\|_{L^2}^2)\\
&\le CC_0^{1/24}(\|\na u\|_{L^2}^2+\|\na H\|_{L^2}^2+\|\na\te\|_{L^2}^2+ \|\n-1\|_{L^2}^2),
\end{aligned}
\end{equation}
where $W(t),D(t)$ are defined in \eqref{ww} and \eqref{dd}.
Combining \eqref{lim-1} with \eqref{tdu1}, \eqref{tdh1} and \eqref{p-b} yields that
\begin{equation}\label{lim-2}
\begin{aligned}
& W'(t)+\hat C_1(\|\na u\|_{L^2}^2+\|\na H\|_{L^2}^2+\|\na \te\|_{L^2}^2)
\\&\le \hat C_2C_0^{1/24}(\|\na u\|_{L^2}^2+\|\na H\|_{L^2}^2+\|\na\te\|_{L^2}^2+ \|\n-1\|_{L^2}^2).
\end{aligned}
\end{equation}
Next, rewriting $\eqref{CMHD}_2$ as
\begin{align*}
&(\n u)_t+\div(\n u\otimes u)-\mu\Delta u-(\mu+\lambda)\na(\div u)\\&=\div(H\otimes H)-\frac12\nabla(|H|^2-\overline{|H|^2})-\na(\n(R\te-\bp))-\bp\na(\n-1),
\end{align*}
multiplying this by $\mathcal{B}[\n-1]$ and using Lemma \ref{lem-divf}, \eqref{key1}, and \eqref{p-lp1},  one gets that for any $t\geq 1$,
\be\notag \ba
&\bp\int(\n-1)^2 dx \\
&= \left(\int\rho u\cdot\mathcal{B}[\n-1] dx\right)_t -\int\rho u\cdot\mathcal{B}[\n_t]dx
  -\int\rho u\cdot\nabla\mathcal{B}[\n-1]\cdot u dx \\
& \quad  +\mu\int\p_j u\cdot\p_j\mathcal{B}[\n-1] dx +(\mu+\lambda)\int(\rho-1)\div udx -\int\n(R\te-\bp)(\n-1)dx\\
& \quad  +C\int H\cdot\nabla\mathcal{B}[\n-1]\cdot Hdx-\frac12\int(\rho-1)(|H|^2-\overline{|H|^2})dx \\
& \le\left(\int\n u\cdot\mathcal{B}[\n-1]dx\right)_t+C\|\n u\|_{L^2}^2+C\|u\|_{L^{4}}^{2}\|\n-1\|_{L^2}\\
& \quad  +C\|\rho-1\|_{L^2}\|\na u\|_{L^2}+C\|\rho-1\|_{L^2}\|R\te-\bp\|_{L^2}+C\|H\|_{L^{4}}^{2}\|\n-1\|_{L^2} \\
& \leq \left(\int\rho u\cdot\mathcal{B}[\n-1] dx\right)_t+\frac{P_1}{2}\|\n-1\|_{L^2}^2+C(\|\na u\|_{L^2}^2+\|\na H\|_{L^2}^2+\|\na \te\|_{L^2}^2),
\ea\ee
which as well as \eqref{p-b} leads to
\be\la{lim-3}\|\rho-1\|_{L^2}^2\le\frac{2}{P_1}\left(\int\rho u\cdot\mathcal{B}[\n-1] dx\right)_t+\hat C_3(\|\na u\|_{L^2}^2+\|\na \te\|_{L^2}^2).\ee
By virtue of \eqref{grho}, \eqref{p-b}, and Lemma \ref{lem-divf}, it holds
\be \la{lim-4} \ba
\left|\int\rho u\cdot\mathcal{B}[\n-1] dx\right|
&\leq  C \left(\|\n u\|^2_{L^2}\!+\!\|\n\!-\!1\|_{L^2}^2\right)
\le  C\left(\frac12\|\n^{1/2} u\|^2_{L^2}\!+\!\bp G(\n)\right)
.
\ea\ee
Adding \eqref{lim-2} to \eqref{lim-3} multiplied by $\hat C_5$ with $\hat C_5=\min\{\frac{\pi_1}{4\hat C_4},\frac{\hat C_1}{4\hat C_3}\}$ yields
\be\ba\la{lim-5}
&W_1'(t)+\frac{3\hat C_1}{4}(\|\na u\|_{L^2}^2+\|\na H\|_{L^2}^2+\|\na \te\|_{L^2}^2)+\hat C_5\|\n-1\|_{L^2}^2\\
&\le \hat C_2C_0^{1/24}(\|\na u\|_{L^2}^2+\|\na H\|_{L^2}^2+\|\na\te\|_{L^2}^2+ \|\n-1\|_{L^2}^2),\ea\ee
where
$$W_1(t)\triangleq W(t)-\frac{2\hat C_5}{\pi_1}\int\rho u\cdot\mathcal{B}[\n-1] dx,$$
satisfies
\begin{equation}\label{w1}
\displaystyle  W_1(t) \sim \|\sqrt{\rho} u\|_{L^2}^2+\|\rho-1\|_{L^2}^2+\|H\|_{L^2}^2+\|\sqrt{\rho}(R\te-\overline{P}) \|_{L^2}^2
\end{equation}
due to \eqref{lim-4}, \eqref{grho} and \eqref{p-b}.
Thus we infer from \eqref{lim-5} that
\be\ba\la{lim-6}
&W_1'(t)+\frac{\hat C_1}{2}(\|\na u\|_{L^2}^2+\|\na H\|_{L^2}^2+\|\na \te\|_{L^2}^2)+\frac{\hat C_5}{2}\|\n-1\|_{L^2}^2\leq 0,\ea\ee
provided
$C_0\le\ve_0\triangleq \min\left\{\ve_1,\cdots,\ve_5,\left(\frac{\hat C_5}{2\hat C_2}\right)^{24},\left(\frac{\hat C_1}{4\hat C_2}\right)^{24}\right\}.$
Furthermore, by \eqref{w1}, \eqref{key1}, \eqref{p-lp1}, and \eqref{p-b}, we have
\begin{equation*}
\displaystyle W_1(t)\leq  {\hat C_6}(\|\na u\|_{L^2}^2+\|\na H\|_{L^2}^2+\|\na \te\|_{L^2}^2+\|\n-1\|_{L^2}^2),
\end{equation*}
which together with \eqref{lim-6} implies that for $\alpha=\frac{1}{3}\min\{\frac{\hat C_1}{2\hat C_6},\frac{\hat C_5}{2\hat C_6}\}$,
\be\notag W_1'(t)+3\alpha W_1(t)\le 0.\ee
Therefore, it follows that for any $t\geq 1$,
\be\la{lim-7}\|\sqrt{\rho} u\|_{L^2}^2+\|\rho-1\|_{L^2}^2+\|H\|_{L^2}^2+\|\sqrt{\rho}(R\te-\overline{P}) \|_{L^2}^2\le C W_1(t)\le Ce^{-3\al t}.\ee
Moreover, we deduce from \eqref{lim-6} and \eqref{lim-7} that for any $1\le t\le T<\infty$,
\be\la{lim-8}\int_1^Te^{\al t}(\|\na u\|_{L^2}^2+\|\na \te\|_{L^2}^2)dt\le C.\ee

Next, multiplying \eqref{a1-3} by $e^{\al t}$ and using \eqref{key1} imply that for $B_1$ defined in \eqref{b11},
\be\ba \la{lim-9}
&(e^{\al t}B_1(t))'+ e^{\al t}(\frac{1}{2}\|\sqrt{\rho}\dot u\|_{L^2}^2+\frac{\nu^2}{4}\|\curl^2H\|_{L^2}^2+\frac{1}{2}\|H_t\|_{L^2}^2)\\&\le Ce^{\al t}\left(\|\na u\|_{L^2}^2+\|\na H\|_{L^2}^2+\|\na\te\|_{L^2}^2+\|P-\overline P\|_{L^2}^2\right).
\ea\ee
Note that by \eqref{key1}, \eqref{p-b}, and \eqref{lim-7},
\be\la{lim-10}\|P-\overline P\|_{L^2}\le \|\n(R\te-\bp)\|_{L^2}+\bp\|\n-1\|_{L^2}\le Ce^{-\al t},\ee
which together with  \eqref{lim-8}, \eqref{lim-9}, and Lemma \ref{lem-f-td} gives for any $1\le t\le T<\infty$,
\be\la{lim-11}\ba
\sup_{1\leq t\leq T}\!e^{\al t}\left(\|\na u\|_{L^2}^2\!\!+\!\!\|\na H\|_{L^2}^2\right)\!+\!\int_1^T\!\!\! e^{\al t}(\|\sqrt{\rho}\dot u\|_{L^2}^2\!\!+\!\|\curl^2H\|_{L^2}^2\!\!+\!\|H_t\|_{L^2}^2)dt\!\le C.
\ea \ee
Next, choosing $m=0$ in \eqref{b44}, it follows from \eqref{b44-3}, \eqref{b44-4}-\eqref{b44-6} and \eqref{key1} that for $t\geq 1$,
\begin{equation}\label{b44-lim}
\begin{aligned}
&B_4'(t) + \frac{C_1}{2}\left(\|\nabla\dot{u}\|_{L^2}^2+\|\nabla H_t\|_{L^2}^2\right) +\|\sqrt{\rho}\dot{\theta}\|_{L^2}^2\\
	&\le \!C(\|\sqrt{\rho}\dot{u}\|_{L^2}^{2}\!\!+\!\|\curl^2\!  H\|\ltwo\!+\!\|H_t\|\ltwo\!\!+\!\|\na\! u\|_{L^2}^2\! +\!\|\na\! H\|_{L^2}^2\!\! +\! \|\na\! \te\|_{L^2}^2\!\!+\!\|\rho\!-\!1\|_{L^2}^2).
\end{aligned}
\end{equation}
Similarly, multiplying \eqref{b44-lim} by $e^{\al t}$ and using \eqref{key1}, \eqref{b44-22}, \eqref{b44-9}, \eqref{h2xd1}, \eqref{lim-7}, \eqref{lim-8}, and \eqref{lim-11} imply that for any $t\geq 1$,
\be\la{t15}\ba
&\sup_{1\leq t\leq T}e^{\al t}\left(\|\n^{1/2}\dot u\|_{L^2}^2+\|\curl^2H\|_{L^2}^2+\|H_t\|_{L^2}^2+\|\na\te\|_{L^2}^2\right)\\&+\int_1^T e^{\al t}(\|\na\dot u\|_{L^2}^2+\|\na H_t\|_{L^2}^2+\|\n^{1/2}\dot \te\|^2_{L^2})dt\le C.
\ea \ee

Next, adopting the analogous method and applying \eqref{3.99}, \eqref{vu15}, \eqref{lim-7}, \eqref{lim-8}, \eqref{lim-11}, \eqref{t15}, \eqref{b44-3}, \eqref{b44-5} \eqref{td2-th}, and  \eqref{key1}, we obtain that
\be\la{t17}\ba
\sup_{1\leq t\leq T}\left(e^{\al t}(\|\n^{1/2}\dot\te\|_{L^2}^2+\|\na^2\te\|_{L^2}^2)\right)+\int_1^T e^{\al t}\|\na\dot \te\|_{L^2}^2dt\le C.
\ea \ee
Finally, it remains to determine the limit of $\te$ as $t$ tends to infinity. Combining \eqref{pt}, \eqref{lim-10}, and \eqref{lim-11} shows that for any $t\geq 1$,
\be\ba\notag
|\bp_t|\le C(\|\na u\|_{L^2}^2+\|\na H\|_{L^2}^2+\|P-\bp\|_{L^2}^2)\le C e^{-\al t},
\ea\ee
which implies there exists a constant $P_\infty$ such that
\begin{equation}\label{t18}
\displaystyle  \lim_{t\rightarrow \infty}\overline{P}=P_\infty,\quad \textit{and} \quad |\overline{P}-P_\infty|\le Ce^{-\al t}.
\end{equation}\label{te-inf}
Denote \begin{equation}
\displaystyle  \te_\infty\triangleq P_\infty/R,
\end{equation}
by \eqref{p-lp1}, \eqref{t15}, and \eqref{t18}, we have
\be\la{t19}\|\te-\te_\infty\|_{L^2}^2\le C\|R\te-\bp\|_{L^2}^2+C|\bp-R\te_\infty|^2\le Ce^{-\al t}.\ee
Combining \eqref{tdu2}, \eqref{lim-7}, \eqref{lim-11}, \eqref{t15}, \eqref{t17}, and \eqref{t19} yields \eqref{lim-0} and finishes the proof of Lemma \ref{lem-lim}.
\end{proof}

\section{\label{se4} A priori estimates (II): higher order estimates }
In this section, we derive the time-dependent higher order estimates, which are necessary for the global existence of classical solutions. Here we adopt the method of the article \cite{cl2019,lxz2013,lx2016}, and follow their work with a few modifications. We sketch it here for completeness. Let $(\rho,u,\theta,H)$ be a smooth solution of \eqref{CMHD}-\eqref{boundary} satisfying Proposition \ref{pr1} and the initial energy $C_0\leq \ve_0$, and the positive constant $C $ may depend on $T,$ $\mu$, $\lambda$, $\nu$, $\kappa$, $\ga$, $\on,$ $\hat{\theta},$ $\Omega$, $R,$ $M$,  $\|\na u_0\|_{H^1}, \|\na \te_0\|_{L^2}, \|\na H_0\|_{H^1}, \|\n_0\|_{W^{2,q}},  \| \tilde g\|_{L^2}$ for $q\in(3,6)$, where $\tilde g\in L^2(\Omega)$ is defined as
\be \la{co12}\tilde g\triangleq\n_0^{-1/2}\left(
-\mu \Delta u_0-(\mu+\lambda)\na\div u_0+R\na (\n_0\te_0)-\curl H_0\times H_0\right).\ee

\begin{lemma}\label{lem-x1}
 There exists a positive constant $C,$ such that
\begin{align}\label{x1b1}
&\begin{aligned}
 &\sup_{0\le t\le T}(\|\rho^{\frac{1}{2}}\dot{u}\|_{L^2}^2+\| H_t\|_{L^2}^2+\|H\|_{H^2}^2+\|\te\|_{H^1}^2)\\
&\quad+\int_0^T(\|\nabla\dot{u}\|_{L^2}^{2}+\|\nabla H_t\|_{L^2}^2+\|\rho^{\frac{1}{2}}\dot{\te}\|_{L^2}^2+\| \nabla^2\te\|_{L^2}^2)dt\leq C,
 \end{aligned}\\
& \sup_{0\le t\le T}\sigma(\|\rho^{\frac{1}{2}}\dot{\te}\|_{L^2}^2+\| \nabla^2\te\|_{L^2}^2)+\int_0^T\sigma\|\nabla \dot{\te}\|_{L^2}^{2}dt\leq C,\label{x1b3}\\
&\sup_{0\le t\le T}(\|\rho\|_{H^2}\!\!+\!\|u\|_{H^2})\!\!+\!\int_0^T\!\!(\|\nabla u\|_{L^\infty}^{3/2}\!\!+\!\|u\|_{H^3}^{2}\!\!+\!\|H\|_{H^3}^{2}\!\!+\!\sigma\|\nabla^3\te\|_{L^2}^{2})dt\!\leq\! C.\label{x1b2}
\end{align}
\end{lemma}
\begin{proof}
First, choosing $m=0$ in \eqref{b44}, it follows from \eqref{b44-3}, \eqref{b44-4}-\eqref{b44-6}, \eqref{h2xd1}, and \eqref{key1} that for $t\geq 1$,
\begin{equation}\label{x1-1}
\begin{aligned}
&B_4'(t) + \frac{C_1}{2}\left(\|\nabla\dot{u}\|_{L^2}^2+\|\nabla H_t\|_{L^2}^2\right) +\|\sqrt{\rho}\dot{\theta}\|_{L^2}^2\\
  &\le \!C(\|\sqrt{\rho}\dot{u}\|_{L^2}^{2}\!\!+\!\|\curl^2\!  H\|\ltwo\!+\!\|H_t\|\ltwo\!\!+\!\|\na\! u\|_{L^2}^2\! +\!\|\na\! H\|_{L^2}^2\!\! +\! \|\na\! \te\|_{L^2}^2\!\!+\!\|\rho\!-\!1\|_{L^2}^2)\\
  &\quad+ \!C(\|\sqrt{\rho}\dot{u}\|_{L^2}^{3}\!\!\!+\!\|H_t\|_{L^2}^{3}\!\! +\! \|\na\! \te\|_{L^2}^3)\!\!+\!\|\na\! \te\|_{L^2}^2(\|\sqrt{\rho}\dot{u}\|_{L^2}^{2}\!\!+\!\|H_t\|\ltwo\!\!+\!\|\na\! \te\|_{L^2}^2),
\end{aligned}
\end{equation}
Denote that $$\varphi(t)\triangleq \|\rho^{\frac{1}{2}}\dot{u}\|_{L^2}^2+\| H_t\|_{L^2}^2+\|\nabla\theta\|_{L^2}^2,$$
and
\begin{equation}\label{hh00}
H_t(\cdot,0) \triangleq -\nu \curl^2H_0-u_0\cdot\nabla H_0+H_0\cdot\nabla u_0-H_0\div u_0,
\end{equation}
which along with \eqref{dt2}, \eqref{b44-22} and \eqref{co12} yields that
\be \la{wq02}
|\varphi(0)|\le C\| \tilde g\|_{L^2}^2+C\le C.
\ee
Then integrating \eqref{x1-1} over $(0,T)$, using \eqref{key1}, \eqref{bz5}, \eqref{b44-22}, \eqref{b44-9} and \eqref{wq02} implies that
\begin{align}
\displaystyle  &\varphi(t)+\int_0^t(\|\nabla\dot{u}\|_{L^2}^{2}+\|\nabla H_t\|_{L^2}^2+\|\sqrt{\rho} \dot{\theta}\|_{L^2}^2)dt \nonumber\\
& \leq -\left(\int_{\p\O} F\left( u\cdot \na n \cdot u \right)dS\right)_t+C\int_0^t \left(\|\n^{1/2}\dot u\|_{L^2}^2+\|H_t\|_{L^2}^2+ \|\na \te \|_{L^2}^2\right)\varphi ds+C\nonumber\\
&\le  C (\|\na u\|_{L^2}^2\|\nabla F\|_{L^2})(t)+C\int_0^t \left(\|\n^{1/2}\dot u\|_{L^2}^2+\|H_t\|_{L^2}^2 +\|\na \te \|_{L^2}^2\right)\varphi ds+C\nonumber\\
&\le  \frac{1}{2} \varphi(t) + C \int_{0}^{t} \left(\|\n^{1/2}\dot u\|_{L^2}^2+\|H_t\|_{L^2}^2+ \|\na \te \|_{L^2}^2\right)\varphi ds+C,\nonumber
\end{align}
and applying Gr\"{o}nwall's inequality to this, it holds
\begin{align}\label{x1-2}
\displaystyle  &\sup_{0\le t\le T} \varphi(t)+\int_0^t(\|\nabla\dot{u}\|_{L^2}^{2}+\|\nabla H_t\|_{L^2}^2+\|\sqrt{\rho} \dot{\theta}\|_{L^2}^2)dt\leq C,
\end{align}
which together with \eqref{2tdh}, \eqref{tdh-4}, \eqref{h2xd1}, \eqref{p-lp1}, \eqref{p-b} and  \eqref{td2-th} yields
\be\la{ff1} \ba
\|H\|_{H^2}+\|\te\|_{L^2} &\le C+C( \|R \te-\overline P \|_{L^2} +\overline P ) \le C(\|\na \te\|_{L^2}+1) \le C.
\ea \ee
Furthermore, it follows from \eqref{td2-th}, \eqref{b44-3}, \eqref{b44-5}, \eqref{b44-6}, and \eqref{x1-2} that
\be \notag \int_{0}^{T}\|\na^2\te\|_{L^2}^2 dt \le C,\ee
which together with \eqref{x1-2} and \eqref{ff1} yields \eqref{x1b1}.

Next, multiplying (\ref{3.99}) by $\sigma$ and integrating over $(0,T)$  lead to
\be  \ba \la{a5}
&\sup\limits_{0\le t\le T} \si \int \n|\dot\te|^2dx+\int_0^T \si \|\na\dot\te\|_{L^2}^2dt\\
&\le
C\int_0^T\left(  \|\na^2\te\|_{L^2}^2+\|\na\dot u\|_{L^2}^2+\|\na H_t\|_{L^2}^2+ \|\n^{1/2}\dot\te\|_{L^2}^2 \right)dt + C\\
&\le  C,
\ea\ee
where we have used (\ref{x1-2}),  (\ref{key1}), (\ref{b44-3}), \eqref{b44-5},   (\ref{td2-th}), and the following fact:
\be\la{u-l6-2}\|\na u\|_{L^6}\le C\ee
due to \eqref{u-l6}, \eqref{key1}, and \eqref{x1-2}.
Furthermore, it follows from \eqref{td2-th}, \eqref{b44-3}, \eqref{b44-5}, \eqref{b44-6}, and \eqref{key1} that
\be \notag \sup\limits_{0\le t\le T} \si\|\na^2\te\|_{L^2}^2 \le C,\ee
which together with \eqref{a5} yields \eqref{x1b3}.

Finally, it remains to prove \eqref{x1b2}.
Standard calculations show  that for $ 2\le p\le 6$,
\be\la{L11}\ba
\partial_t\norm[L^p]{\nabla\rho}
&\le C\norm[L^{\infty}]{\nabla u} \norm[L^p]{\nabla\rho}+C\|\na^2u\|_{L^p}\\
&\le C\left(1\!+\!\|\na u\|_{L^{\infty}}\!+\!\|\na^2\te \|_{L^2}\right)
\norm[L^p]{\nabla\rho}\!+\!C\left(1\!+\!\|\na\dot u\|_{L^2}\!+\!\|\na^2\te \|_{L^2}\right), \ea\ee
where we have used
\be\ba\la{ua1}
\|\na^2 u\|_{L^p} 
& \le   C\left(1+\|\na\dot u\|_{L^2}+\|\na \te\|_{L^p}+\|\te\|_{L^\infty}\|\nabla\n\|_{L^p}\right)\\
&\le  C\left(1+\|\na\dot u\|_{L^2}+\|\na^2\te\|_{L^2}+(\|\na^2\te \|_{L^2} + 1)\|\nabla\n\|_{L^p}\right)
\ea\ee
due to \eqref{g1}, \eqref{udot}, \eqref{2tdu},  \eqref{key1},  and  \eqref{x1b1}.
It follows from the Beale-Kato-Majda-type inequality (see \cite[Lemma 2.7]{cl2019}), \eqref{key1}, and (\ref{ua1})  that
\be\la{u13}\ba
\|\na u\|_{L^\infty }
&\le C\left( \|{\rm div}u\|_{L^\infty } + \|\curl u\|_{L^\infty }
\right)\log(e+ \|\na\dot u\|_{L^2 } + \|\na^2\te \|_{L^2})\\
&\quad +C\left(\|{\rm div}u\|_{L^\infty }+ \|\curl u\|_{L^\infty } \right)
\log\left(e  + \|\na \n\|_{L^6}\right)+C.
\ea\ee
Denote
\bnn\la{gt}\begin{cases}
  f(t)\triangleq  e+\|\na \n\|_{L^6},\\
  h(t)\triangleq 1+  \|{\rm div}u\|_{L^\infty }^2+ \|\curl u\|_{L^\infty }^2
  + \|\na\dot u\|_{L^2 }^2 +\|\na^2\te \|_{L^2}^2.
\end{cases}\enn
One obtains after submitting \eqref{u13} into (\ref{L11}) with  $p=6$ that
\bnn f'(t)\le   C h(t) f(t)\ln f(t) ,\enn
which implies \be\la{w2}  (\ln(\ln f(t)))'\le  C h(t).\ee
Note that by virtue of \eqref{flux}, (\ref{p-b}), (\ref{x1b1}),  (\ref{key1} ), \eqref{bz5}, and (\ref{bz6}),  one gets
\be \la{p2}\ba
& \int_0^T\left(\|\div u\|^2_{L^\infty}+\|\curl u\|^2_{L^\infty} \right)dt \\
& \le  C\int_0^T\left(\|F\|^2_{L^\infty}+ \|\curl u\|^2_{L^\infty}+\|P-\bp\|^2_{L^\infty}+\||H|^2-\overline{|H|^2}\|^2_{L^\infty}\right)dt \\
& \le   C\ia\left(\|\na F\|^2_{L^6}+\|\na\curl  u\|^2_{L^6}+\|\na^2\te \|_{L^2}^2\right)dt+C \\
&\le C\ia(\|\na \dot u\|^2_{L^2}+\|\na^2\te \|_{L^2}^2)dt+C \\
&\le  C,
\ea\ee
which as well as \eqref{w2} and \eqref{x1b1}  yields that
\be \la{u113} \sup\limits_{0\le t\le T}\|\nabla \rho\|_{L^6}\le C.\ee
Combining this with (\ref{u13}),  \eqref{p2}, and (\ref{x1b1}) leads to
\be \la{v6}\ia\|\nabla u\|_{L^\infty}^{3/2}dt \le C.\ee
Moreover, it follows from (\ref{2tdu}), \eqref{basic0}, \eqref{key1}, \eqref{u113}, and (\ref{x1b1})  that
\be\ba\la{aa95}
\sup\limits_{0\le t\le T} \| u\|_{H^2}
\le &C \sup\limits_{0\le t\le T}\left(1+\|\nabla P\|_{L^2}\right)\le C.
\ea\ee
Moreover, by \eqref{CMHD}$_4$, \eqref{3tdh}, \eqref{x1b1} and \eqref{aa95}, one has
\begin{equation}\label{3tdh-1}
\displaystyle  \|H\|_{H^3}\leq C(\|\nabla H_t\|_{L^2}+\| u\|_{H^2}\| H\|_{H^2})\leq C(\|\nabla H_t\|_{L^2}+ 1).
\end{equation}
Next, applying operator $\p_{ij}~(1\le i,j\le 3)$ to $(\ref{CMHD})_1$ gives
\be\la{4.52}  (\p_{ij} \n)_t+\div (\p_{ij} \n u)+\div (\n\p_{ij} u)+\div(\p_i\n\pa_j u+\p_j\n\p_i u)=0. \ee
Multiplying (\ref{4.52}) by $2\p_{ij} \n$ and  integrating the resulting equality over $\O,$ it holds
\be\la{ua2}\ba
\frac{d}{dt}\|\na^2\n\|^2_{L^2}
& \le C(1+\|\na u\|_{L^{\infty}})\|\na^2\n\|_{L^2}^2+C\|\na u\|^2_{H^2}\\
& \le C(1+\|\na u\|_{L^{\infty}} +\|\na^2\te \|_{L^2}^2)(1+\|\na^2\n\|_{L^2}^2) +C\|\na\dot u \|^2_{L^2}, \ea\ee
where one has used \eqref{key1}, \eqref{u113}, and the following estimate:
\be\ba\la{va2}
\| u\|_{H^3}&\le C\left(\|\na(\n\dot u)\|_{L^2}+\| \na^2 P\|_{L^2}+1\right)
\\ &\le  C\|\na\dot u \|_{L^2}+ C  (1+\|\na^2\te \|_{L^2})(1+\|\na^2\n\|_{L^2}),
\ea\ee
due to (\ref{3tdu}), (\ref{udot}), (\ref{x1b1}), (\ref{u113}),  and (\ref{key1} ).
Then applying Gr\"{o}nwall's inequality to \eqref{ua2} and using  (\ref{x1b1}), (\ref{v6}) yield
\be\la{ja3} \sup_{0\le t\le T} \|\na^2\n \|_{L^2}  \le C,\ee
which together with \eqref{va2}, \eqref{3tdh-1} and \eqref{x1b1} gives
\be\la{ja4} \int_0^T(\|u\|_{H^3}^2+\|H\|_{H^3}^2 )dt\le C . \ee
Furthermore, applying the standard $H^1$-estimate to elliptic problem (\ref{CMHD-th}), one derives from \eqref{key1} ,  \eqref{x1b1}, \eqref{u113}, \eqref{kk}, and \eqref{aa95} that
\be\la{ex4}\ba
\|\na^2\te\|_{H^1}
& \le C\left(\|\n \dot \te\|_{H^1}+\|\n\te\div u\|_{H^1}+\||\na u|^2\|_{H^1}+\||\na H|^2\|_{H^1}\right)\\
& \le C\left(1\!+\! \|\na \dot \te\|_{L^2} \!+ \! \|\rho^{1/2} \dot \theta\|_{L^2} \! +\! \|\na(\n\te\div u)\|_{L^2}\!+\! \|u\|_{H^3}\!+\!\|H\|_{H^3} \right) \\
& \le C\left(1+ \|\na \dot \te\|_{L^2} +  \|\rho^{1/2} \dot \theta\|_{L^2}  +\|\na^2 \theta\|_{L^2}+ \|u\|_{H^3}+\|H\|_{H^3} \right),
\ea\ee
which along with \eqref{key1} , \eqref{ja3}, \eqref{ja4}, \eqref{u113}--\eqref{aa95}, and \eqref{x1b1} yields (\ref{x1b2}).
This finishes the proof.

\end{proof}

\begin{lemma}\label{lem-x3}
There exists a positive constant $C$ such that
\begin{align}
& \sup_{0\le t\le T}\|\rho_t\|_{H^1}^2 + \int\left(\|u_t\|_{H^1}+\|\sqrt{\rho}u_t\|_{H^1}^2+\sigma(\|\te_{t}\|_{H^1}^2\!+\!\|\rho\te_{t}\|_{H^1}^2)\right)dt\le C, \label{x3b1}\\
& \begin{aligned}\label{x3b2}
&\sup\limits_{0\le t\le T}\sigma (\|\nabla u_t\|_{L^2}^2+\|\nabla H_t\|_{L^2}^2+\|\rho_{tt}\|_{L^2}^2+\|u\|_{H^3}^2+\|H\|_{H^3}^2)\\
&\quad  + \int_0^T\sigma(\|\rho^{\frac{1}{2}}u_{tt}\|_{L^2}^2+\|H_{tt}\|_{L^2}^2+\|\nabla H_{t}\|_{H^1}^2+\|\nabla u_{t}\|_{H^1}^2)dt
\le C,
\end{aligned}\\
& \sup_{0\le t\le T}\|\rho\|_{W^{2,q}}+\int_{0}^T\|\nabla ^2u\|_{W^{1,q}}^{p_0}dt\le C,\label{y2}
\end{align}
for any $q\in (3,6)$ and $1<p_0<\frac{4q}{5q-6}$.
\end{lemma}
\begin{proof}
First, one deduces from $(\ref{CMHD})_1$, \eqref{x1b1}, (\ref{x1b2}), and \eqref{hs} that
\begin{equation}
\begin{aligned}\label{sp1}
&\|\n_t\|_{H^1}\le \|\div (\rho u)\|_{H^1}
\le  C \|u\|_{H^2}\|
\n\|_{H^2}
\le C,\\
&\|u_t\|_{H^1}\!\!+\!\|\sqrt{\rho}u_t\|_{H^1}\le \|\dot{u}\|_{H^1}\!\!+\!\|u\cdot\nabla u\|_{H^1}\!\!+\!\|\nabla\rho\|_{L^3}\|u_t\|_{L^6}\le  C \|\nabla\dot{u}\|_{L^2}\!\!+\!\|u\|_{H^3}.
 \end{aligned}
\end{equation}
Similarly, it follows from \eqref{x1b2} and \eqref{kk} that
\begin{equation}
\begin{aligned}\label{va1}
\|\theta_t\|_{H^1}+\|\rho\theta_t\|_{H^1}&\le \|\nabla\dot{\te}\|_{L^2}+\|\sqrt{\rho}\dot{\te}\|_{L^2}+\|u\cdot\nabla \theta\|_{H^1}+\|\nabla\rho\|_{L^3}\|\theta_t\|_{L^6}\\
&\le  C \|\nabla\dot{\theta}\|_{L^2}+\|\sqrt{\rho}\dot{\te}\|_{L^2}+\|\nabla\theta\|_{H^1},
\end{aligned}
\end{equation}
where we have used
\be\la{w3}\|\te_t\|_{L^6}\le C\|\n^{1/2}\te_t\|_{L^2}+C\|\na\te_t\|_{L^2}.
\ee
Combining \eqref{sp1}, \eqref{va1} and  (\ref{x1b2}) gives \eqref{x3b1}.

Next, differentiating $(\ref{CMHD})_2$ with respect to $t$ yields that
\be   \la{va7}\ba (\n u_t)_t
=-(\n u\cdot\na u)_t \!+\! \left(\mu\Delta u+(\mu+\lambda)\na\div u \right)_t\! -\!\na P_t\!+\!\left(H\cdot\nabla H\!-\!\nabla|H|^2/2\right)_t.\ea\ee
It follows from (\ref{x3b1}), \eqref{x1b1}--\eqref{x1b2}, \eqref{td2-th}, and \eqref{kk} that
\be   \la{va9}\ba  &\|(\n u\cdot\na u)_t  \|_{L^{2}}
\le C\|\n_t\|_{L^6} \|\na u\|_{L^3}+  C\|u_t\|_{L^6} \|\na u\|_{L^3}+  C\|u \|_{L^\infty} \|\na u_t\|_{L^2}\\
&\qquad\qquad\qquad\ \le C+   C\|\na u_t\|_{L^2},\\
&\|(H\cdot\nabla H\!-\!\nabla|H|^2/2)_t  \|_{L^{2}}
\le C\|H_t\|_{L^6} \|\na H\|_{L^3}+  C\|H \|_{L^\infty} \|\na H_t\|_{L^2}\\
&\qquad\qquad\qquad\qquad\qquad\quad\le C+   C\|\na H_t\|_{L^2},\\
\ea\ee
and
\be   \la{va10}\ba  \|\na P_t  \|_{L^2}
&=R\|\n_t\na\te +\n \na\te_t +\na\n_t\te +\na\n\te_t\|_{L^2}\\
&\le C\left(\|\n_t\|_{L^6}\|\na\te\| _{L^3}+\|\na\te_t\|_{L^2}
+\|\te\|_{L^\infty}\|\na \n_t\|_{L^2}+\|\na\n\|_{L^6}\|\te_t\| _{L^3}\right)\\
&\le C+C\|\na\te_t\|_{L^2}+C\|\n^{1/2}\te_t\|_{L^2}+C\|\nabla^2\te\|_{L^2}.\ea\ee
Combining (\ref{va7})--(\ref{va10}) with   (\ref{x3b1}) shows
\be\la{vva04}\ba\int_0^T \si \|(\n u_t)_t\|_{H^{-1}}^2dt\le C.\ea\ee
Similarly, we have \bnn\int_0^T \si  \|(\n \te_t)_t\|_{H^{-1}}^2dt\le C,  \enn
which combined with (\ref{vva04}) implies
\be\la{vva5}\ba
  \int_0^T \sigma \left( \|(\n u_t)_t\|_{H^{-1}}^2+\|(\n \te_t)_t\|_{H^{-1}}^2
  \right)dt\le C.\ea\ee

Next, we need to prove \eqref{x3b2}. Introducing the function
$$\widetilde{D}(t)=(\lambda+2\mu)\int(\div u_t)^{2}dx+\mu\int|\omega_t|^{2}dx +\nu \int|\curl H_t|^{2}dx.$$
Since $u_t\cdot n = 0, H_t \cdot n=0$ on $\partial\Omega$, by Lemma \ref{lem-vn}, we have
\begin{align}\label{x4b6}
\displaystyle  \|\nabla u_t\|_{L^2}^2+\|\nabla H_t\|_{L^2}^2\leq C(\Omega)\widetilde{D}(t).
\end{align}
Differentiating  $\eqref{CMHD}_{2,4}$  with respect to $t,$
\begin{align}\label{utt}\rho u_{tt}\!-\!(\lambda\!+\!2 \mu)\nabla \div u_t\!+\!\mu \nabla \!\times\! \omega_t\!=-\nabla\! P_t\!-\!\rho_t u_t\!-\!(\rho u\! \cdot\! \nabla\! u)_t\!+\!(H\! \cdot\! \nabla\! H\!-\!\nabla |H|^2/2)_t, \end{align}
and
\begin{align}\label{htt} \quad H_{tt}-\nu \nabla \times \curl H_t =(H \cdot \nabla u-u \cdot \nabla H-H\div u)_t, \end{align}
then multiplying \eqref{utt} by $2u_{tt}$, multiplying \eqref{htt} $2H_{tt}$ respectively, we obtain
\begin{align}\label{x4b7}
\displaystyle &\quad\frac{d}{dt}\widetilde{D}(t)+2\int(\rho|u_{tt}|^2+|H_{tt}|^2)dx \nonumber \\
&=\frac{d}{dt}\Big(-\int\rho_t|u_t|^{2}dx-2\int\rho_tu\cdot\nabla u\cdot u_tdx+2\int P_t\div u_tdx \nonumber \\
&\qquad\quad -\int (2(H \otimes H)_t : \nabla u_t-|H|^2_t\div u_t dx)\Big)\nonumber \\
&\quad +\int\rho_{tt}|u_t|^{2}dx + 2\int(\rho_tu\cdot\nabla u)_t\cdot u_tdx-2\int\rho (u\cdot\nabla u)_t\cdot u_{tt}dx \nonumber\\
&\quad - 2\int P_{tt}\div u_tdx+\int (2(H \otimes H)_{tt} : \nabla u_t-|H|^2_{tt}\div u_t) dx\nonumber\\
&\quad +2\int (H \cdot \nabla u-u \cdot \nabla H- H \div u)_t \cdot H_{tt} dx\nonumber\\
&\triangleq\frac{d}{dt}D_0 + \sum\limits_{i=1}^6 D_i .
\end{align}
Let us estimate $D_i$, $i=0,1,\cdots, 6.$
We conclude from $\eqref{CMHD}_1$, \eqref{udot}, \eqref{x1b2}, \eqref{x1b3}, \eqref{x4b6} and Sobolev's, Poincar\'{e}'s  inequalities that
\begin{align}
&\begin{aligned}\label{x4k0}
D_0 & \le \left|\int{\rm div}(\rho u)\,|u_t|^2dx\right|+C\norm[L^3]{\rho_t}\| u\|_{L^\infty}\|\nabla u\|_{L^2}\norm[L^6]{u_t}\\&\quad+C\|P_t\|_{L^2}\|\nabla u_t\|_{L^2} + C \|H\|_{L^\infty}\|H_t\|_{L^2}\|\nabla u_t\|_{L^2}  \\
&\le C(1+\|\n^{1/2}\te_t\|_{L^2})\|\nabla u_t\|_{L^2},
\end{aligned}\\
&D_1 \leq \left|\int_{ }\rho_{tt}\, |u_t|^2 dx\right|
\le C\|\nabla u_t\|_{L^2}^4 +C\norm[L^2]{\rho_{tt}}^2,\notag
\\
&D_2+D_3 
 \le C\norm[L^2]{\rho_{tt}}^2 + C\norm[L^2]{\nabla u_t}^2+\delta\norm[L^2]{\rho^{{\frac12}}u_{tt}}^2 ,\notag
 \\
&D_5 
 \leq \delta\|H_{tt}\|_{L^2}^2+C\|H_{t}\|_{L^2}^2\widetilde{D}(t)+C(\|\nabla H_t\|_{L^2}^2+\|\nabla u_t\|_{L^2}^2),\notag
 \\
&D_6 
 \leq \delta\|H_{tt}\|_{L^2}^2+C(\|\nabla H_t\|_{L^2}^2+\|\nabla u_t\|_{L^2}^2).\notag
\end{align}
Next, by virtue of \eqref{op3}, \eqref{x3b1}, \eqref{va10}, and Lemma \eqref{lem-x1}, one
\begin{align}
\begin{aligned}\notag
D_4 &\leq C\left| \int(P_{tt}-\kappa(\gamma-1)\Delta\te)\div u_{t} dx\right| +C\left| \int_{ }\nabla\te_t\cdot\nabla\div u_tdx\right|\\
& \le   C\|P_{tt}-\kappa(\gamma-1)\Delta\te\|_{L^2}\|\nabla u_t\|_{L^2}+C\|\nabla^2u_t\|_{L^2}\|\na \te_t\|_{L^2} \\
& \le   C(\|(u\!\cdot\na\! P)_t\|_{L^2}\!\!+\!\!\|(P\div u)_t\|_{L^2}\!\!+\!\!\||\na u||\na u_t|\|_{L^2}\!\!+\!\!\||\na H||\na H_t|\|_{L^2})\|\nabla u_t\|_{L^2}\\
&\quad+C\|\nabla^2u_t\|_{L^2}\|\na \te_t\|_{L^2} \\
&\leq \delta\|\rho^{1/2}u_{tt}\|_{L^2}^2+C\left(1+\|\na u\|_{L^\infty}+\|H\|_{H^3}+\|\na^2\te \|_{L^2}\right)\widetilde{D}(t)\\
&\quad+C\left(1+\|\na\te_t\|^2_{L^2}+\|\na^2\te\|^2_{L^2}+\|\n^{1/2}\te_t\|_{L^2}^2\right),
\end{aligned}
\end{align}
where we used the fact
\be\la{nt4}\ba
\|\na^2u_t\|_{L^2}
&\le C\left(\|\n u_{tt}\|_{L^2}\!+\!\|\na u_t\|_{L^2}\!+\!\|\na H_t\|_{L^2}\!+\!\|\n^{1/2} \te_{t}\|_{L^2}\!+\!\|\na \te_t\|_{L^2}\!+\!1\right),\ea\ee
due to \eqref{va7} and Lemma \ref{lem-lame}.
Putting the estimates of $D_i (i=1,\cdots,6)\ $ in \eqref {x4b7} and choosing $\delta$ suitably small implies
\begin{align}
\begin{aligned}\label{x4b8}
\displaystyle  &\quad\frac{d}{dt}(\widetilde{D}(t)-D_0)+\int(\rho|u_{tt}|^{2}+|H_{tt}|^2)dx  \\
&\le C\left(1+\|\na u\|_{L^\infty}+\|H\|_{H^3}+\|\na u_t \|_{L^2}+\|\na^2\te \|_{L^2}\right)\widetilde{D}(t)\\
&\quad+C(1+\|\nabla \te_t\|_{L^2}^2+\|\rho^{1/2} \te_t\|_{L^2}^2+\|\rho_{tt}\|_{L^2}^2),
\end{aligned}
\end{align}
where we used the fact
\be \la{s4} \ba
\|\n_{tt}\|_{L^2} 
& \le C\left(\|\n_t\|_{L^6}\|\nabla u\|_{L^3}+ \|\nabla u_t\|_{L^2}
+\|u_t\|_{L^6}\|\nabla \n\|_{L^3}+\|\nabla \n_t\|_{L^2}\right) \\ &\le C+C\|\na u_t\|_{L^2}.\ea\ee
Consequently, multiplying \eqref{x4b8} by $\sigma$, together with and Gr\"{o}nwall's inequality, \eqref{x1b2}, \eqref{x3b1}, \eqref{x4b6}, \eqref{x3b1} and \eqref{x4k0}, we derive that
\begin{align}\label{x4b9}
\displaystyle \sup_{0\le t\le T}\sigma(\|\nabla u_t\|_{L^2}^2+\|\nabla H_t\|_{L^2}^2)+\int_0^T\sigma(\|\rho^{\frac{1}{2}}u_{tt}\|_{L^2}^2+\|H_{tt}\|_{L^2}^2)dt\le C.
\end{align}
Moreover, by Lemma \ref{lem-x1}, \eqref{s4}, (\ref{va2}), \eqref{nt4}, \eqref{va1}, and  (\ref{x4b9}), we have
\be\notag
\sup\limits_{0\le t\le T}\si\left(\|\n_{tt}\|_{L^2}^2+\|u\|^2_{H^3}+\|H\|^2_{H^3}\right) + \int_0^T \si(\|\nabla u_t\|_{H^1}^2+\|\nabla H_t\|_{H^1}^2) dt\le C,
\ee
which along with \eqref{x4b9} gives \eqref{x3b2}.

Finally, it remains to prove \eqref{x6b}.
By Sobolev's inequality, \eqref{udot}, \eqref{x1b2}, \eqref{x3b1} and \eqref{x3b2}, we get for any $q\in (3,6)$,
\be\notag
\ba     \|\na(\n\dot u)\|_{L^q}
&\le C\|\na \n\|_{L^6}\|\dot u \|_{L^{6q/(6-q)}}+C\|\na\dot u \|_{L^q}\\
&\le C\|\na u_t \|_{L^2}^{(6-q)/2q}\|\na u_t \|_{L^6}^{3(q-2)/{2q}}\\
& \quad+C\|\na u \|_{L^q}\| \na u \|_{L^\infty}+C\| u \|_{L^\infty}\|\na^2 u \|_{L^q}+C\\
&\le C\si^{-\frac12} \left(\si\|\na u_t \|^2_{H^1}\right)^{3(q-2)/{4q}}
+C\|u\|_{H^3}+C,\ea \ee
and
\begin{equation}\notag
    \begin{aligned}
    \| \na^2 \theta\|_{L^q} \le& C \|\na^2 \theta\|_{L^2}^{(6-q)/2q} \|\na^3 \theta\|_{L^2}^{3(q-2)/2q} +C\|\na^2 \theta\|_{L^2}\\
    \le & C \si^{-\frac12} \left(\si\|\na^3 \theta \|^2_{L^2} \right)^{3(q-2)/{4q}} +C \|\na^2 \theta\|_{L^2},
\end{aligned}
\end{equation}
which combined with Lemma \ref{lem-x1} and \eqref{x3b2} shows that for $1<p_0<\frac{4q}{5q-6}$,
\be \la{4.53}\int_0^T \left(\|\na(\n \dot u)\|^{p_0}_{L^q} + \|\na^2 \theta\|_{L^q}^{p_0} \right) dt\le C. \ee
Next, multiplying (\ref{4.52}) by $q|\p_{ij} \n|^{q-2}\p_{ij} \n$ and  integrating the resulting equality over $\O,$ we obtain that
\be\la{sp28}\ba
\frac{d}{dt}\|\na^2\n\|_{L^q}^q &\le C\|\na u\|_{L^\infty}\|\na^2\n\|_{L^q}^q
+C\|\na^{2} u\|_{W^{1,q}}\|\na^2\n\|_{L^q}^{q-1}(\|\na\n\|_{L^q}+1)\\
&\le C\left(\| u\|_{H^3}  \!\! +\! \|\na\dot u\|_{L^2} \!\!+\!\|\na(\n\dot u)\|_{L^q}\!\!+\!\|\na^2 \theta\|_{L^q}\!\!+\!1\right)\left(\|\na^2 \n\|_{L^q}^q\!\!+\!1\right).
\ea\ee
where in the last inequality we have used the  following fact that
\be\la{a4.74}\ba
\|\na^2u\|_{W^{1,q}}
\le & C\left(\|\n \dot u\|_{W^{1,q}}+\|H\cdot\nabla H\|_{W^{1,q}}+\|\na P\|_{W^{1,q}}+\|\na u\|_{L^2}\right)\\
\le & C\left(\|\na\dot u\|_{L^2}+\|H\|_{H^3}+\|\na(\n\dot u)\|_{L^q}  + \| \na^2 \theta\|_{L^q}+1 \right)\\
&+ C(1+ \|\na^2 \theta\|_{L^2})\| \na^2 \n\|_{L^q}.
\ea\ee
due to \eqref{2tdu}, \eqref{3tdu}, \eqref{x1b2} and \eqref{x3b1}.
Hence, applying  Gr\"{o}nwall's inequality to (\ref{sp28}), we obtain after using Lemma \ref{lem-x1} and (\ref{4.53}) that
\bnn  \sup\limits_{0\le t\le T}\|\na^2 \n\|_{L^q}\le C,\enn
which together with   Lemma \ref{lem-x1},    (\ref{4.53}),
and (\ref{a4.74}) gives \eqref{y2}. The proof of Lemma \ref{lem-x3} is finished.
\end{proof}

\begin{lemma}\label{lem-x6}
There exists a positive constant $C$ such that
\begin{align}
\begin{aligned}\label{x6b}
\displaystyle & \sup_{0\le t\le T}\sigma\left(\|\sqrt{\rho} u_{tt}\|_{L^2}+\|H_{tt}\|_{L^2}+\|\te_t\|_{H^1}\right)\\
& +\sup_{0\le t\le T}\sigma(\|\te_t\|_{H^1}+\|u_t\|_{H^2}+\|H_t\|_{H^2}+\|H\|_{H^4}+\|u\|_{W^{3,q}}+\|\nabla^2\te\|_{H^1})\\
& +\int_{0}^T\sigma^2(\|\nabla u_{tt}\|_{2}^2+\|\nabla H_{tt}\|_{2}^2+\|\n^{1/2}\te_{tt}\|_{L^2}^2)dt\le C,
\end{aligned}
\end{align}

for any $q\in (3,6)$.
\end{lemma}

\begin{proof} Differentiating $\eqref{CMHD}_{2,4}$ with respect to $t$ twice,
multiplying them by $2u_{tt}$ and $2H_{tt}$ respectively, and integrating over $\Omega$ lead to
\begin{align}\begin{aligned}\label{x6b2}
&\quad \frac{d}{dt}\int(\rho |u_{tt}|^2+|H_{tt}|^2)dx  \\
 &\quad +2(\lambda+2\mu)\int(\div u_{tt})^2dx+2\mu\int|\omega_{tt}|^2dx+2\nu\int|\curl H_{tt}|^2dx   \\
&=-8\int_{ }  \n u^i_{tt} u\cdot\na
 u^i_{tt} dx-2\int_{ }(\n u)_t\cdot \left[\na (u_t\cdot u_{tt})+2\na
u_t\cdot u_{tt}\right]dx   \\
&\quad -2\int_{}(\n_{tt}u\!+\!2\n_tu_t)\cdot\na u\cdot u_{tt}dx\!-\!2\int (\n
u_{tt}\cdot\na u\cdot  u_{tt}-P_{tt}{\rm div}u_{tt})dx   \\
&\quad -2\int (H\!\! \cdot\!\! \nabla\! H\!-\!\nabla |H|^2/2)_{tt}u_{tt}dx\!\!+\!2 \int(H\!\! \cdot\!\! \nabla u\!\!-\! u \cdot\!\! \nabla\! H\!\!-\! H \div u)_{tt}H_{tt}dx   \\
&\triangleq\sum_{i=1}^6 R_i.
\end{aligned}
\end{align}
Let us estimate $R_i$ for $i=1,\cdots,6$. H\"{o}lder's inequality and \eqref{x1b2} give
\begin{align}\notag
\displaystyle  R_1 &\le
C\|\sqrt{\rho}u_{tt}\|_{L^2}\|\na u_{tt}\|_{L^2}\| u \|_{L^\infty}
\le \de \|\na u_{tt}\|_{L^2}^2+C(\de)\|\sqrt{\rho}u_{tt}\|^2_{L^2} .
\end{align}
By \eqref{x1b2}, \eqref{x3b1}, \eqref{x3b1} and \eqref{x3b2}, we conclude that
\begin{align}
&\begin{aligned}\notag
R_2&\le \de \|\na u_{tt}\|_{L^2}^2+C(\de)\|\nabla u_t\|_{L^2}^3+C(\de)\|\nabla u_t\|_{L^2}^2
\le \de \|\na u_{tt}\|_{L^2}^2+C(\de)\sigma^{-3/2},
\end{aligned}\\
&R_3 
\le \de \|\na u_{tt}\|_{L^2}^2+C(\de)\sigma^{-1},\notag
\\
&\begin{aligned}\notag
R_4 
&\le \de \|\na u_{tt}\|_{L^2}^2+C(\de)\left( \|\n^{1/2}u_{tt}\|^2_{L^2}
+\|\na\te_t\|_{L^2}^2 +\|\n^{1/2}\te_{tt}\|_{L^2}^2+\sigma^{-2}\right),
\end{aligned}\\
& R_5 
\le \de \|\na u_{tt}\|_{L^2}^2+C(\de)(\|H_{tt}\|^2_{L^2}+\sigma^{-3/2}),\notag
\\
& R_6 
 \le \de (\|\na u_{tt}\|_{L^2}^2+\|\na H_{tt}\|_{L^2}^2)+C(\de)(\|H_{tt}\|^2_{L^2}+\sigma^{-3/2}+\sigma^{-2}).\notag
\end{align}
Substituting these estimates these estimates into \eqref{x6b2}, utilizing the fact that
\begin{align}\notag
\displaystyle  \|\nabla u_{tt}\|_{L^2}\leq C(\|\div u_{tt}\|_{L^2}+\|\omega_{tt}\|_{L^2}), \quad \|\nabla H_{tt}\|_{L^2}\leq C\|\curl H_{tt}\|_{L^2},
\end{align}
due to Lemma \ref{lem-vn}. Since $u_{tt}\cdot n=0, H_{tt}\cdot n=0\ $ on $\partial\Omega,$ and then choosing $\de$ small enough, we can get
\begin{align}
\begin{aligned}\label{x6b4}
&\frac{d}{dt}(\|\sqrt{\rho}u_{tt}\|^2_{L^2}+\|H_{tt}\|^2_{L^2})+\widetilde{C}_1(\|\na u_{tt}\|_{L^2}^2+\|\na H_{tt}\|_{L^2}^2) \\
&\le C (\|\sqrt{\rho}u_{tt}\|^2_{L^2}+\|H_{tt}\|^2_{L^2}+\sigma^{-2}+\|\na\te_t\|_{L^2}^2) +\widetilde{C}_2\|\n^{1/2}\te_{tt}\|_{L^2}^2.
\end{aligned}
\end{align}

Next, differentiating \eqref{CMHD-th} with respect to $t$ 
and multiplying the resulting equality by $\te_{tt}$ and integrating over $\Omega$ lead to
\be\la{ex5}\ba
& \left(\frac{\ka(\ga-1)}{2R}\|\na \te_t\|_{L^2}^2+N_0\right)_t+ \int\n\te_{tt}^2dx \\
&=\frac{1}{2}\int\n_{tt}\left( \te_t^2
+2\left(u\cdot\na \te+(\ga-1)\te\div u\right)\te_t\right)dx\\
&\quad + \int\n_t\left(u\cdot\na\te+(\ga-1)\te\div u \right)_t\te_{t}dx\\
& \quad-\int\n\left(u\cdot\na\te+(\ga-1)\te\div u\right)_t\te_{tt}dx\\
& \quad -\frac{\ga-1}{R}\int \left(\lambda (\div u)^2+2\mu |\mathbb{D}u|^2+\nu |\curl H|^2\right)_{tt}\te_t dx\triangleq\sum_{i=1}^4 N_i,\ea\ee
where
\bnn\ba N_0\triangleq & \frac{1}{2}\int \n_t\te_{t}^2dx
+\int\n_t\left(u\cdot\na\te+(\ga-1)\te\div u\right) \te_tdx\\
&- \frac{\ga-1}{R}\int\left(\lambda (\div u)^2+2\mu |\mathbb{D}u|^2+\nu |\curl H|^2\right)_t\te_t dx. \ea\enn
It follows from  $(\ref{CMHD})_1,$   (\ref{va1}),    \eqref{w3}, \eqref{s4}, and Lemmas \ref{lem-x1}--\ref{lem-x3} that
\be\la{ex6}\ba
|N_0|\le & C\int \n|u||\te_{t}||\na\te_{t}|dx+C\|\n_t\|_{L^3}\|\te_t\|_{L^6}\left( \|\na\te\|_{L^2} \|u\|_{L^\infty}+ \|\theta\|_{L^6} \|\na u\|_{L^3}\right)\\
&+ C(\|\na u\|_{L^3}\|\na u_t\|_{L^2}+\|\na H\|_{L^3}\|\na H_t\|_{L^2}) \|\te_t\|_{L^6} \\
\le &C  (\|\rho^{1/2}\theta_t \|_{L^2}+\|\na\te_t\|_{L^2})\left(\|\n^{1/2}\te_t\|_{L^2}+\|\na u_t\|_{L^2}+\|\na H_t\|_{L^2}+1\right)\\
\le &\frac{\ka(\ga-1)}{4R} \|\na\te_t\|_{L^2}^2+C\si^{-1},\ea\ee
and
\be\la{ex7}\ba
|N_1|& \le C\|\n_{tt}\|_{L^2}\left(\|\te_t\|_{L^4}^{2}
+\|\te_t\|_{L^6}\left(\|u\cdot\na \te\|_{L^3}+\| \te\div u\|_{L^3} \right)\right)\\
& \le  C(1+\|\na u_{t}\|_{L^2} )\|\na \te_t\|^2_{L^2}+C\si^{-3/2}.\ea\ee
Similarly, by \eqref{x1b1}, \eqref{x3b1},  (\ref{va1}), and \eqref{w3}, one obtains that
\be\la{ex9}\ba
|N_2|+|N_3|&\le C\|\left(u\cdot\na\te+(\ga-1)\te\div u \right)_t\|_{L^2}
\left(\|\n_t\|_{L^3} \|\te_t\|_{L^6}+\|\n \te_{tt}\|_{L^2}\right)\\
&\le \frac12\|\n^{1/2} \te_{tt}\|_{L^2}+C\|\na\te_t\|_{L^2}^2+C\si^{-1 } \|\na u_t\|^2_{L^2}+C\si^{-1 },\ea\ee
where we used the fact that
\be\la{eg12}\ba
 \|\left(u\!\cdot\!\na\!\te\!\!+\!(\ga\!\!-\!1)\te\div u \right)_t\|_{L^2}
\le  C\|\na\! u_t\|_{L^2}(\|\na^2\! \te\|_{L^2}\!\!+\!1)\!\!+\! C\|\na\! \te_t\|_{L^2}\!\!+\!C\|\rho^{1/2} \theta_t \|_{L^2},\ea\ee
due to Lemma \ref{lem-x1} and (\ref{w3}). Furthermore, one deduces from \eqref{x1b2}, (\ref{va1}), \eqref{w3}, and (\ref{x3b2}) that
\be\la{ex10}\ba
|N_4|&\le C\int \left(|\na u_t|^2+|\na u||\na u_{tt}|+|\na H_t|^2+|\na H||\na H_{tt}|\right)|\te_t|dx\\
&\le \de(\|\na u_{tt}\|^2_{L^2}+\|\na H_{tt}\|^2_{L^2})+C(\|\na^2 u_t\|^2_{L^2}+\|\na^2 H_t\|^2_{L^2})\\
&\quad+C(\de)(\|\na\te_t\|_{L^2}^2+\si^{-1 })+C\si^{-2}(\|\na u_t\|_{L^2}^2+\|\na H_t\|_{L^2}^2).\ea\ee
Substituting (\ref{ex7}), (\ref{ex9}), and (\ref{ex10}) into (\ref{ex5}) gives
\be\la{ex11}\ba
& \frac{d}{dt}\left(\frac{\ka(\ga-1)}{2R}\|\na \te_t\|_{L^2}^2+N_0\right)
+\frac{1}{2}\|\n^{1/2}\te_{tt}\|_{L^2}^2 \\
&\le  \de(\|\na u_{tt}\|^2_{L^2}+\|\na H_{tt}\|^2_{L^2})+C(\de)((1+\|\na u_{t}\|_{L^2}) \|\na \te_t\|^2_{L^2}+\si^{-3/2})\\
&\quad +C\left(\|\na^2 u_t\|^2_{L^2}+\|\na^2 H_t\|^2_{L^2} +\si^{-2}(\|\na u_t\|_{L^2}^2+\|\na H_t\|^2_{L^2})\right).\ea\ee
Next, adding (\ref{ex11}) multiplied by $2 (\widetilde{C}_2+1) $ to  (\ref{x6b4}) and choosing $\de$ suitably small yield that
\begin{equation}
\begin{aligned}\la{ex13}
& \left[ 2 (C_5+1)\left(\frac{\ka(\ga-1)}{2R}\|\na \te_t\|_{L^2}^2+N_0\right)
+\|\sqrt{\rho}u_{tt}\|^2_{L^2}+\|H_{tt}\|^2_{L^2}\right]_t\\
&\quad + \|\n^{1/2}\te_{tt}\|_{L^2}^2+\frac{\widetilde{C}_1}{2}(\|\na u_{tt}\|^2_{L^2}+\|\na H_{tt}\|^2_{L^2})\\
&\le C (1+\|\na u_{t}\|_{L^2}^2) (\si^{-2} +\|\na \te_t\|^2_{L^2})+\si^{-2}\|\na H_{t}\|_{L^2}^2\\
&\quad  +C(\|\n^{1/2}u_{tt}\|^2_{L^2} +\|H_{tt}\|^2_{L^2}) + C\|\na^2 u_t\|^2_{L^2} + C\|\na^2 H_t\|^2_{L^2}.
\end{aligned}
\end{equation}

Multiplying this by $\si^2$ and integrating the resulting inequality over $(0,T),$
we  obtain after using (\ref{ex6}),  (\ref{x3b2}),  (\ref{x3b1}), \eqref{ex6}, and Gr\"{o}nwall's inequality that
\be \la{eg10}
\begin{aligned}
&\sup_{ 0\le t\le T}\si^2\left(\|\sqrt{\rho}u_{tt}\|^2_{L^2}+\|H_{tt}\|^2_{L^2}+\|\na \te_t\|_{L^2}^2\right)\\
&+\int_{0}^T\si^2\left(\|\n^{1/2}\te_{tt}\|_{L^2}^2+\|\na u_{tt}\|^2_{L^2}+\|\na H_{tt}\|^2_{L^2}\right)dt\le C,
\end{aligned}\ee
which together with Lemmas \ref{lem-x1}-\ref{lem-x3},  (\ref{nt4}), (\ref{ex4}),  \eqref{va1}, and (\ref{a4.74}) gives
\be\la{sp20} \sup_{ 0\le t\le T}\si \left(\|\na u_t\|_{H^1}
+ \|\na H_t\|_{H^1}+ \|\na^2 H\|_{H^2}
+\|\na^2\te\|_{H^1}+\|\na^2u\|_{W^{1,q}} \right)\le C.\ee
We thus derive (\ref{x6b}) from Lemma \ref{lem-x1}-\ref{lem-x3}, (\ref{eg10}),   (\ref{sp20}), and complete the proof of Lemma \ref{lem-x6}.
\end{proof}

\begin{lemma}\la{lem-x7}There exists a positive constant $C$ such that the following estimate holds:
	\be \la{egg17}\sup_{ 0\le t\le T}\si^2  \left(\|\na^2\te\|_{H^2}+\| \te_t\|_{H^2}+\|\n^{1/2}\te_{tt}\|_{L^2}  \right)
	+\int_0^T\si^4\|\na \te_{tt}\|_{L^2}^2 dt\le C.\ee
\end{lemma}
\begin{proof}
First, differentiating $\eqref{CMHD-th}$ with respect to $t$ twice,
multiplying the resulting equality by $\te_{tt}$ and integrating that over $\Omega$ yield that
\be\la{eg3}\ba
&\frac{1}{2}\frac{d}{dt}\int\n|\te_{tt}|^2dx +\frac{\ka(\ga-1)}{R}\int|\na \te_{tt}|^2dx\\
&=-4\int \te_{tt}\n u\cdot\na\te_{tt}dx  -\int \n_{tt}\left(\te_t
+ u\cdot\na \te+(\ga-1)\te\div u\right)\te_{tt}dx\\
&\quad - 2\int\n_t\left(u\cdot\na \te+(\ga-1)\te\div u\right)_t\te_{tt}dx\\
& \quad - \int\n\left(u_{tt}\cdot\na \te+2u_t\cdot\na\te_t
+(\ga-1)(\te\div u)_{tt}\right)\te_{tt}dx\\
& \quad +\frac{\ga-1}{R}\int  \left(\lambda (\div u)^2+2\mu |\mathbb{D}u|^2+\nu |\curl H|^2\right)_{tt}\te_{tt}dx
\triangleq \sum_{i=1}^5Q_i.\ea\ee
It follows from Lemmas \ref{lem-x1}--\ref{lem-x6},  \eqref{kk}, (\ref{eg10}), (\ref{eg12}), and \eqref{va1}  that
\begin{align}
&\ba\notag
|Q_1|&\le C\|\n^{1/2}\te_{tt}\|_{L^2}\|\na \te_{tt}\|_{L^2}\|u\|_{L^\infty}\\
&\le \de \|\na \te_{tt}\|_{L^2}^2+C(\de) \|\n^{1/2}\te_{tt}\|^2_{L^2} ,\ea\\
&\ba\notag
|Q_2|&\le C \|\n_{tt}\|_{L^2}\|\te_{tt}\|_{L^6} \left( \|\te_t\|_{H^1}
+\|\na\te\|_{L^3}+\|\na u\|_{L^6}\|\te\|_{L^6}\right) \\
&\le \de\|\na \te_{tt}\|_{L^2}^2+C(\de)(\|\n^{1/2} \te_{tt}\|_{L^2}^2+\si^{-4}),\ea\\
&\ba\notag
|Q_3|&\le C \|\n_t\|_{L^3} \|\te_{tt}\|_{L^6}\|u\cdot\na \te+(\ga-1)\te\div u\|_{L^2} \\
&\le \de\|\na \te_{tt}\|_{L^2}^2+ C (\|\n^{1/2} \te_{tt}\|_{L^2}^2 +\si^{-4}),\ea\\
&\ba\notag
|Q_4|&\le C\|\te_{tt}\|_{L^6}
\left( \sigma^{-2}+ \|\n \te_{tt}\|_{L^2}\right)+C\|\te\|_{L^\infty}\|\n\te_{tt}\|_{L^2} \|\na u_{tt}\|_{L^2} \\
&\le \de\|\na \te_{tt}\|_{L^2}^2
+C(\de)\left(\|\n^{1/2} \te_{tt}\|_{L^2}^2+\si^{-1}\|\na u_{tt}\|_{L^2}^2+\si^{-4}  \right),\ea\\
&\ba\notag
|Q_5|&\le C\|\te_{tt}\|_{L^6}
\left( \sigma^{-2}+\|\na u\|_{L^3}\|\na u_{tt}+\|\na H\|_{L^3}\|\na H_{tt}\|_{L^2}\right)  \\
&\le \de\|\na \te_{tt}\|_{L^2}^2
+C(\de)\left(\|\n^{1/2} \te_{tt}\|_{L^2}^2+\|\na u_{tt}\|_{L^2}^2+\|\na H_{tt}\|_{L^2}^2+\si^{-4}  \right).\ea
\end{align}
Then, multiplying (\ref{eg3})  by $\si^4,$ substituting the estimates of $Q_i (i=1,\cdots,5)\ $ into the resulting equality and choosing $\de$ suitably small, one obtains
\begin{equation*}
\begin{aligned}
& \frac{d}{dt}\left(\si^4\|\n^{1/2} \te_{tt}\|_{L^2}^2\right) +\frac{\ka(\ga-1)}{R}\si^4\|\na \te_{tt}\|_{L^2}^2\\
& \le  C\si^2\left(\|\n^{1/2} \te_{tt}\|_{L^2}^2
+\|\na u_{tt}\|_{L^2}^2+\|\na H_{tt}\|_{L^2}^2  \right)+C,
\end{aligned}
\end{equation*}
which together with (\ref{eg10})  gives
\be\la{eg13} \sup_{ 0\le t\le T}\si^4\|\n^{1/2} \te_{tt}\|_{L^2}^2
+\int_{0}^T\si^4\|\na \te_{tt}\|_{L^2}^2dt\le C.\ee
Finally, one obtains after using Lemmas \ref{lem-x1}--\ref{lem-x6}, (\ref{va1}), (\ref{hs}), and (\ref{eg13}) that
\be\notag
\ba
&\sup_{0\le t\le T}\si^2(\|\na^2\te_t\|_{L^2}+\|\na^2\te\|_{H^2})\le  C,\ea\ee
which together with (\ref{eg13}) shows (\ref{egg17}).
The proof of Lemma \ref{lem-x7} is completed.
\end{proof}

\section{Proof of  Theorem  \ref{th1}-\ref{th3}}\label{se5}

With all the a priori estimates in Section \ref{se3} and Section \ref{se4} at hand, we are going to  prove the main results of the paper in this section.
We first state the global existence of strong solution $(\rho,u,\te, H)$ whose proof is similar to that of \cite[Proposition 5.1]{llw2022}.

\begin{proposition} \la{pro2}
For  given numbers $M>0$ (not necessarily small),  $\on> 2,$ and $\bt>1,$   assume that  $(\rho_0,u_0,\te_0,H_0)$ satisfies (\ref{2.1}),  (\ref{3.1}), and   (\ref{z01}). Then    there exists a unique classical solution  $(\rho,u,\te, H) $ of problem (\ref{CMHD})--(\ref{boundary}) in $\Omega\times (0,\infty)$ satisfying (\ref{mn5})--(\ref{mn2}) with $T_0$ replaced by any $T\in (0,\infty).$
  Moreover,  (\ref{key2}), (\ref{basic0}), (\ref{ae3.7}), and (\ref{vu15})  hold for any $T\in (0,\infty)$ and (\ref{key3}) holds for any $t\geq 1$.
\end{proposition}
\begin{proof}
By Lemma \ref{lem-local}, there exists a $T_0>0$ which may depend on
$\inf\limits_{x\in \Omega}\n_0(x), $  such that the  problem
 (\ref{CMHD})--(\ref{boundary})  with   initial data $(\n_0 ,u_0,\te_0, H_0)$
 has a unique classical solution $(\n,u,\te,H)$ on $\O\times(0,T_0]$  satisfyinng (\ref{mn6})--(\ref{mn2}). One may use the a priori estimates, Proposition \ref{pr1} and Lemmas \ref{lem-x1}-\ref{lem-x7} to extend the classical solution $(\rho,u,\theta,H)$ globally in time.

First, it follows from (\ref{As1})--(\ref{3.1}) and (\ref{z01}) that
\bnn A_1(0)\le 2M^2,\quad  A_2(0)\le  C_0^{1/4},\quad A_3(0)+A_4(0)=0, \quad  \n_0<
 \hat{\rho},\quad \te_0\le \bt,\enn  which implies  there exists a
$T_1\in(0,T_0]$ such that (\ref{key1} ) holds for $T=T_1.$
 We set \bnn \notag T^* =\sup\left\{T\,\left|\, \sup_{t\in [0,T]}\|(\n,u,\te,H)\|_{H^3}<\infty\right\},\right.\enn  and \be \la{s1}T_*=\sup\{T\le T^* \,|\,{\rm (\ref{key1} ) \
holds}\}.\ee Then $ T^*\ge T_* \geq T_1>0.$

Next, we claim that
 \be \la{s2}  T_*=\infty.\ee  Otherwise,    $T_*<\infty.$
Proposition \ref{pr1} shows   (\ref{key2}) holds for all $0<T<T_*,$ which together with \eqref{z01} yields
  Lemmas \ref{lem-x1}--\ref{lem-x7} still hold for all  $0< T< T_*.$
Note here that all  constants $C$  in  Lemmas \ref{lem-x1}--\ref{lem-x7} 
depend  on $T_*  $ and $\inf\limits_{x\in \Omega}\n_0(x)$, and are in fact independent  of  $T$.
Then,  we claim that  there
exists a positive constant $\tilde{C}$ which may  depend  on $T_* $
and $\inf\limits_{x\in \Omega}\n_0(x)$   such that, for all  $0< T<
 T_*,$  \be\la{y12}\ba \sup_{0\le t\le T}
\| \n\|_{H^3}   \le \tilde{C},\ea \ee which together with Lemmas \ref{lem-x1}-\ref{lem-x6},  \eqref{mn2}, and (\ref{3.1}) gives
 \bnn
 \|(\n(x,T_*),u(x,T_*),\te(x,T_*),H(x,T_*))\|_{H^3}
 \le \tilde{C},\quad\inf_{x\in \Omega}\n(x,T_*)>0,\quad\inf_{x\in \Omega}\te(x,T_*)>0.\enn
  Thus, Lemma \ref{lem-local} implies that there exists some $T^{**}>T_*,$  such that
(\ref{key1}) holds for $T=T^{**},$   which contradicts (\ref{s1}).
Hence, (\ref{s2}) holds. This along with Lemmas \ref{lem-local}, \ref{lem-basic2}, \ref{lem-th},  and Proposition \ref{pr1},  thus
finishes the proof of   Proposition \ref{pro2}.

Finally, it remains to prove (\ref{y12}).
By $(\ref{CMHD})_3$ and (\ref{mn6}), we can define
 \bnn
 \theta_t(\cdot,0)\triangleq - u_0 \cdot\na \te_0\! +\! \frac{\ga\!-\!1}{R} \rho_0^{-1}
\left(\ka\Delta\te_0\!-\!R\rho_0 \theta_0 \div u_0\!+\!\lambda (\div u_0)^2\!+\!2\mu |\mathbb{D}u_0|^2\!+\!\nu |\curl H_0|^2\right),
 \enn
which together with (\ref{2.1}) gives
 \be \la{ssp91}\|\theta_t(\cdot,0)\|_{L^2}\le \tilde{C}.\ee
 Thus, one deduces from  (\ref{3.99}), \eqref{2.1}, \eqref{ssp91}, and Lemma \ref{lem-x1} that
\be  \ba \la{a51}
\sup\limits_{0\le t\le T} \int \n|\dot\te|^2dx+\int_0^T \|\na\dot\te\|_{L^2}^2dt \le \tilde{C},
\ea\ee
which together with  (\ref{td2-th} ) and Lemma \ref{lem-x1}  yields
\be\la{sp211}
\sup\limits_{0\le t\le T}\|\na^2 \theta\|_{L^2} \le \tilde{C}.\ee
Similarly, by  $(\ref{CMHD})_2$ and  (\ref{mn6}), we can define
\bnn
u_t(\cdot,0) \triangleq -\!u_0\!\cdot\!\na\! u_0\!+\!\n_0^{-1}\left( \mu \Delta u_0 \!+\! (\mu\!\!+\!\lambda) \na\! \div u_0\! -\! R\na (\n_0\te_0)\!+\!H_0\!\cdot\!\na\! H_0\!-\!\frac12\na|H_0|^2\right),
\enn
which along with \eqref{hh00} and  (\ref{2.1}) yields
\be \la{ssp9}\|\na u_t(\cdot,0)\|_{L^2}+\|\nabla H_t(\cdot,0)\|_{L^2}\le \tilde{C}.\ee
Thus, it follows from Lemmas \ref{lem-x1},  (\ref{x4b8}),  (\ref{s4}),  (\ref{a51})--(\ref{ssp9}), and Gr\"{o}nwall's inequality that
\be \la{ssp1}
\sup_{0\le t\le T}(\|\na u_t\|_{L^2}+\|\na H_t\|_{L^2})+\int_0^T\int (\n |u_{tt}|^2+|H_{tt}|^2)dxdt\le \tilde{C},
\ee
which as well as (\ref{va2}), \eqref{sp211}, \eqref{3tdh-1} and (\ref{x1b2}) yields
\be\la{sp221} \sup\limits_{0\le t\le T}(\|u\|_{H^3}+\|H\|_{H^3}) \le \tilde{C}.\ee
Combining this with Lemma \ref{lem-x1}, \eqref{ex4}, \eqref{nt4}, (\ref{a51}), (\ref{sp211}), (\ref{ssp1}),   and  (\ref{sp221}) gives
\be\la{ssp24} \ia\left(\|\na^3\te\|_{L^2}^2+ \|\nabla u_t\|_{H^1}^2+ \|\nabla H_t\|_{H^1}^2\right)dt\le \tilde{C} . \ee
Then, applying \eqref{hs}, \eqref{sp211}, \eqref{ssp1}, \eqref{sp221}, and Lemma \ref{lem-x1}, one has
\be \notag\ba \|\na^2 u\|_{H^2}
&\le\tilde{C}\left(\| \n \dot u \|_{H^2}+\|\na  P\|_{H^2}+\|H\na  H\|_{H^2} + \|\na u\|_{L^2}\right)\\
&\le \tilde{C}\left(\| \n \|_{H^2}\|  u_t \|_{H^2}+\| \n \|_{H^2}\|  u \|_{H^2}\|  \na u \|_{H^2}+\| H \|_{H^2}\|  \na H \|_{H^2}\right)\\
&\quad+\tilde C\left(\|\na\n \|_{H^2}\|\te\|_{H^2}+\|\n\|_{H^2}\|\na  \te\|_{H^2}+1\right)\\
& \le \tilde{C} (1+ \|\na^2  u_t\|_{L^2}+\|\na^3 \n \|_{L^2}+\|\na^3 \te \|_{L^2}),\ea\ee
which along with (\ref{sp221}) and Lemma \ref{lem-x1} leads to
\begin{equation}
\begin{aligned}\la{sp134}
& \left(\|\na^3 \n\|_{L^2} \right)_t \\
&\le \tilde{C}\left(\| |\na^3u| |\na \n| \|_{L^2}+ \||\na^2u||\na^2
      \n|\|_{L^2}+ \||\na u||\na^3 \n|\|_{L^2} +\| \na^4u \|_{L^2} \right)\\
&\le \tilde{C}(1+ \| \na^3\n \|_{L^2}+ \| \na^2 u_t\|^2_{L^2}+ \|\na^3\te\|^2_{L^2}).
\end{aligned}
\end{equation}
Combining this with (\ref{ssp24})  and
Gr\"{o}nwall's inequality  yields   \bnn\la{sp26} \sup\limits_{0\le t\le
T}\|\nabla^3  \n\|_{L^2} \le \tilde{C},\enn which together with
(\ref{x1b2}) gives (\ref{y12}).
The proof of Proposition \ref{pro2} is completed.
\end{proof}

With  Proposition \ref{pro2} at hand, we are now in a position to prove  Theorem \ref{th1}.

\noindent{\it Proof of Theorem \ref{th1}.}
Assume that  $C_0$  satisfies (\ref{co14}) with
\be\la{xia}\ve\triangleq \ve_0/2,\ee
where  $\ve_0$  is given in Proposition \ref{pr1}. Suppose the initial data $(\n_0,u_0,\te_0,H_0)$ satisfies (\ref{dt1})--(\ref{dt3}).
First, we construct the approximate initial data $(\n_0^{m,\eta},u_0^{m,\eta}, \te_0^{m,\eta}, H_0^{m,\eta})$ as follows:
\begin{align*}
\n_0^{m,\eta} = \n_0^{m}+ \eta,\ \  u_0^{m,\eta}=\frac{u_0^m }{1+\eta},\ \  \te_0^{m,\eta}= \frac{\te_0^{m} + \eta}{1+2\eta},\ \  H_0^{m,\eta}=\frac{H_0^m }{1+\eta},
\end{align*}
where
\be\la{5d0}
m \in \mathbb{Z}^+,\ \  \eta \in \left(0, \eta_0 \right),\ \  \eta_0\triangleq \min\xl\{1,\frac{1}{2}(\on-\sup\limits_{x\in \O}\n_0(x)) \xr\},
\ee
and $(\n_0^{m},u_0^{m}, \te_0^{m}, H_0^{m})$ satisfies the boundary conditions \eqref{navier-b}-\eqref{boundary} and
\be\notag   \begin{cases}
0 \le \n_0^{m} \in C^{\infty},\ \ \displaystyle  \lim_{m \to \infty} \|\n_0^{m} -\rho_0\|_{W^{2,q}}=0,\\
\Delta u_0^m=\Delta \tilde{u}_0^m,\qquad\text{in}\,\, \O,\\
\tilde{u}_0^m \in C^{\infty},\ \ \displaystyle\lim_{m \to \infty}\| \tilde{u}_0^m -{u}_0\|_{H^2}=0,\\
\Delta \te_0^m=\Delta \tilde{\te}_0^m- \overline{\Delta \tilde{\te}_0^m},\qquad\text{in}\,\,\O,\\
0\le \tilde{\te}_0^m \in C^{\infty},\ \ \displaystyle\lim\limits_{m \to \infty}\| \tilde{\te}_0^m -{\te}_0\|_{H^2}=0,\\
\Delta H_0^m=\Delta \tilde{H}_0^m,\qquad\text{in}\,\, \O,\\
\tilde{H}_0^m \in C^{\infty},\ \ \displaystyle\lim_{m \to \infty}\| \tilde{H}_0^m -{H}_0\|_{H^2}=0,
\end{cases}\ee
Then for any $\eta\in (0, \eta_0)$, there exists $m_1(\eta)\ge 1$ such that for $m \ge m_1(\eta)$, the approximate initial data
$(\n_0^{m,\eta},u_0^{m,\eta}, \te_0^{m,\eta}, H_0^{m,\eta})$ satisfies the boundary conditions \eqref{navier-b}-\eqref{boundary} and
\be \la{de3}\begin{cases}(\n_0^{m,\eta},u_0^{m,\eta}, \te_0^{m,\eta}, H_0^{m,\eta})\in C^\infty ,\\
  \dis \eta\le  \n_0^{m,\eta}  <\hat\n,~~\, \frac{\eta}{4}\le \te_0^{m,\eta} \le \hat \te,~~\,\|\na u_0^{m,\eta}\|_{L^2} \le M,~~\,\|\na H_0^{m,\eta}\|_{L^2} \le M, \\
\lim\limits_{\eta\rightarrow 0}\! \lim\limits_{m\rightarrow \infty}\!
\left(\| \n_0^{m,\eta} \!\!-\!\! \n_0 \|_ {W^{2,q}}\!\!+\!\!\| u_0^{m,\eta}\!\!-\!u_0\|_{H^2}\!\!+\!\!\| \te_0^{m,\eta}\!\!-\!\te_0  \|_{H^2}\!\!+\!\!\| H_0^{m,\eta}\!\!-\!\! H_0  \|_{H^2}\right)\!=\!0.
\end{cases}
\ee
Moreover,  the initial norm $C_0^{m,\eta}$
for $(\n_0^{m,\eta},u_0^{m,\eta}, \te_0^{m,\eta}, H_0^{m,\eta}),$ which is defined by  the right-hand side of (\ref{c0})
with $(\n_0,u_0,\te_0,H_0)$   replaced by
$(\n_0^{m,\eta},u_0^{m,\eta}, \te_0^{m,\eta}, H_0^{m,\eta}),$
satisfies \bnn \lim\limits_{\eta\rightarrow 0} \lim\limits_{m\rightarrow \infty} C_0^{m,\eta}=C_0.\enn
Therefore, there exists  an  $\eta_1\in(0, \eta_0) $
such that, for any $\eta\in(0,\eta_1),$ we can find some $m_2(\eta)\geq m_1(\eta)$  such that   \be \la{de1} C_0^{m,\eta}\le C_0+\ve_0/2\le  \ve_0 , \ee
provided that\be  \la{de7}0<\eta<\eta_1 ,\,\, m\geq m_2(\eta).\ee

  We assume that $m,\eta$ satisfy (\ref{de7}).
  Proposition \ref{pro2} together with (\ref{de1}) and (\ref{de3}) thus yields that
   there exists a smooth solution  $(\n^{m,\eta},u^{m,\eta}, \te^{m,\eta}, H^{m,\eta}) $
   of problem (\ref{CMHD})--(\ref{boundary}) with  initial data $(\n_0^{m,\eta},u_0^{m,\eta}, \te_0^{m,\eta}, H_0^{m,\eta})$
   on $\Omega\times (0,T] $ for all $T>0. $
    Moreover, one has (\ref{esti-rho}), \eqref{key2}, \eqref{key3}, (\ref{basic0}), \eqref{ae3.7}, and (\ref{vu15})   with $(\n,u,\te,H)$  being replaced by $(\n^{m,\eta},u^{m,\eta}, \te^{m,\eta}, H^{m,\eta}).$

 Next, for the initial data $(\n_0^{m,\eta},u_0^{m,\eta}, \te_0^{m,\eta}, H_0^{m,\eta})$, the function $\tilde g$ in (\ref{co12})  is
 \be \la{co5}\ba \tilde g & \triangleq(\n_0^{m,\eta})^{-\frac12}\left(-\!\mu \Delta u_0^{m,\eta}\!-\!(\mu\!+\!\lambda)\na\div
 u_0^{m,\eta}\!+\!R\na (\n_0^{m,\eta}\te^{m,\eta}_0)\!-\!\curl H_0^{m,\eta}\!\times\! H_0^{m,\eta}\right)\\
& = (\n_0^{m,\eta})^{-\frac12}\sqrt{\n_0}g+\mu(\n_0^{m,\eta})^{-\frac12}\Delta(u_0-u_0^{m,\eta})\\
&\quad+(\mu+\lambda) (\n_0^{m,\eta})^{-\frac12} \na \div(u_0-u_0^{m,\eta})+ R(\n_0^{m,\eta})^{-\frac12} \na(\n_0^{m,\eta}\te_0^{m,\eta}-\n_0\te_0),\\
&\quad+(\n_0^{m,\eta})^{-\frac12}(\curl H_0\!\times\! H_0-\curl H_0^{m,\eta}\!\times\! H_0^{m,\eta}),\ea\ee
where in the second equality we have used (\ref{dt3}).
Since $g \in L^2,$ one deduces from (\ref{co5}),  (\ref{de3}), and  (\ref{dt1})  that for any $\eta\in(0,\eta_1),$ there exist some $m_3(\eta)\geq m_2(\eta)$ and a positive constant $C$ independent of $m$ and $\eta$ such that
 \be\la{de4}
  \|\tilde g\|_{L^2}\le \|g\|_{L^2}+C\eta^{-\frac12}\de(m) + C\eta^{\frac12},
  \ee
with   $0\le\de(m) \rightarrow 0$ as
$m \rightarrow \infty.$ Hence,  for any  $\eta\in(0,\eta_1),$ there exists some $m_4(\eta)\geq m_3(\eta)$ such that for any $ m\geq m_4(\eta)$,
 \be \la{de9}\de(m) <\eta.\ee  We thus obtain from (\ref{de4}) and (\ref{de9}) that
there exists some positive constant $C$ independent of $m$ and $\eta$ such that  \be\la{de14}  \|\tilde g \|_{L^2}\le \|g \|_{L^2}+C,\ee provided that\be \la{de10} 0<\eta<\eta_1,\,\,  m\geq m_4(\eta).\ee

 Now, we   assume that $m,$  $\eta$ satisfy (\ref{de10}).
 It thus follows from (\ref{de3})--(\ref{de1}),  (\ref{de14}), Proposition \ref{pr1}, 
 and Lemmas \ref{lem-th}, \ref{lem-x1}--\ref{lem-x7} that for any $T>0,$
 there exists some positive constant $C$ independent of $m$ and $\eta$ such that
 (\ref{esti-rho}), (\ref{key2}),   (\ref{basic0}),  \eqref{ae3.7}, (\ref{vu15}), \eqref{x1b1}, \eqref{x1b2},  (\ref{x3b1}),  (\ref{vva5}), (\ref{x6b}),
  and  (\ref{egg17})  hold for  $(\n^{m,\eta},u^{m,\eta}, \te^{m,\eta}, H^{m,\eta}) .$
   Then passing  to the limit first $m\rightarrow \infty,$ then $\eta\rightarrow 0,$
   together with standard arguments yields that there exists a solution $(\n,u,\te,H)$ of the problem (\ref{CMHD})--(\ref{boundary})
   on $\Omega\times (0,T]$ for all $T>0$, such that the solution  $(\n,u,\te,H)$
   satisfies  (\ref{esti-rho}), (\ref{basic0}),    \eqref{ae3.7},  (\ref{vu15}),  \eqref{x1b1}, \eqref{x1b2}, (\ref{x3b1}),  (\ref{vva5}), (\ref{x6b}),  (\ref{egg17}),
   and  the estimates of $A_i(T)\,(i=1,\cdots,4)$ in
   (\ref{key2}). Hence, $(\n,u,\te,H)$ satisfying  (\ref{esti-rho}) and \eqref{esti-uh} can be proved in the same way as that in \cite{hl2018}. Moreover, one deduces from Proposition \ref{pr1} that the desired exponential decay property \eqref{esti-t}.
Furthermore,  the proof of the uniqueness of $(\n,u,\te,H)$ is similar to that  of \cite{ck2006} and we omit the details.
The proof of Theorem \ref{th1} is finished.   \endproof

\noindent{\it{Proof of Theorem \ref{th2}.} }
As is shown by \cite{cl2019}, we sketch the proof for completeness. First, we show that, for $T>0$, the Lagrangian coordinates of the system are given by
  \be \la{c61}  \begin{cases}\frac{\partial}{\partial \tau}X(\tau; t,x) =u(X(\tau; t,x),\tau),\,\,\,\, 0\leq \tau\leq T\\
 X(t;t,x)=x, \,\,\,\, 0\leq t\leq T,\,x\in\bar{\Omega}.\end{cases}\ee
 By \eqref{esti-uh}, the transformation \eqref{c61} is well-defined. In addition, by $\eqref{CMHD}_1$, we have
 \be\la{c62}\ba
\rho(x,t)=\rho_0(X(0; t, x)) \exp \{-\int_0^t\div u(X(\tau;t, x),\tau)d\tau\}.
\ea \ee
 If there exists some point $x_0\in \Omega$ such that $\n_0(x_0)=0,$ then there is a point $x_0(t)\in \bar{\Omega}$ such that $X(0; t, x_0(t))=x_0$. Hence, by \eqref{c62}, $\rho(x_0(t),t)\equiv 0$ for any $t\geq 0.$
By \eqref{g2}, we get for $ \tilde{r}\in  (3,\infty)$ and $\theta=\frac{2(\tilde{r}-3)}{5\tilde{r}-6}$,
\begin{align*}
\displaystyle  1 \leq\|\rho-1\|_{C\left(\ol{\O }\right)} \le C
\|\rho-1\|_{L^2}^{\theta}\|\na \rho\|_{L^{\tilde{r}}}^{1-\theta},
\end{align*}
Combining this with \eqref{esti-t} gives \eqref{esti-2} and completes the proof of Theorem \ref{th2}. \endproof

\noindent{\it Proof of  Theorem   \ref{th3}. } We adopt the idea in \cite{llw2022} to  prove  Theorem   \ref{th3} in two steps.

 {\it Step 1. Construction  of approximate  solutions.} Assume $(\n_0,u_0,\te_0,H_0)$ satisfying (\ref{dt1}) and \eqref{dt7} is the initial data in Theorem \ref{th3} and  $C_0$ satisfies (\ref{co14})   with  $\ve $  as in  (\ref{xia}).   For $j_{m^{-1}}(x)$ being the standard mollifying kernel of width $m^{-1}$, we construct
\begin{align*}
\hat\n_0^{m,\eta} =\! (\n_0 1_\O)\ast j_{m^{-1}}1_\O\!+\! \eta,\ \  \hat u_0^{m,\eta}=\!\frac{u_0^m }{1\!+\!\eta},\ \  \hat\te_0^{m,\eta}=\! \frac{(\n_0\te_0 1_{\O_m})\ast j_{m^{-1}} \!+\! \eta}{(\n_01_{\O_m})\ast j_{m^{-1}}\!+ \!\eta},\ \  \hat H_0^{m,\eta}=\!\frac{H_0^m }{1\!+\!\eta},
\end{align*}
where  $\O_m=\{x\in\O| dist(x,\p\O)>2/m\}$ and $u_0^{m},H_0^{m}$ satisfies the boundary conditions \eqref{navier-b}, \eqref{boundary} and
\be\notag (u_0^{m},H_0^{m})\in C^\infty\cap H^1\ \ \text{and}\ \ \lim_{m\rightarrow \infty}\|(u_0^{m}-u_0,H_0^{m}-H_0)\|_{H^1}=0.\ee
 Then for any $\eta\in (0, \eta_0)$ with $\eta_0$  as in (\ref{5d0}), there exists $m(\eta)>1$ such that for $m \ge m(\eta)$, the approximate initial data
$(\hat\n_0^{m,\eta},\hat u_0^{m,\eta},\hat\te_0^{m,\eta},\hat{H}_0^{m,\eta})$ satisfies the boundary conditions \eqref{navier-b}-\eqref{boundary} and
\be \la{dee3}\begin{cases}(\hat\n_0^{m,\eta},\hat u_0^{m,\eta},\hat\te_0^{m,\eta},\hat{H}_0^{m,\eta})\in C^\infty ,\\
  \dis \eta\le \hat \n_0^{m,\eta}  <\hat\n,~~\, \frac{\eta}{\hat\n+\eta}\le \hat\te_0^{m,\eta} \le \hat \te,~~\,\|\na \hat u_0^{m,\eta}\|_{L^2} \le M,~\,\|\na \hat H_0^{m,\eta}\|_{L^2} \le M,\end{cases}
\ee
and for any $p\geq1$,
 \be \la{dee03}
\lim\limits_{\eta\rightarrow 0} \lim\limits_{m\rightarrow \infty}
\left(\| \hat\n_0^{m,\eta}\!\! -\!\! \n_0 \|_ {L^p}\!\!+\!\!\|\hat u_0^{m,\eta}\!\!-\!\!u_0\|_{H^1}\!\!+\!\!\|\hat\n_0^{m,\eta}\hat\te_0^{m,\eta}\!\!-\!\!\n_0\te_0  \|_{L^2}\!\!+\!\!\|\hat H_0^{m,\eta}\!\!-\!\!u_0\|_{H^1}\right)\!\!=\!0
\ee
owing to (\ref{dt1})  and (\ref{co14}).

Now, we claim that  the initial norm $\hat C_0^{m,\eta}$
for $(\hat\n_0^{m,\eta},\hat u_0^{m,\eta},\hat\te_0^{m,\eta},\hat{H}_0^{m,\eta}),$  i.e., the right hand side of
(\ref{c0}) with $(\n_0,u_0,\te_0,H_0)$  replaced by
$(\hat\n_0^{m,\eta},\hat u_0^{m,\eta},\hat\te_0^{m,\eta},\hat{H}_0^{m,\eta}),$   satisfies
\be \la{uv9}  \lim\limits_{\eta\rightarrow 0}
\lim\limits_{m\rightarrow \infty}\hat C_0^{m,\eta}\le C_0,\ee
which leads to that there exists an $\hat\eta\in(0,\eta_0)$ such that, for any $\eta\in
(0,\hat\eta ),$ there exists some $\hat m (\eta)\geq m(\eta)$ such that
\be \la{uv8}\hat C_0^{m,\eta}\le C_0+\ve_0/2\le \ve_0 , \ee
provided \be\la{uv01}
0<\eta<\hat\eta  , \quad m\geq\hat m (\eta).\ee
Then if we assume (\ref{uv01}) holds, it directly follows from Proposition \ref{pro2}, (\ref{dee3}) and (\ref{uv8}) that there exists a classical solution  $(\hat\n^{m,\eta},\hat u^{m,\eta},\hat\te^{m,\eta},\hat H^{m,\eta})  $  of problem (\ref{CMHD})--(\ref{boundary}) with  initial data $(\hat\n_0^{m,\eta},\hat u_0^{m,\eta},\hat\te_0^{m,\eta},\hat{H}_0^{m,\eta})$   on $\O\times(0,T]$ for all $T>0$.  Furthermore, $(\hat\n^{m,\eta},\hat u^{m,\eta},\hat\te^{m,\eta},\hat H^{m,\eta}) $  satisfies \eqref{esti-rho}, (\ref{key2}), (\ref{basic0}), \eqref{p-b}, (\ref{ae3.7}),  (\ref{vu15}), and \eqref{key3} respectively  for any $T>0$ and $t\geq1$ with $(\n,u,\te,H)$   replaced by $(\hat\n^{m,\eta},\hat u^{m,\eta},\hat\te^{m,\eta},\hat H^{m,\eta})$.

It remains to prove (\ref{uv9}). Indeed, we just need to infer
\be\la{uv10}\lim_{\eta\rightarrow 0}\lim_{m\rightarrow \infty}\int\hat\n_0^{m,\eta}\Phi(\hat\te_0^{m,\eta})dx\le \int\n_0\Phi(\te_0)dx ,\ee since   the other terms in  (\ref{uv9}) can be proved in a  similar and even simpler way.
Note that
\begin{align}\notag
\begin{aligned}
\hat\n_0^{m,\eta}\Phi(\hat\te_0^{m,\eta})
&=\hat\n_0^{m,\eta}(\hat\te_0^{m,\eta}-1)^2 \int_0^1\frac{\al}{\al(\hat\te_0^{m,\eta}-1)+1}d\al\\
 &= \frac{\hat\n_0^{m,\eta}(j_{m^{-1}}*(\n_0(\te_0-1)1_{\O_m}))^2}{j_{m^{-1}}* (\n_01_{\O_m} )+\eta}\\
 &\quad\cdot\int_0^1\frac{\al}{\al j_{m^{-1}}*(\n_0(\te_0-1)1_{\O_m})+j_{m^{-1}}* (\n_01_{\O_m} )+\eta}d\al \\
 &\in \left[0, \, \hat\n\eta^{-2}(j_{m^{-1}}*(\n_0(\te_0-1)1_{\O_m}))^2 \right],
\end{aligned}
\end{align}
which combined with  Lebesgue's dominated convergence theorem yields that
\begin{align}\notag
\begin{aligned}
&\lim_{m\rightarrow \infty}\int\hat\n_0^{m,\eta}\Phi(\hat\te_0^{m,\eta})dx\\
&=\int(\n_0+\eta)\left(\frac{\n_0\te_0+\eta}{\n_0+\eta}
 - \ln\frac{\n_0\te_0+\eta}{\n_0+\eta}-1 \right)dx\\
&=\int\left(\n_0\te_0 -\n_0+(\n_0+\eta)\ln(\n_0+\eta)\right)dx\\
 &\quad-\int\left(\n_0\ln(\n_0\te_0+\eta)+\eta\ln(\n_0\te_0+\eta)\right)dx\\
 &\le \int\left(\n_0\te_0 -\n_0+(\n_0+\eta)\ln(\n_0+\eta)\right)dx-\int\left(\n_0\ln(\n_0\te_0)+\eta\ln \eta\right)dx\\
 & \rightarrow \int \n_0\left(\te_0-  \ln\te_0 -1 \right)dx=\int \n_0\Phi(\te_0)dx,\quad \mbox{ as }\eta\rightarrow 0.
\end{aligned}
\end{align}
It thus gives \eqref{uv10}.

{\it Step 2. Compactness results.}
 With the approximate solutions $(\hat\n^{m,\eta},\hat u^{m,\eta},\hat\te^{m,\eta},\hat H^{m,\eta}) $ obtained in the previous step at hand,  we can derive the global existence of weak solutions by passing to the limit  first   $m\rightarrow \infty,$ then $\eta\rightarrow 0 .$   Since the two steps are similar, we will only   sketch the  arguments for $m\rightarrow \infty.$  For any fixed  $\eta\in (0,\hat\eta)$, we simply denote  $(\hat\n^{m,\eta},\hat u^{m,\eta},\hat\te^{m,\eta},\hat H^{m,\eta}) $
   by $(\n^m,u^m,\te^m,H^m).$ Then  the combination of Aubin-Lions Lemma with (\ref{key2}), (\ref{basic0}),  \eqref{p-b},  (\ref{vu15}),  and Lemma \ref{lem-f-td} yields that there exists some appropriate subsequence $ m_j \rightarrow \infty$ of $m\rightarrow \infty$ such that, for any $0<\tau<T<\infty $, $p\in[1,\infty)$, and $\tilde p\in[1,6)$,
\be \la{vu4}
u^{m_j}\rightharpoonup u,\,H^{m_j}\rightharpoonup H \,\,\mbox{ weakly star in }\,\, L^\infty(0,T; H^1),
\ee
\be\la{lk}
\te^{m_j}\rightharpoonup \te, \,H^{m_j}\rightharpoonup H \,\,\mbox{ weakly in }\,\, L^2(0,T;H^1),
\ee
\be \la{vu1}
\n^{m_j}\rightarrow \n  \quad \mbox{ in }\,\, C([0,T];L^p\mbox{-weak})\cap C([0,T];H^{-1}),\ee
\be \la{vu5}
\n^{m_j}u^{m_j}\rightarrow \n u,\, \n^{m_j}\te^{m_j}
\rightarrow \n \te   \, \mbox{ in }\,  C([0,T];L^2 \mbox{-weak})\cap C([0,T];H^{-1}),
\ee
\be \la{vu26} \n^{m_j} |u^{m_j}|^2\rightarrow \n |u|^2 \,\, \mbox{ in }\,\, C([0,T];L^3 \mbox{-weak})\cap C([0,T];H^{-1}),\ee
\be\la{ghh}
F^{m_j}\rightarrow F,\,\omega^{m_j}\rightarrow\omega\,\,\mbox{ in }\,\,C([\tau,T];H^1\mbox{-weak})\cap C([\tau,T];L^{\tilde p}),
\ee
\be\la{lll}u^{m_j}\rightarrow u\,\mbox{ in }\,\,C([\tau,T];W^{1,6}\mbox{-weak})\cap C(\O\times[\tau,T]),\ee
and
\be \la{vu18}
\te^{m_j}\rightarrow \te, \,H^{m_j}\rightarrow H\,\,\mbox{in}\,\,C([\tau,T];H^2\mbox{-weak})\cap C([\tau,T];W^{1,\tilde p}).
\ee
Now we consider the approximate solutions $(\n^{m_j},u^{m_j},\te^{m_j},H^{m_j})$ in the weak forms, i.e. \eqref{def1}--\eqref{def3}, then  take appropriate limits. Standard arguments as well as (\ref{dee03}) and (\ref{vu4})--(\ref{vu18}) thus conclude that the limit $(\n,u,\te,H)$ is a weak solution  of   (\ref{CMHD}),   (\ref{navier-b})-(\ref{boundary})  in the sense of Definition \ref{def} and satisfies (\ref{hq1})--(\ref{hq4}) and the exponential decay property \eqref{esti-t}. 
Moreover, we obtain the estimates (\ref{hq5})--(\ref{hq8}) with the aid of (\ref{key2}), (\ref{basic0}), (\ref{vu15}),    and (\ref{vu4})--(\ref{vu18}). Finally, (\ref{vu019}) shall be obtained by adopting the similar way as in \cite{hl2018}.
The proof of Theorem \ref{th3} is finished. \endproof




\appendix


\section*{Acknowledgements} The research  is partially supported by the National Natural Science Foundation of China (No. 11901025).

\end{document}